\documentclass{article}

\usepackage[margin=3.2cm]{geometry}

\usepackage{setspace}

\usepackage{graphics}
\usepackage[affil-it]{authblk}

\usepackage{tikz}
\usetikzlibrary{arrows.meta,positioning,calc}

\usepackage{lmodern}

\usetikzlibrary{arrows.meta, shapes, shadows.blur, calc}

\definecolor{boxfill}{HTML}{EAF2FF}
\definecolor{boxfillB}{HTML}{DDEBFF}
\definecolor{boxdraw}{HTML}{2B4C7E}

\usepackage{docmute}

\usepackage[utf8]{inputenc}

\usepackage{mathtools}

\usepackage{amsmath}

\usepackage{amsfonts}

\usepackage{amssymb}

\usepackage{constants}

\usepackage{amsthm}
\usepackage{aliascnt}

\usepackage{makeidx}
\makeindex

\usepackage{textcomp}
\usepackage[sb]{libertine}
\usepackage[varqu,varl]{zi4}%
\usepackage[libertine,bigdelims,vvarbb]{newtxmath}
\usepackage[supscaled=1.2,raised=-.13em]{superiors}
\usepackage{bm}
\useosf

\usepackage[breaklinks,linktocpage]{hyperref}
\usepackage[hyperpageref]{backref}

\usepackage{url}
\usepackage{breakurl}

\usepackage{tikz-cd}

\usepackage{enumitem}
\usepackage[UKenglish]{babel}
\usepackage[makeroom]{cancel}
\usepackage{pstricks-add}

\definecolor{shadecolor}{rgb}{0.88,0.91,0.95}       
\usepackage{tcolorbox}

\usepackage{dynkin-diagrams}
\usetikzlibrary{backgrounds}

\usepackage{booktabs}

\usepackage{fancyhdr}
\newcommand{\R}{\mathbb{R}}

\newcommand{\End}{\operatorname{End}}

\newcommand{\Ric}{\operatorname{Ric}}
\newcommand{\Riem}{\operatorname{Riem}}

\newcommand{\Hess}{\operatorname{Hess}}

\newcommand{\Spec}{\operatorname{Spec}}

\newcommand{\Id}{\operatorname{Id}}

\renewcommand{\|}[1]{\left| \left| #1 \right| \right|}

\newcommand{\<}{\left\langle}
\renewcommand{\>}{\right\rangle}

\newcommand{\vol}{\operatorname{vol}}

\renewcommand{\tilde}[1]{\widetilde{#1}}

\newcommand\Item[1][]{%
  \ifx\relax#1\relax  \item \else \item[#1] \fi
  \abovedisplayskip=0pt\abovedisplayshortskip=0pt~\vspace*{-\baselineskip}}

\usepackage[capitalise]{cleveref}

\numberwithin{equation}{section}

\newtheorem{proposition}[equation]{Proposition}
\crefname{proposition}{Proposition}{Propositions}

\newtheorem{lemma}[equation]{Lemma}
\crefname{lemma}{Lemma}{Lemmas}

\newtheorem{corollary}[equation]{Corollary}
\crefname{corollary}{Corollary}{Corollaries}

\newtheorem*{corollary*}{Corollary}
\crefname{corollary*}{Corollary}{Corollaries}

\newtheorem{theorem}[equation]{Theorem}
\crefname{theorem}{Theorem}{Theorems}

\crefname{exercise}{Exercise}{Exercises}

\crefname{task}{Task}{Task}

\crefname{conjecture}{Conjecture}{Conjectures}

\crefname{algorithm}{Algorithm}{algorithms}

\newtheorem*{theorem*}{Theorem}
\crefname{theorem*}{Theorem}{Theorems}

\crefname{claim}{Claim}{Claims}

\theoremstyle{remark}

\crefname{question}{Question}{Questions}

\newtheorem{definition}[equation]{Definition}
\crefname{definition}{Definition}{Definitions}

\crefname{example}{Example}{Examples}

\newtheorem{remark}[equation]{Remark}
\crefname{remark}{Remark}{Remarks}

\crefname{assumption}{Assumption}{Assumptions}

\crefname{problem}{Problem}{Problems}

\usepackage{todonotes}

\author[1]{Timothy Buttsworth}
\author[1]{William Hadden}
\author[2]{Elli Heyes}
\author[2]{Daniel Platt}
\author[2]{Toby Wiseman}

\affil[1]{\small School of Mathematics and Statistics\\
The University of New South Wales\\ Kensington\\
 Sydney\\ NSW 2052, Australia}
 \affil[2]{\small Faculty of Natural Sciences\\
Imperial College London\\ South Kensington\\
 London\\ SW7 2AZ, United Kingdom}

\date{\today}
\title{The K\"ahler-Einstein metric on the third Del Pezzo surface}

\begin{document}

\maketitle

\abstract{The compact four-dimensional manifold $\mathbb{CP}_2\# 3\overline{\mathbb{CP}_2}$ is known to admit a toric K\"ahler-Einstein metric $g_{\text{KE}}$, but the metric is not known in closed form, which makes it difficult to draw conclusions about its geometry. In this article, we use a combination of analytic and computer-assisted techniques to produce an approximate Einstein metric $g$ described explicitly herein, and also prove that the true Einstein metric $g_{\text{KE}}$ is close to $g$, where both the closeness and the topology are described explicitly. As an application, we prove bounds on the first invariant eigenvalue of the Laplace-Beltrami operator, and prove that this K\"ahler-Einstein metric does not have positive holomorphic sectional curvature everywhere.} 

\tableofcontents

\section{Introduction}

Let $M$ be a smooth manifold. 
A Riemannian metric $g$ on $M$ is said to be an \textit{Einstein metric} if there exists a constant $\lambda$ so that
\begin{equation}\label{GRS}
 \mathrm{Ric} (g)=\lambda g \ \text{on} \ M,
\end{equation}
where $\mathrm{Ric}(g)$ is the Ricci curvature of $g$. One of the most important topics in both geometry and physics is finding and classifying Einstein metrics; a thorough motivation for and discussion of the subject can be found, for example, in \cite{Besse}. If the dimension of $M$ is at most three, then any complete and connected Einstein manifold is either the round sphere, hyperbolic space, Euclidean space, or a quotient thereof. On the other hand, classifications of higher-dimensional Einstein metrics are typically out of reach. Even in dimension four, where Einstein metrics are sometimes referred to as \textit{gravitational instantons}, many important open questions remain (see e.g. \cite{AndersonSurvey} for a survey of the four-dimensional situation). 

The study of Einstein metrics is challenging because \eqref{GRS} is a system of second order, quasi-linear, \emph{weakly elliptic} partial differential equations. A gauge fixing trick is employed to replace \eqref{GRS} with an equivalent elliptic problem, for which the local theory is well-understood, see e.g. \cite[Section 5]{Besse}. On the other hand, the global theory is far from understood, and more \textit{ad hoc} methods and simplifying assumptions are typically employed to study Einstein metrics in dimensions four and higher.

One simplifying assumption that is often employed in the study of Einstein metrics is that of \textit{special holonomy}, i.e., the assumption that the holonomy group of the Riemannian manifold $(M,g)$ is strictly smaller than the special orthogonal group. By work of Berger \cite{Berger} and Brown-Gray \cite{BrownGray}, it is known that apart from quotients, products and locally symmetric spaces, the possible exceptional holonomy groups are $\mathsf{U}(n)$, $\mathsf{SU}(n)$, $\mathsf{Sp}(n)\cdot \mathsf{Sp}(1)$, $\text{Sp}(n)$, $G_2$, and $\mathsf{Spin}(7)$, in which case we say that the Riemannian manifold $(M,g)$ is \textit{K\"ahler}, \textit{Calabi-Yau}, \textit{quaternion-K\"ahler}, \textit{hyper-K\"ahler}, $G_2$ or $\mathsf{Spin}(7)$, respectively.  Apart from the K\" ahler assumption, all of these holonomy assumptions are so strong that any such metric is automatically Einstein (and is even Ricci-flat, apart from the quaternion-K\"ahler case). Finding Riemannian manifolds with these holonomy groups is quite a difficult task; see \cite{Joyce} for some special examples and some other survey material. 

Given a smooth manifold $M$ with a complex structure $J$, the K\"ahler metrics are those Riemannian metrics $g$ for which the so-called \textit{K\"ahler form} $\omega$, which is defined according to the rule  $\omega(u,Jv)=g(u,v)$, is closed. The question of whether the complex manifold $(M,J)$ admits a K\"ahler-Einstein metric has a complete answer  which is intimately related to the sign of the \textit{first Chern class $c_1(M,J)$}. Of the compact and four-dimensional complex manifolds with positive first Chern class, Tian \cite{Tian} shows that there is a K\"ahler-Einstein metric with $\lambda>0$ precisely on the manifolds $\mathbb{CP}^2$, $\mathbb{CP}^1\times \mathbb{CP}^1$ and $\mathbb{CP}^2\# k \overline{\mathbb{CP}^2}$ for $3\le k\le 8$ (with the canonical choice of complex structure). The K\"ahler-Einstein metrics on $\mathbb{CP}^2$ and  $\mathbb{CP}^1\times \mathbb{CP}^1$ are both known explicitly (they are the Fubini-Study metric and the round product metrics, respectively), but the K\"ahler-Einstein metrics on $\mathbb{CP}^2\# k \overline{\mathbb{CP}^2}$ for $3\le k\le 8$ are much more opaque. 

Another popular assumption that is often employed in the study of Einstein metrics is \textit{symmetry}, i.e., one takes a smooth manifold $M$ equipped with a smooth action by a Lie group $\mathsf{G}$, and searches for Einstein metrics amongst those for which the Lie group acts by isometries. Rather than reducing the system of PDEs to a scalar PDE (as in the K\"ahler case), the effect of assuming symmetry is to reduce the number of independent variables in the equations. For example, if the Lie group $\mathsf{G}$ acts \textit{transitively} on $M$, then $M$ is referred to as a \textit{homogeneous space}, and a Riemannian metric on which $\mathsf{G}$ acts isometrically is called a \textit{homogeneous Riemannian manifold}; solving \eqref{GRS} reduces to a system of \textit{algebraic equations} if we restrict our search to the space of homogeneous Riemannian manifolds. Although these equations are often quite complicated, much is now known about homogeneous Einstein metrics; see \cite{Alekseevskii}, \cite{BWZ} and \cite{BohmLafuente} and the references therein for some of the highlights in this field. Another assumption is that the Lie group $\mathsf{G}$ acts with \textit{cohomogeneity one}, i.e., the generic orbits have co-dimension one. In this case, \eqref{GRS} reduces to a system of ordinary differential equations. This field of enquiry has also been pursued diligently in recent decades, and while the references are too numerous to detail fully, we highlight the construction of a novel cohomogeneity one Einstein metric  on the compact four-dimensional manifold $\mathbb{CP}^2\# \overline{\mathbb{CP}^2}$ by Page \cite{Page}, a classification of all cohomogeneity one Einstein metrics on closed four-manifolds \cite{Broder26}, and a surprising sequence of pairwise-non-isometric cohomogeneity one Einstein metrics on the sphere $S^n$ for $5\le n\le 9$ by B\"ohm \cite{Bohm}.  

Another popular symmetry condition is the assumption that $M$ is $2n$-dimensional, and that there is an effective action on $M$ by the $n$-torus $T^n$, in which case we say the manifold is \textit{toric}. There are a number of important examples of toric Einstein metrics in dimension four. Firstly, the round sphere $S^4$ admits an isometric action of $\mathsf{O}(4)$, and this action contains an effective action of $\mathsf{SO}(2)\times \mathsf{SO}(2)\simeq T^2$, so the round sphere is \textit{toric}. Similarly, the Fubini-Study metric on $\mathbb{CP}^2$ is also toric. For a less trivial four-dimensional example, it is interesting to note that the Einstein metric on $\mathbb{CP}^2\# \overline{\mathbb{CP}^2}$ from Page \cite{Page} is invariant under a cohomogeneity-one action of $\mathsf{U}(2)$, which also contains a copy of $\mathsf{SO}(2)\times \mathsf{SO}(2)\simeq T^2$ acting effectively, so the Page metric is also toric. Another notable example is the Chen-Lebrun-Weber metric on $\mathbb{CP}^2\# 2 \overline{\mathbb{CP}^2}$ which was constructed using the toric assumption in \cite{ChenLebrunWeber}. Finally, it is also known that the K\"ahler-Einstein metric on $\mathbb{CP}^2\# 3 \overline{\mathbb{CP}^2}$ is  toric as well as K\"ahler, in stark contrast to the Page and Chen-Lebrun-Weber metrics which are both toric, but \textit{conformally K\"ahler}, rather than K\"ahler. 

The Einstein metric on $\mathbb{CP}^2\# 3\overline{\mathbb{CP}^2}$ \textbf{is the only known compact four-dimensional Einstein metric which is both K\"ahler and toric, but for which no closed form solution exists.} 
This leads to an embarrassing situation for differential geometers: even with the independently-powerful assumptions of toric and K\"ahler, we know almost nothing about this Einstein metric. 
This situation has been partially remedied in  \cite{Wiseman}, where the authors construct several numerical approximations for this Einstein metric.
The main purpose of this paper is to refine these numerical techniques, and combine them with several rigorous analytical techniques to \textit{verify} that these approximations do actually get us close to the true Einstein metric. 

\begin{theorem}\label{maintheorem}
    Let $(M_P,\omega_P,g,J)$ be the approximate Kähler-Einstein metric from \cref{Data} on $\mathbb{CP}^2\# 3\overline{\mathbb{CP}^2}$.
    Then there exists $\varphi \in H^5(M)$ such that $(M_P,\omega_P+i \partial \bar \partial \varphi, J)$ is a Kähler-Einstein metric in the Kähler class $[\omega_P]$ and satisfies the bound for $\|{\varphi}_{H^5(M,g)}$ recorded in \cref{tab:numerical-constants}. 
\end{theorem}

As a consequence, we are now in a position to draw some reliable conclusions about the Kähler-Einstein metric $g_{KE}$ from \cref{maintheorem}.
We start with the holomorphic sectional curvature of this K\"ahler-Einstein metric. 

\begin{corollary}\label{corollary:not-positive-hsc}
The holomorphic sectional curvature of $g_{\text{KE}}$ is not uniformly positive. 
\end{corollary}
This is related to \cite[Conjecture 1.3]{broder2023some}, suggesting that if $\Ric>0$, then there is a cohomologous metric with positive holomorphic sectional curvature (the recent work \cite{zhang2026positiveholomorphicsectionalcurvature} confirms that such a metric indeed exists).
Our result shows that allowing cohomologous metrics in the conjecture statement is essential, because the Kähler-Einstein metric itself has positive Ricci, but its holomorphic sectional curvature changes its sign.
Yang's proof of Yau's conjecture in \cite{yang2018rc} gives a necessary criterion for the existence of a strictly positive holomorphic sectional curvature metric, however the third del Pezzo surface satisfies this criterion, so it does not say anything new about this surface.

We also consider the spectrum of the Laplace-Beltrami operator of this K\"ahler-Einstein metric. Since we have toric symmetries, \cite[Theorem 11.52]{Besse} implies that the first eigenvalue is precisely $2$, which corresponds to a toric eigenfunction. However, there is an additional $D_6$-symmetry coming from the automorphism group of the image of the moment map, and it is natural to consider the first eigenfunction which is invariant under \textit{both} symmetries. 

\begin{corollary}
    The first non-zero $T^2 \rtimes D_6$-invariant eigenvalue $\lambda_{\min}^{KE}$ of the Laplace operator for $g^{KE}$ satisfies the lower and upper bound recorded in \cref{tab:numerical-constants}.
\end{corollary}

Torus-invariant eigenvalues can behave quite differently from non-invariant eigenvalues, see e.g. \cite{abreu2002invariant}.
Toric Kähler manifolds have received special attention, see e.g. \cite{legendre2018toric}.
Even more specifically, the eigenvalues on the third del Pezzo surface have been studied \emph{numerically} in the literature:
\cite{Doran2008} compute the first non-zero toric and $D_6$-eigenvalue with two different methods to be approximately $6.32$, which is numerically confirmed in \cite[Section 4.4]{HallMurphy}.
It is then shown in Theorem 1.2 and Section 4.4 of
\cite{HallMurphy} that the second eigenvalue of the Laplace-Beltrami without any symmetry assumption satisfies $\lambda_2\le 5.29$. 
So remarkably, this invariant eigenvalue is not even the \textit{second smallest} eigenvalue. 

The approximate Kähler-Einstein metric is fully explicit, so it is possible to numerically search for other interesting geometric structures on this Riemannian manifold.
As one example, minimal surfaces on Kähler-Einstein manifolds have been studied in the literature, for example \cite{Lee1994,Arezzo2000}.
Particular attention has been paid to $\mathbb{CP}^2$ and nearly Kähler manifolds, in which cases the metric is given explicitly.
The cohomogeneity two situation is particularly amenable to numerical calculations, as demonstrated in \cite{Tomter1996}, using the theory developed in \cite{Hsiang1971}.
As a second example, one could search for Hermite-Einstein connections.
From \cite{Donaldson1985} it is known that stable bundles admit Hermite-Einstein connections.
Explicitly computing them is needed for applications in Physics and is an active area of research, see \cite{Douglas2007} for the first paper, and the more recent \cite{Ashmore2022} and the references therein.

The minimal surface and Hermite-Einstein equations are elliptic, so we expect that one can apply the same approach of constructing approximate solutions and perturbing them to genuine solutions with respect to the approximate Kähler-Einstein metric $g$.
\Cref{maintheorem} can then be used to perturb them to solutions for the \emph{genuine} Kähler-Einstein metric.

On a more general note, there are numerous open problems for other Kähler manifolds, in particular Calabi-Yau manifolds.
The strategy of computing an approximate Kähler-Einstein metric and perturbing it to a genuine Kähler-Einstein metric (in this case Ricci-flat) is applicable in this setting, too.
There is a sprawling literature on approximate Calabi-Yau metrics, see the foundational \cite{Headrick2005,Donaldson2005}, with recent approaches too many to list.
Certifying that one of these numerical candidates is near the genuine Calabi-Yau metric is a central challenge in theoretical physics.
In mathematics, it would allow constructions of new minimal surfaces, Hermite-Einstein connections, and solutions to other elliptic systems such as the Strominger system.

The overall strategy for the proof of Theorem \ref{maintheorem} is as follows:
\begin{itemize}
    \item use numerical techniques to construct an \textit{ostensibly}-accurate approximate K\"ahler-Einstein metric by solving for a symplectic potential as described in \cref{subsection:numerical-approach};    
    \item verify the approximate solution is actually accurate using interval-arithmetic to estimate the \textit{a posteriori} error (residue) in the equation, see \cref{subsection:a-posteriori-error};
    
    \item combine an analysis of the linearisation of the Einstein equation with fixed point methods to prove that there is a true solution nearby, see \cref{perturb}.
\end{itemize} 
Thus, Theorem \ref{maintheorem} is part of a growing collection of theorems in the field of partial differential equations that are proved using computer-assisted techniques. A survey of these techniques is available in \cite{PDESurvey}. See also \cite{ButtsworthHodgkinson24} and \cite{Wang} for the use of these techniques in the recent construction of novel cohomogeneity-one Einstein metrics. 

Let us discuss in some more detail how the proof of Theorem \ref{maintheorem} proceeds, and how the paper is organised:
\begin{itemize}
    \item Section \ref{prelim} provides a detailed study of Riemannian metrics on $\mathbb{CP}^2\# 3 \overline{\mathbb{CP}^2}$ that are both toric and K\"ahler, and also provides some estimates for the norms of various covariant derivatives of the Riemann curvature operator. We also discuss how the Einstein condition \eqref{GRS} appears for such metrics, simplifying the condition into a single scalar PDE and compute its linearisation.
    
    \item Section \ref{approx} is where we produce a Riemannian metric $g$ on $\mathbb{CP}^2\# 3\overline{\mathbb{CP}^2}$ which is both K\" ahler and toric, and which approximately satisfies the simplified Einstein condition discussed in Section \ref{prelim}. We use interval-arithmetic to rigorously \textit{a posteriori} estimate: several curvature quantities; the error of the approximation, i.e. the extent to which $g$ fails to be a true Einstein metric; and a lower bound for the smallest $T^2 \rtimes D_6$-invariant eigenvalue of the Laplacian with respect to $g$. 
    
    \item Section \ref{perturb} is where we ``fix the error" from Section \ref{approx} with a perturbation, thus constructing a true toric and K\"ahler Einstein metric $g^{KE}$. The perturbation is found by casting the Einstein equation in fixed-point form, guided by the linearised form of this equation found in Section \ref{prelim}, and invoking the Banach fixed-point theorem. This requires estimating the size of the inverse-Laplacian, which is possible with the eigenvalue estimate of Section \ref{approx}, and also requires estimating the ``non-linear remainder term" of the linearisation, which requires further estimates presented in Appendix \ref{GAE}. 
    \item Section 5 exhibits two applications. We rigorously bound the first eigenvalue of the Laplacian associated to this metric $g^{KE}$ acting on toric functions. We also look at the holomorphic sectional curvature, and prove a sign change.
\end{itemize}

\begin{figure}[htbp]
\centering
\scalebox{0.7}{\begin{tikzpicture}[
  font=\sffamily,
  block/.style={
    rounded corners=2mm,
    draw=boxdraw,
    very thick,
    top color=boxfill,
    bottom color=boxfillB,
    text width=4.3cm,
    align=center,
    inner sep=6pt,
    general shadow={
      shadow xshift=0.6ex,
      shadow yshift=-0.6ex,
      fill=black!20,
      draw=none
    }
  },
  wire/.style={
    draw=black!65,
    thick,
    line cap=round,
    line join=round,
    rounded corners=6pt
  },
  arrow/.style={
    wire,
    -{Stealth[length=2.2mm,width=1.9mm]},
    shorten <=0pt,
    shorten >=3.2pt
  }
]

\def\xsep{5.55cm}
\def\ysep{2.15cm}
\def\botsep{2.10cm}

\node[block] (A) at (0,0)
  {Approximate solution\\[-1pt]
   {\footnotesize\textit{\cref{subsection:numerical-approach}}}};

\node[block] (C) at (0,-\ysep)
  {$\|{\Ric},\ \|{\Riem}$ bounds\\[-1pt]
   {\footnotesize\textit{
     \cref{subsection:a-posteriori-geometric-bounds}}}};

\node[block] (B) at (-\xsep,-2*\ysep)
  {A~posteriori error\\[-1pt]
   {\footnotesize\textit{\cref{subsection:a-posteriori-error}}}};

\node[block] (E) at (0,-2*\ysep)
  {Eigenvalue for approx.\ solution\\[-1pt]
   {\footnotesize\textit{
     \cref{subsection:non-sharp-eigenvalue-bound}}}};

\node[block] (F) at (\xsep,-2*\ysep)
  {A~priori estimate\\[-1pt]
   {\footnotesize\textit{\cref{corollary:a-priori-estimate}}}};

\node[block] (D) at (-\xsep,-3*\ysep)
  {Higher order estimate\\[-1pt]
   {\footnotesize\textit{\cref{subsection:quadratic-estimate}}}};

\node[block] (H) at (0,-3*\ysep)
  {Injectivity estimate\\[-1pt]
   {\footnotesize\textit{\cref{corollary:inj-estimate}}}};

\node[block] (I) at (0,-4*\ysep)
  {Exact solution\\[-1pt]
   {\footnotesize\textit{\cref{maintheorem}}}};

\node[block] (J) at (-0.62*\xsep,-4*\ysep-\botsep)
  {Holomorphic sectional curvature\\[-1pt]
   {\footnotesize\textit{\cref{HSC}}}};

\node[block] (K) at (0.62*\xsep,-4*\ysep-\botsep)
  {Eigenvalue for true solution\\[-1pt]
   {\footnotesize\textit{\cref{Firsteigenvalueestimate}}}};

\draw[arrow] (A.south) -- (C.north);

\path let
  \p1=(C.south),
  \p2=(E.north),
  \p3=(B.north),
  \p4=(F.north)
in
  coordinate (ForkC) at (\x1,{(\y1+\y2)/2})
  coordinate (ForkB) at (\x3,{(\y1+\y2)/2})
  coordinate (ForkF) at (\x4,{(\y1+\y2)/2});

\draw[wire]  (C.south) -- (ForkC);
\draw[arrow] (ForkC) -- (ForkB) -- (B.north);
\draw[arrow] (ForkC) -- (E.north);
\draw[arrow] (ForkC) -- (ForkF) -- (F.north);

\draw[arrow] (B.south) -- (D.north);

\path (B.west) ++(-9mm,0) coordinate (BIleft);
\path let \p1=(BIleft), \p2=(I.west) in
  coordinate (BIturn) at (\x1,\y2);

\draw[arrow]
  (B.west) -- (BIleft) -- (BIturn) -- (I.west);

\coordinate (Ilefttop)
  at ($(I.north west)!0.30!(I.north east)$);

\path let \p1=(D.south), \p2=(Ilefttop) in
  coordinate (DIleft)  at (\x1,{(\y1+\y2)/2})
  coordinate (DIright) at (\x2,{(\y1+\y2)/2});

\draw[arrow]
  (D.south) -- (DIleft) -- (DIright) -- (Ilefttop);

\draw[arrow] (E.south) -- (H.north);

\path let \p1=(F.south), \p2=(H.north) in
  coordinate (FH1) at (\x1,{(\y1+\y2)/2})
  coordinate (FH2) at (\x2,{(\y1+\y2)/2});

\draw[arrow]
  (F.south) -- (FH1) -- (FH2) -- (H.north);

\draw[arrow] (H.south) -- (I.north);

\path let
  \p1=(I.south),
  \p2=(J.north),
  \p3=(K.north)
in
  coordinate (ForkI) at (\x1,{(\y1+\y2)/2})
  coordinate (ForkJ) at (\x2,{(\y1+\y2)/2})
  coordinate (ForkK) at (\x3,{(\y1+\y2)/2});

\draw[wire]  (I.south) -- (ForkI);
\draw[arrow] (ForkI) -- (ForkJ) -- (J.north);
\draw[arrow] (ForkI) -- (ForkK) -- (K.north);

\end{tikzpicture}}
\end{figure}

The computer code carrying out the certified calculations described in the paper can be found in the companion repository \url{https://github.com/william-hadden/toric-verified-numerics}.

\section*{Acknowledgments}
We are grateful to the \textit{Problem Solving Workshop: Computational Geometric Analysis}, held at the CUNY Graduate Center, at which work on this program began. We are grateful to Kyle Broder for his suggestions regarding the holomorphic sectional curvature applications of our approximations. The fourth author thanks Michael Douglas for helpful conversations. 

\newpage
\section{Preliminaries}\label{prelim}
The purpose of this paper is to study a K\"ahler-Einstein metric on $\mathbb{CP}_2\# 3 \overline{\mathbb{CP}_2}$ which happens to be invariant under a certain effective $T^2$ action. 
 In this section, we discuss the details of this manifold, its complex structure, the associated torus action, as well as the corresponding space of $T^2$-invariant K\"ahler metrics, including their curvatures, the Einstein condition for such metrics,  as well as the linearised Einstein condition. 

\subsection{Canonical K\"ahler geometries from Delzant polytopes}
Our pathway to K\"ahler geometries on $\mathbb{CP}_2\# 3 \overline{\mathbb{CP}_2}$ will be somewhat indirect, and proceeds by building up the manifold from its torus action, rather than its complex structure. 

To start, we consider a compact and connected four-dimensional symplectic manifold $(M,\omega)$ equipped with an effective torus action $\tau:T^2\times M\to M$. If $\tau$ preserves the symplectic form $\omega$ and admits a momentum map $\mu:M\to \mathbb{R}^2$, we say that the triple  $(M,\omega,\tau)$ is a \textit{Hamiltonian $T^2$-space}. 
Delzant \cite{Delzant} demonstrates that the image $\mu(M)\subseteq \mathbb{R}^2$ of the momentum map is always a convex polytope $P$, satisfying the following so-called \textit{Delzant conditions}:
\begin{itemize}
    \item each vertex $v$ meets exactly two edges; 
    \item the edges meeting each vertex $v$ are rational, i.e., the two edges can be described as  $v+te_1$ and $v+te_2$, for some $e_1,e_2\in \mathbb{Z}\times \mathbb{Z}$; and 
    \item it is possible to choose $e_1,e_2$ from the previous line to be a basis for $\mathbb{Z}\times \mathbb{Z}$. 
\end{itemize}
Delzant polytopes in $\mathbb{R}^2$ are typically found by specifying affine functions $l_r:\mathbb{R}^2\to \mathbb{R}$ for $r=1,\cdots,N$, and defining $P\subseteq \mathbb{R}^2$ to be the set of points $x$ where $l_r(x)\ge 0$ for $r=1,\cdots, N$. 
Delzant proves that there is a one-to-one correspondence between such polytopes and Hamiltonian $T^2$-spaces. 
\begin{theorem}\label{DelzantConstruction}
    For each Delzant polytope in $P\subseteq \mathbb{R}^2$, there is a Hamiltonian $T^2$-space $(M_P,\omega_P,\tau_P)$ whose momentum map is surjective onto $P$. Furthermore, if $(M,\omega,\tau)$ is any other Hamiltonian $T^2$-space with the same momentum map image, then the two spaces are isomorphic (in the Hamiltonian $T^2$-space sense).  
\end{theorem}
Let us talk about the topological construction of these Hamiltonian $T^2$-spaces. 
For a given Delzant polytope $P$ with interior $P^o$, the associated Hamiltonian $T^2$-space can be found by applying a certain compactification process to the open manifold $P^o\times T^2$, with $T^2$ acting freely in the obvious way (see Chapter 1 of \cite{Guillemin} for more details). Then the principal part of $M_P$ (the points where $T^2$ acts freely) is equivariantly diffeomorphic to $P^o\times T^2$ through the moment map $\mu_P$. On the other hand, points $p\in M_P$ with $\mu_P(p)\in \partial P$ are on singular orbits of the group action.

This construction of a Hamiltonian $T^2$-space $M$ leads naturally to K\"ahler geometry. Indeed, using $(x_1,x_2)$ as the standard coordinates for $P\subseteq \mathbb{R}^2$ and $(\theta_1,\theta_2)$ as coordinates for $T^2$, the associated symplectic form appears as $\omega_P=dx_1\wedge d\theta_1+dx_2\wedge d\theta_2$. Further, Abreu explains in \cite[Section 2.1]{Abreu} that $(M_P,\omega_P,\tau_P)$ comes equipped with a canonical complex structure $J_P$ that is both compatible with $\omega_P$ and invariant under the action of $\tau_P$. In this way, we get a canonical K\"ahler-toric geometry on a compact four-dimensional manifold associated to each Delzant polytope $P\subseteq \mathbb{R}^2$. 

\subsection{Canonical K\"ahler geometry on $\mathbb{CP}^2\# 3\overline{\mathbb{CP}^2}$}
Rather than provide a detailed description of the canonical K\"ahler geometries for each Delzant polytope, we will specialise immediately to the Delzant polytope of primary interest in this paper, namely the polytope $P\subseteq \mathbb{R}^2$ defined by the condition that the following six affine functions are non-negative: 
\begin{align}\label{dP3polytope}
    l_{\pm 1}(x_1,x_2)=1\pm x_1, \qquad l_{\pm 2}(x_1,x_2)=1\pm x_2, \qquad l_{\pm 3}(x_1,x_2)=1\mp (x_1+x_2). 
\end{align}
Note in particular that $P$ is a hexagon with the following six vertices: 
\begin{align*}
    (1,0), (0,1), (-1,1), (-1,0), (0,-1), (1,-1). 
\end{align*}

Using $(x_1,x_2,\theta_1,\theta_2)$ to denote the standard coordinates for $P^o\times T^2\subset \mathbb{R}^2\times T^2$, we have an obvious basis $\{\partial_{x_1},\partial_{x_2},\partial_{\theta_1},\partial_{\theta_2}\}$ on the tangent space of each point in the principal part $M_P^o$ of $M_P$. In \cite{Abreu}, and \cite{Guillemin94}, it is shown that, with respect to this basis, the canonical complex structure $J_P$ has the matrix form 
\begin{align*}
  J_P=  \begin{pmatrix}
        0&0&*&*\\
        0&0&*&*\\
        \frac{\partial^2 u_{P}}{\partial x_1^2}&\frac{\partial^2 u_P}{\partial x_1 \partial x_2}&0&0\\
         \frac{\partial^2 u_P}{\partial x_1\partial x_2}&\frac{\partial^2 u_P}{\partial x_2^2}&0&0
    \end{pmatrix},
\end{align*}
where the $*$s are chosen so that $J_P$ squares to minus the identity, and $u_P:P\to \mathbb{R}$ is given by
\begin{align}\label{canonicalpotential}
    u_P(x)=\frac{1}{2}\sum_{|r|=1,2,3}l_r(x)\text{log}(l_r(x)). 
\end{align}

It turns out that this abstract  K\"ahler geometry generated from our specific Delzant polytope $P$ is indeed isomorphic to our target manifold $\mathbb{CP}^2\# 3\overline{\mathbb{CP}^{2}}$.

\begin{proposition}
    Let $(M_P, \omega_P, J_P)$ be the canonical Kähler geometry associated to the hexagon $P \subset \R^2$ from \eqref{dP3polytope}.
    Then $M_P$ is diffeomorphic to $\mathbb{CP}^2\# 3\overline{\mathbb{CP}^2}$ and the K\"ahler class $[\omega_P]$ is equal to $2\pi c_1(M_P,J_P)$.
\end{proposition}

\begin{proof}
    In \cite[p.12]{Abreu}, Abreau identifies the hexagon as the image under the momentum map of $X=\mathbb{CP}^2\# 3\overline{\mathbb{CP}^2}$ with Kähler form s.t. $[\omega_X]=2\pi c_1(X,J_X)$.
    By \cref{DelzantConstruction}, $X$ and $M_P$ are symplectomorphic, hence also $[\omega_P]=2\pi c_1(M_P,J_P)$.
\end{proof}

There is an important observation to make about the behaviour of the complex structure $J_P$ as we approach the singular orbits. A simple computation gives

\begin{align}\label{ucanHess}
    \begin{pmatrix}
        \frac{\partial^2 u_{P}}{\partial x_1^2}&\frac{\partial^2 u_P}{\partial x_1 \partial x_2}\\
         \frac{\partial^2 u_P}{\partial x_1\partial x_2}&\frac{\partial^2 u_P}{\partial x_2^2}
    \end{pmatrix}=\frac{1}{2}\begin{pmatrix}
        \frac{1}{l_{1}}+\frac{1}{l_{-1}}+\frac{1}{l_{3}}+\frac{1}{l_{-3}}&\frac{1}{l_{3}}+\frac{1}{l_{-3}}\\
        \frac{1}{l_{3}}+\frac{1}{l_{-3}}&\frac{1}{l_{2}}+\frac{1}{l_{-2}}+\frac{1}{l_{3}}+\frac{1}{l_{-3}}
    \end{pmatrix}
\end{align}
and 
\begin{align}\label{deltacan}
\begin{split}
    \text{det}\begin{pmatrix}
        \frac{\partial^2 u_{P}}{\partial x_1^2}&\frac{\partial^2 u_P}{\partial x_1 \partial x_2}\\
         \frac{\partial^2 u_P}{\partial x_1\partial x_2}&\frac{\partial^2 u_P}{\partial x_2^2}
    \end{pmatrix}&=\frac{1}{4}\left( \frac{1}{l_{1}}+\frac{1}{l_{-1}}\right)\left(\frac{1}{l_{2}}+\frac{1}{l_{-2}}\right)+\frac{1}{4}\left( \frac{1}{l_{3}}+\frac{1}{l_{-3}}\right)\left(\frac{1}{l_{1}}+\frac{1}{l_{-1}}+\frac{1}{l_{2}}+\frac{1}{l_{-2}}\right)\\
    &=\left(\delta_P(x)\Pi_{|r|=1,2,3} l_r(x)\right)^{-1},
    \end{split}
\end{align}
where $\delta_P$ is smooth and positive on $P$ (including on the boundary). 
Clearly $u_P$ is not smoothly extendible to $\partial P$, but this is to be expected because the equivariant diffeomorphism from $P^o\times T^2$ to $M_P$ cannot be extended to the boundary of the polytope. In fact, the singularities of $u_P$ precisely accommodate that degeneration, thus allowing the resulting complex structure $J_P$ to be smooth over all of $M_P$, not just the principal part $M_P^o$. 

\subsection{Other toric K\"ahler metrics on $\mathbb{CP}^2\# 3\overline{\mathbb{CP}^{2}}$}\label{othertoricKahler}
We now have a compact four-dimensional toric K\"ahler manifold $(M_P,\omega_P,J_P)$, with its canonical toric K\"ahler metric $g_P$. This Riemannian metric is not necessarily Einstein, so we search for a K\"ahler-Einstein metric by fixing the symplectic form $\omega_P$, and replacing $J_P$ with another compatible toric complex structure $J$. As discussed in Appendix A of \cite{Abreu} (especially Proposition A.1), this is equivariantly-biholorphically equivalent to fixing the complex structure $J_P$, and varying the symplectic form $\omega$, which is usually what is done in the search for K\"ahler-Einstein metrics. 

We will again define our new complex structure $J$ in the basis $\{\partial_{x_1},\partial_{x_2},\partial_{\theta_1},\partial_{\theta_2}\}$ according to the matrix 
\begin{align}\label{newcomplex}
    J
    =
    \begin{pmatrix}
        0&0&*&*\\
        0&0&*&*\\
        \frac{\partial^2 u}{\partial x_1^2}&\frac{\partial^2 u}{\partial x_1 \partial x_2}&0&0\\
         \frac{\partial^2 u}{\partial x_1\partial x_2}&\frac{\partial^2 u}{\partial x_2^2}&0&0
    \end{pmatrix},
\end{align}
where the $*$s are chosen so that $J$ squares to minus the identity, and 
\begin{align}\label{perturbedpotential}
    u=u_P+h,
\end{align} where $u_P$ was defined in \eqref{canonicalpotential}, and $h:P\to \mathbb{R}$ is smooth, all the way up to and including the boundary of $P$. Theorem 2.8 of \cite{Abreu} tells us that the construction actually reaches \textit{all} compatible complex structures.
\begin{theorem}\label{smoothnessthm}
    The complex structure defined according to \cref{newcomplex} and \cref{perturbedpotential} is smooth, toric and compatible with $\omega_P$ if the Hessian of $u$ is positive definite on $P^o$, and its determinant has the form 
    \begin{align}
        \label{equation:delta-smoothness}
        \text{det}(\text{Hess}_x(u))(x)=\left(\delta(x)\Pi _{|r|=1,2,3}l_r(x)\right)^{-1}
    \end{align} 
    for some function $\delta:P\to \mathbb{R}$ that is smooth and positive, all the way up to and including the boundary $\partial P$. Furthermore, if the complex structure $J$  is toric and compatible with $\omega_P$, then $J$ satisfies \cref{newcomplex},  \cref{perturbedpotential} for some smooth function $h:P\to \mathbb{R}$, as well as \cref{equation:delta-smoothness} for some smooth and positive function $\delta:P\to \mathbb{R}$. 
\end{theorem}
Since the positivity of $\delta$ is crucial in our construction, it is helpful to describe it more explicitly. 
\begin{lemma}\label{deltaform}
    For a given smooth $h:P\to \mathbb{R}$, the function $\delta$ defined according to \cref{equation:delta-smoothness} satisfies 
    \begin{align*}
    \delta_P(x)\delta(x)^{-1}=1+\delta_P(x)\Pi _{|r|=1,2,3}l_r(x) \left(\text{det}(\text{Hess}_x (h))+\frac{\partial^2 h}{\partial x_1^2}\frac{\partial^2 u_P}{\partial x_2^2}  +\frac{\partial^2 u_P}{\partial x_1^2}\frac{\partial^2 h}{\partial x_2^2}-2 \frac{\partial^2 h}{\partial_{x_1}\partial_{x_2}}\frac{\partial^2 u_P}{\partial_{x_1}\partial_{x_2}}\right),
\end{align*}
where $\delta_P$ is given by \cref{deltacan}. 
\end{lemma}
\begin{proof}
Clearly 
\begin{align*}
     \text{det}(\text{Hess}_x(u))&=\text{det}(\text{Hess}_x(u_P)+\text{Hess}_x(h))\\
     &=\text{det}(\text{Hess}_x(u_P))+\text{det}(\text{Hess}_x(h))+\frac{\partial^2 h}{\partial x_1^2}\frac{\partial^2 u_P}{\partial x_2^2}  +\frac{\partial^2 u_P}{\partial x_1^2}\frac{\partial^2 h}{\partial x_2^2}-2 \frac{\partial^2 h}{\partial_{x_1}\partial_{x_2}}\frac{\partial^2 u_P}{\partial_{x_1}\partial_{x_2}};  
\end{align*}
combining with \cref{deltacan} gives the result. 
\end{proof}
Equation \eqref{ucanHess} makes it clear that the expression $\delta_P\delta^{-1}$ in Lemma \ref{deltaform} is smooth (since the product $\Pi _{|r|=1,2,3}l_r(x)$ can be used to cancel all singular terms in $\text{Hess}_x(u_P)$). In determining whether $u$ can be used to determine a compatible complex structure $J$, it therefore suffices to verify that the expression for $\delta_P(x)\delta^{-1}(x)$ is positive, but this is only possible once $h$ is specified. 

If the conditions of Theorem \ref{smoothnessthm} hold, then the Riemannian metric associated to the toric and compatible complex structure $J$ is given by 
\begin{align}\label{toricKahlermetric}
    g=u_{ij}dx_i dx_j+u^{ij}d\theta_id\theta_j,
\end{align}
where $u_{ij}=\frac{\partial^2 u}{\partial{x_i}\partial{x_j}}$, and $u^{ij}$ are the inverse components of this same matrix, i.e., 
\begin{align}
    \label{inverse_metric}
    \begin{pmatrix}
        u^{11}&u^{12}\\
        u^{21}&u^{22}
    \end{pmatrix}=\frac{1}{U} \begin{pmatrix}
        u_{22}&-u_{12}\\
        -u_{21}&u_{11}
    \end{pmatrix},
\end{align}
where 
\begin{align*}
    \frac{1}{U}=\frac{1}{u_{11}u_{22}-u_{12}u_{21}}=\left(\delta(x)\Pi _{|r|=1,2,3}l_r(x)\right). 
\end{align*}There are some simple yet important geometric observations that hold for such metrics. First, the volume form on the principal part of the manifold is given by 
\begin{align}\label{volumeform}
    dV_g=dx_1dx_2d\theta_1d\theta_2,
\end{align}
and in particular, the total volume of $(M_P,g)$ is always $12\pi^2$ (since the Euclidean area of $P\subseteq \mathbb{R}^2$  is $3$). Secondly, the positive definite Laplacian of $g$ acting on a toric and scalar function $\varphi:M_P\to \mathbb{R}$ is given by 
\begin{align}\label{toricLaplacian}
    \Delta_g(\varphi)=
    -
    \sum_{i,j=1}^{2}\frac{\partial}{\partial_{x_i}}\left(u^{ij}\frac{\partial \varphi}{\partial{x_j}}\right).
\end{align}

\subsection{Toric K\"ahler curvature}
In this paper, we are interested in determining whether or not a  toric K\"ahler metric on $M_P$ of the form \cref{toricKahlermetric} is Einstein. It is possible to compute the Ricci curvature of such a metric in terms of the matrix components $u_{ij}$, the inverse components $u^{ij}$, as well as some of their derivatives. Since $u^{ij}:P\to \mathbb{R}$ is smooth, but $u_{ij}:P\to \mathbb{R}$ is not, we will always aim to apply our derivatives to \textit{inverse components}. For example, we use the notation $u^{ij}_{ab}=\frac{\partial^2 u^{ij}}{\partial_{x_a}\partial_{x_b}}$. 
\begin{proposition}
    The Ricci curvature of the Riemannian metric \eqref{toricKahlermetric} is given by 
    \begin{align}\label{toricKahlerRicci}
        \Ric(g)= R_{ij}dx_idx_j+R^{ij}d\theta_id\theta_j,
    \end{align}
    where 
    \begin{align*}
        R_{ij}=-\frac{1}{2}u_{im}u_{jn}u^{nr}_{lr}u^{ml},
    \end{align*}
      and $R^{ij}=\sum_{k,l=1}^{2} u^{ik}u^{jl} R_{kl}$.
\end{proposition}
\begin{proof}
   This computation is standard for toric K\"ahler metrics, but follows directly from results in Appendix \ref{CDE}.
\end{proof}
Thus, we are able to determine how close the Riemannian metric \cref{toricKahlermetric} is to being Einstein at each point on $M_P$. 

\begin{corollary}\label{Riccibounds}
    If $g$ is the toric K\"ahler Riemannian metric on $M_P$ having the form \cref{toricKahlermetric}, then for each $\lambda\in \mathbb{R}$, we have 
    \begin{align*}
        |\Ric(g)-\lambda g|_g^2=\frac{1}{2}\sum_{i,k,a,b=1}^2u^{i a}_{k a}u^{k b}_{i b}+2\lambda\sum_{a,b=1}^{2}u_{ab}^{ab}+4\lambda^2.
    \end{align*}
\end{corollary}

For our purposes, it is also important to consider the Riemann curvature of $g$, its first and second covariant derivatives and the norms of these three objects. These formulae, as well as their proofs, are also presented in Appendix \ref{CDE}.

\subsection{The $i\partial\overline{\partial}$-operator for toric K\"ahler metrics}
One can use Corollary \ref{Riccibounds} to determine when a Riemannian metric of the form \cref{toricKahlermetric} is Einstein, but since such metrics are all K\"ahler with respect to the associated complex structure $J$ from \cref{newcomplex}, there are other ways to simplify the situation. Indeed, a composition of the differential operators $\partial=\partial_J$ and $\bar{\partial}=\bar{\partial}_J$ (both taken with respect to complex structure $J$) is known to be enormously helpful when formulating the Einstein condition for K\"ahler metrics, and even more can be said in the presence of our $T^2$-symmetries. In this subsection, we examine these simplifications, and start by giving a formula for the $i\partial\overline{\partial}$ operator. 

\begin{lemma}\label{deldelbar}
    \leavevmode
   For any $T^2$-invariant smooth function $\varphi:M_P\to \mathbb{R}$, the smooth $2$-form $i\partial \overline{\partial} \varphi$ on $M_P$ is given by 
    \begin{align}\label{ddb}
        i\partial \overline{\partial} \varphi=\frac{1}{2}\sum_{j,l} A_{jl}dx_l\wedge d\theta_j, \qquad A_{jl}=\sum_{k} \partial_{x_l} (u^{jk}\varphi_k).
    \end{align}
\end{lemma}

\begin{proof}
\leavevmode
    Let $d^c:=i(\partial-\bar\partial)$ be the \textit{complex differential}, so that $(d^c \varphi )(X)=(d\varphi)(J X)$.
    One then sees that $d^c\varphi(\partial_{x_l})=0$ by the $T^2$-invariance of $\varphi$ and $d^c\varphi(\partial_{\theta_l})=-\sum_{k=1}^2 u^{jk} \varphi_k$ by the definition of $J$ in \cref{newcomplex}.
    Thus, 
    \begin{align*}
        d^c \varphi=-\sum_j \left(\sum_k u^{jk} \varphi_k \right) d\theta_j.
    \end{align*}
    Applying $d$ to both sides and using $dd^c=-2i \partial \bar \partial$ gives \cref{ddb}. \qedhere
\end{proof}
This formula gives us a convenient way to generate the symplectic form $\omega_P$ from a scalar potential function. 
\begin{corollary}
    The canonical symplectic form $\omega_P=dx_1\wedge d\theta_1+dx_2\wedge d\theta_2$ can be expressed on $M^o_P$ as 
    \begin{align}
        \label{equation:omegaP-potential}
        \omega_P=2i \partial \overline{\partial} \left( u_1x_1+u_2x_2-u\right).
    \end{align}
\end{corollary}
\begin{proof}
    
    Letting $\varphi=u_1x_1+u_2x_2-u$, we see that $\varphi_k=\sum_{i=1}^{2}u_{ik}x_i$, so that 
    \begin{align*}
        A_{jl}=\sum_{i,k=1}^{2}\partial_{x_l} (u^{jk}u_{ik}x_i)=\sum_{i=1}^{2}\partial_{x_l}(\delta_i^j x_i)=\delta_i^j \delta_i^l=\delta_{jl},
        \end{align*} 
        as required. \qedhere
\end{proof}
The formula for $i\partial\overline{\partial}$ also gives us a convenient way to generate the Ricci curvature from a scalar potential. Indeed, in K\"ahler geometry, it is common to express the Ricci curvature as the \textit{Ricci form}, which is the $2$-form defined by $\rho(X,Y)=\text{Ric}(JX,Y)$. The following lemma demonstrates that this $2$-form describing the Ricci curvature also appears as $i\partial \overline{\partial}$ applied to a scalar function (cf. \cite[Eqn. 4.64]{Ballmann2006}). 
\begin{lemma}
   \label{lemma:ricci-potential}
   The Ricci form $\rho$ on $M^o_P$ appears as 
\begin{align}
\label{equation:ricci-potential-equation}
    \rho =i\partial \overline{\partial}\ln (U),
\end{align}
where $U=u_{11}u_{22}-u_{12}u_{21}$.
\end{lemma}
\begin{proof}
       Let $\varphi=\ln U$.
    We have
    \begin{align*}
        \rho(\partial_{\theta_j}, \partial _{x_l})
        &=
        -u^{jk} \Ric(\partial _{x_k}, \partial_{x_l})
        =
        -u^{jk} \cdot \frac{1}{2}
        \left(
        \partial_{x_k} \partial_{x_l} 
        -
        u^{mn} 
        \frac{\partial u_{kl}}{\partial_{x_m}} \partial_{x_n}
        \right)
        \varphi
        \\
        &=
        -\frac{1}{2}
        u^{jk} \varphi_{kl}
        +
        \frac{1}{2}
        u^{jk}u^{mn}u_{klm} \varphi_n
        =
        -\frac{1}{2}
        u^{jk} \varphi_{kl}
        -
        \frac{1}{2}
        (\partial_{x_l} u^{jk}) \varphi_k
        =
        -\frac{1}{2}
        \partial_{x_l} (u^{jk} \varphi_k),
    \end{align*}
    where in the first step we used the definition of $\rho$ and the equation for $J$ from \cref{newcomplex};
    in the second step we used \cite[Eqn. 3.30]{Wiseman};
    in the fourth step we used $u_{klm}=u_{kml}$ because derivatives commute and the formula for the derivative of an inverse matrix $\partial_{x_l} u^{jn}=-u^{jk} u_{kml} u^{mn}$.
    
    Comparing this with the formula for $i \partial \bar \partial \varphi$ from \cref{ddb} proves the claimed equation \cref{equation:ricci-potential-equation} evaluated on the vectors $(\partial_{\theta_j}, \partial _{x_l})$.
    The claim then follows because the non-mixed entries of $\rho$ and $-i \partial \bar \partial \varphi$ are zero by \cref{toricKahlerRicci} and \cref{ddb} respectively.
\end{proof}

It follows from the $dd^c$-lemma that there exists some function $F$ such that
$\Ric(\omega_P)=\omega_P+i \partial \bar \partial F$, see \cite[Eqn. 5.1]{Joyce}.
In the toric case, one can describe $F$ explicitly. 

\begin{lemma}\label{KEMA}
  The function $F:M_P\to \mathbb{R}$ given by $F=\ln(U)-2(u_1x_1+u_2x_2-u)$ is smooth, and satisfies \[
\rho
=
\omega_P
+
i \partial \bar \partial F.
\] 
\end{lemma}
\begin{proof}
On $M_P^o$, we compute
    \[
    \omega_P+i \partial \bar \partial F
    =
    2i \partial \bar \partial(u_1x_1+u_2x_2-u)+i \partial \bar \partial (\ln U-2(u_1x_1+u_2x_2-u))
    =
    i\partial \bar \partial \ln U
    =
    \rho,
    \]
    where we used \cref{equation:omegaP-potential} in the first step and use \cref{equation:ricci-potential-equation} in the third step.

    It remains to check that $F$ is smooth on $M_P$.
    By \cref{smoothnessthm}, the smooth function $u:P\to \mathbb{R}$ from \cref{perturbedpotential} satisfies \cref{equation:delta-smoothness} for some smooth and positive function $\delta:P\to \mathbb{R}$. 
    Furthermore,  \cref{canonicalpotential} gives 
    \begin{align}
    \label{equation:uP-grad}
    2u_P-2x \cdot \nabla u_P=\sum_{|r|=1,2,3} \log l_r-l_r+1.
    \end{align}
    Combining gives 
    \begin{align*}
        F
        &=
        \ln \det \Hess u - 2(x \cdot \nabla u-u)
        \\
        &=
        \underbrace{-\ln \delta - \sum_{|r|=1,2,3}
        \ln l_r}
        +
        \underbrace{
        \sum_{|r|=1,2,3}
        (\ln l_r-l_r+1)
        +2h-2x \cdot \nabla h}
        \\
        &=
        -\ln \delta 
        +
        \sum_{|r|=1,2,3}
        (-l_r+1)
        +2h-2x \cdot \nabla h,
    \end{align*}
    where in the first step we used \cref{equation:delta-smoothness} and \cref{equation:uP-grad}.
    All terms on the right hand side are smooth on $P$ up to the boundary, hence smooth on $M_P$.
\end{proof}

\subsection{The Einstein condition for toric K\"ahler metrics}
\label{subsection:einstein-condition-for-toric-kaehler}

The goal of this paper is to produce and study a K\"ahler-Einstein metric on $\mathbb{CP}^2\# 3\overline{\mathbb{CP}^2}$ in the cohomology class of $\omega_P$. It is known by \cite{BandoMabuchi} and \cite{TianYau} that there is precisely one such K\"ahler metric $g$ (up to scaling and complex automorphism) and that it is toric. Thus, a sensible approach would be to search for a function $u:P\to \mathbb{R}$ (as in \cref{perturbedpotential}) for which the scalar function $F$ of \cref{KEMA} vanishes uniformly. 

Abreu \cite{Abreu} suggests that it should be possible to find this function \textit{explicitly}; while we were not able to accomplish this task, the approach of this paper is to explicitly (with computational assistance) find a function $u:P\to \mathbb{R}$ for which $F$ is \textit{almost} zero. Specifically, we are able to find $u$ for which $||F||_{H^5(M_P,g)}$ is small. In this situation, if $\phi:M_P\to \mathbb{R}$ is smooth and toric, then the Kähler-Einstein condition for the perturbed symplectic form $\omega=\omega_P+i\partial\overline{\partial}\phi$ is equivalent to $(\omega_P +i \partial \bar \partial \phi)^2=e^{F-\phi} \omega_P^2$ (see \cite[p.51]{Szekelyhidi2014} for more details). 
This condition is called a \emph{complex Monge–Ampère equation} for the unknown scalar function $\phi$, and is often written as 
\begin{align}
   \label{equation:cx-Monge-Ampere}
  \mathcal{F}(\phi)= \frac{(\omega_P +i \partial \bar \partial \phi)^2}{\omega_P^2}-e^{F-\phi}=0.
\end{align}
If $F$ is small enough, then this equation can be studied using classical perturbation techniques. 
\begin{lemma}\label{Taylordecomp}
   Using a Taylor series expansion, \cref{equation:cx-Monge-Ampere} is equivalent to
\[
0
=
\mathscr{E}+\mathscr{L}\phi+\mathscr{N}\phi,
\]
where
\begin{align*}
  \mathscr{E}
    &
    =
    1-e^F,
    \\
    \mathscr{L}\phi
    &
    =
    -\frac{1}{2} \Delta \phi + e^F \cdot \phi
    \text{ is a linear differential operator},
    \\
    \mathscr{N} \phi
    &
    =
    (i \partial \bar \partial \phi)^2/\omega_P^2
    -
    e^F \cdot (e^{-\phi}-1+\phi)
    =
    \frac12\star (i\partial\bar\partial \phi)^2
    -e^F(e^{-\phi}-1+\phi)
    \text{ are the higher order terms.}
\end{align*}
\end{lemma}
\begin{proof}
Clearly $\mathcal{F}(0)=1-e^{F}$. The linearisation of $\mathcal{F}$ at $0$ in the direction $\phi$ is given by
\begin{align}
\mathscr{L}(\phi)
&=
2\,\frac{\omega_P\wedge (i\partial\bar\partial\phi)}{\omega_P^2}
\;+\; e^F\,\phi
=
-\frac{1}{2}\,\Delta\phi+e^F\,\phi,
\end{align}
where in the last step we used
\[
2\omega_P\wedge (i \partial \overline{\partial}\phi)=
-\left(\frac{1}{2}\Delta \phi \right)\omega_P^2,
\]
which follows from \cref{toricLaplacian} and \cref{deldelbar}, combined with the definition $\omega_P=dx_1\wedge d\theta_1+dx_2\wedge d\theta_2$. 

The equality $\mathscr{N} \phi = (i \partial \bar \partial \phi)^2/\omega_P^2 - e^F \cdot (e^{-\phi}-1+\phi)$ follows from  $\mathscr{N} \phi=\mathcal{F}(\phi)-\mathscr{E}-\mathscr{L}\phi$.
The identity $(i \partial \bar \partial \phi)^2/\omega_P^2=\frac12\star (i\partial\bar\partial \phi)^2$ follows from $\star \omega_P^2=\star (2! dV_g)=2!$, which is \cite[Eqn. 1.45, Eqn. 4.20]{Ballmann2006}.
\end{proof}

Thus, if $\mathscr{E}$ is sufficiently small, $\mathscr{L}$ is invertible, and the norm of $\mathscr{L}^{-1}$ is not too large, then a standard perturbation argument reveals that there is a toric K\"ahler-Einstein metric close to $(M_P,\omega_P,J)$. Since $(M_P,\omega_P,J)$ is known explicitly, we are therefore in a position to make reliable geometric conclusions about the true Einstein metric.

\section{Approximation}\label{approx}

In this section, we use numerical methods to construct an approximate toric K\"ahler Einstein metric $g$ on $\mathbb{CP}^2\# 3\overline{\mathbb{CP}^2}$, then quantify the extent to which $g$ actually satisfies the Einstein condition, and describe some estimates on the first eigenvalue of the associated Laplace-Beltrami operator. 

\subsection{Numerical approach}
\label{subsection:numerical-approach}

We see from \cref{KEMA} that the toric K\"ahler-Einstein metric on $\mathbb{CP}_2\# 3 \overline{\mathbb{CP}_2}$ is given by solving the second order differential equation $F=\ln(U)-2(u_1x_1+u_2x_2-u) = c$ for the symplectic potential $u$ over the polytope $P$ where $c$ is any constant. Our first task is to find a very good approximation to this solution, and then our second is to write this approximation in a convenient form amenable to the analysis of its geometric properties. For the first step we solve the differential equation using pseudospectral numerical methods. For the second, we rationalize this numerical solution, and write it in a Chebychev polynomial basis.

Writing the potential $u$ as in~\eqref{perturbedpotential} as a canonical part $u_P$ and unknown part $h$, then given the canonical potential in~\eqref{canonicalpotential}, this gives the following Monge–Ampère differential equation for the unknown $h$,
\begin{align}\label{MongeAmpereSymplectic}
F = \log\left(  c_1 \left( ( \partial_{x_1}^2 h  )(  \partial_{x_2}^2 h ) -  ( \partial_{x_1} \partial_{x_2} h )^2 \right) + c_2 \partial_{x_1}^2 h + c_3 \partial_{x_2}^2 h + c_4 \partial_{x_1} \partial_{x_2} h + c_5 \right) - 2 \left( x_1 \partial_{x_1} h + x_2 \partial_{x_2} h  - h \right) = c
\end{align}
where,
\begin{align*}
& c_1 = (1 - x_1^2) (1 - x_2^2) ( 1 - (x_1+x_2)^2 ) \; , \quad 
c_2 = (1 - x_1^2) ( 2 - x_1^2 - 2 x_2^2 - 2 x_1 x_2 ) \; , \quad 
c_3 = (1 - x_2^2) ( 2 - 2 x_1^2 - x_2^2 - 2 x_1 x_2 ) \\ \nonumber
& c_4 = - 2 (1 - x_1^2) (1 - x_2^2) \; , \quad 
c_5 = 3 - 2 (x_1^2 + x_2^2 + x_1 x_2 ) \; .
\end{align*}
Different values of $c$ give the same geometry, simply translating into a shift of the potential $h$ by a constant. 

The equation $F = 0$ for $u$ is invariant under the $D_6$ hexagonal symmetry that leaves the polytope $P$ invariant, and since our canonical potential is also invariant, our equation above for $h$ shares this symmetry. Explicitly, the group $D_6$ is generated by the $\mathbb{Z}_2$ reflection symmetry and $\mathbb{Z}_6$ rotation symmetry 
\begin{align*}
    R_1=\begin{pmatrix}0&1\\
    1&0
    \end{pmatrix}, \qquad R_2=\begin{pmatrix}
        1&1\\
        -1&0
    \end{pmatrix}
\end{align*}
acting on $(x_1, x_2)$. 
By \cite[Theorem C]{BandoMabuchi} there exists a $D_6$-invariant solution $u$ to \cref{MongeAmpereSymplectic}, hence also $D_6$-invariant $h$.
While this motivates computing a $D_6$-invariant approximate solution, this would not be strictly necessary.
The subsequent analysis would work identically with a non-invariant solution, except for when we use symmetry to speed up eigenvalue computations on the hexagon.
Regardless, \cref{maintheorem} proves that our solution is $D_6$-invariant without appealing to \cite[Theorem C]{BandoMabuchi}. 

One approach to pursing an approximate solution is to observe that the Taylor expansion of $h$ at the origin is composed of a sum of $D_6$ invariant  polynomials as a consequence of the symmetry (these are generated by $I_2=x_1^2+x_1x_2+x_2^2$ and $I_6=x_1^2x_2^2(x_1+x_2)^2$.)
In~\cite{Wiseman} one method to find numerical approximations was to solve \cref{MongeAmpereSymplectic}   in a finite basis of these invariant polynomials. However, as monomials are a poor basis to approximate functions on the interval, it is unclear that such an approach will yield good behaviour for high order approximation.

The approach we take here, as also previously used in~\cite{Wiseman}, is to solve the equation using real space differencing. The symmetry invariance of $h$ implies that we need only consider the equation on a fundamental domain of the hexagon, and from this we may reconstruct the full solution using the $D_6$ transformations.
In fact we will find it convenient to work with the fundamental domain given by,\
\begin{align*}
P_{comp} = \{ (x_1 , x_2) | 0 \le x_1 \le 1\, , \; -1 \le x_2 \le 0 \, , \; x_1 \le | x_2 |  \} \; .
\end{align*}

We then require conditions at the boundaries of the domain, $x_1 = 0$ and $x_2 = - x_1$, to ensure the function $h$ extends to an appropriately differentiable function on the full polytope $P$. 
Noting that $R_1R_2=\begin{pmatrix}
    -1&0\\
    1&1
\end{pmatrix}$ 
is a reflection across the $x=0$ line (the $y$-axis) with normal $(-2,1)$, and $R_1 (R_3)^3=\begin{pmatrix}
    0&-1\\
    -1&0
\end{pmatrix}$ is a reflection across $y=-x$ with normal $(1,1)$, then a solution $h$ must obey the following conditions if it is to be $D_6$ invariant;
\begin{itemize}
    \item at $x_1=0$ then $(-2\partial_{x_1}+\partial_{x_2})^m h=0$ for odd $m$;
    \item at $x_2=-x_1$, the condition is that $(\partial_{x_1}+\partial_{x_2})^m h=0$ for odd $m$.
\end{itemize}
Hence we may solve for $h$ on $P_{comp}$ subject to the above conditions at the boundaries of this domain.

In order to achieve high accuracy we employ pseudospectral differencing. For smooth functions, such as we expect to find for $h$, this should yield exponential convergence as the order of approximation is increased, and we indeed see this. We discretise the interval $[0,1]$ using a Chebychev grid $\{ y_a = \frac{1}{2} \left( 1- \cos\left( \frac{ \pi n}{N-1} \right) \right) | n = 0, 1, \ldots N-1 \}$ with $N$ points, and 
take a product of these to cover the unit square $[0,1] \times [0,1]$. We then represent the function $h$ as the discrete set of values, $h_{disc} = \{ h_{a,b} | a,b \in 0, 1, \ldots N-1 \; , \quad a \le b \}$, which map to our computational domain $P_{comp}$ in the obvious manner, $h_{a,b} =  h(y_a, - y_b)$ for $a \le b$. We fill the remainder of the unit square in using the $D_6$ symmetry $R_1 (R_3)^3$, so $h_{a,b} =  h(y_b, - y_a)$, and then use the usual Chebychev differencing matrices to evaluate derivatives of $h$, and hence compute  the value of the Monge-Ampere equation $F$ at the grid points $F_{a,b}$.

We must impose the $D_6$ invariance. We have already imposed $R_1 (R_3)^3$ invariance in filling the grid, but have yet to require full invariance. We do this by imposing the leading first order derivative condition for $R_1R_2$,  so $Dh = (-2\partial_{x_1}+\partial_{x_2}) h = 0$ at grid points where $x_1 = 0$, rather than the Monge–Ampère equation at those points.
Since the differential equation is second order it is natural to impose this Robin-like boundary condition, and then hope that the solution found respects the full invariance (i.e.  the conditions for $R_1 R_2$ above for $m \ge 3$) to good accuracy. 

We solve the system of equations $\{ F_{a,b} = 0  | a \in  1, \ldots N-1 \; ,\quad b \in 0,  \ldots N-1 \; , \quad a \le b \}$ together with the derivative conditions $\{ (Dh)_{0,b} = 0 |  b \in 0,  \ldots N-1 \}$ at the domain boundary for the data $h_{disc}$ noting that we have the same number of equations as unknowns. 
We leave $c$ free and fix the value of $h$ at one collocation point, thereby choosing the additive normalization of $h$.
In practice we choose to fix the value of $h$ at $x_{1,1}$, and replace the equation $F_{1,1}$ there by the condition that $h$ remains equal to the value of its initial guess there. This then yields a system that may be solved straightforwardly using the Newton method.
As an example, we quickly converge to an approximate solution taking as initial data the D6 invariant starting guess,
\begin{align*}
h_{guess} = - \frac{1}{4} \left( x_1^2 + x_1 x_2 + x_2^2\right) \; .
\end{align*}
We use Mathematica to implement the Newton method utilizing its higher precision arithmetic. With $~140$ digits of precision and a Chebychev grid with $N = 90$ we obtain a solution after $10$ Newton steps with a residual of $< 10^{-100}$ after a few hours of computation on a laptop.

Having found this numerical approximation we then rewrite the solution in a order $(N-1) \times (N-1)$ two-dimensional Chebychev polynomial basis suitable for approximating functions on the unit square $[0,1]^2$. 
For a polynomial approximation on the interval $[0,1]$, an optimal basis of orthogonal polynomials is given by $\{ T_n( 2 x - 1 ) \}$, with $T_n$ a Chebyshev polynomial of the first kind.
Hence we solve for the coefficients $\{ c_{m,n}  | m,n \in 0, 1, \ldots, N-1 \}$ such that, $h_{ab} = \sum_{m,n = 0}^{N-1} c_{m,n} T_m(2 x_a - 1) T_n(1 - 2 x_b)$ for all $(x_a, x_b)$ in the unit square. The transformation is conveniently implemented using discrete cosine transforms with very high precision arithmetic ($1000$ digits precision).
We note that the $D_6$-symmetry $R_1 (R_3)^3$ implies $h_{ab} = h_{ba}$, and hence $c_{m, n} = (-1)^{n+m} c_{n, m}$. 
Then our approximate potential over the computational domain $P_{comp}$ is,
\begin{align}
\label{eq:approx}
h_{\mathrm{approx}}(x_1, x_2)  = \sum_{m,n = 0}^{N-1} c_{m,n} T_m(2 x_1 - 1) T_n(2 x_2 + 1) \; .
\end{align}
As a polynomial expression this is smooth over $P_{comp}$. 
However we must be careful that it extends sufficiently smoothly to a  function over the entire polytope $P$ under the action of the D6 symmetry. 
We do not require full $C^\infty$-smoothness, but we choose to guarantee $C^5$ so that five derivatives are continuous -- since we later will use $H^5$ norms in various estimates. We have only imposed the condition $(-2\partial_{x_1}+\partial_{x_2})^m h=0$ at $x_1 = 0$ for $m = 1$ which should hold to the numerical precision used. However it does not hold exactly, and we require this for our approximation. Furthermore 
the conditions with $m = 3,5$ would not be expected to hold even to numerical precision. In fact in practice they hold to good accuracy in our approximate solution, but we require all these boundary conditions hold exactly. Hence we must correct our approximation to impose these exactly.

To do this we rationalize the coefficients $c_{m,n}$ in our approximation. We impose $c_{m, n} = (-1)^{n+m} c_{n, m}$ exactly and then compute the conditions $B_m = \left. (-2\partial_{x_1}+\partial_{x_2})^m h \right|_{x_1 = 0} =0$ for $m = 1,3,5$. The $B_m$ are polynomials of degree at most $N-1$ in $x_2$
with rational coefficients, so
$B_m=\sum_{n=0}^{N-1}b_{m,n}x_2^n$.
We impose $b_{m,n}=0$ for $m=1,3,5$ and $0\le n\le N-1$,
solving these linear conditions in exact rational arithmetic
for a subset of the coefficients $c_{m,n}$ while preserving
the coefficient symmetry. 
In practice we may solve them for $\{ c_{m, N-n} | m \le N-n \, , \; n = 1,2,3 \}$. Let us call this new set of rational valued coefficients  with these values updated $c'_{m,n}$,  and the associated approximation $h_{\mathrm{approx}}'$ takes the form above in~\eqref{eq:approx} except with $c \to c'$.
The values $c'_{m,n}$ numerically differ slightly from the $c_{m,n}$ extracted from the original numerical solution, but the change is small since while these conditions were not exact, they were satisfied to high accuracy.

Now the function $h_{\mathrm{approx}}'$ on $P_{comp}$  is a polynomial with rational coefficients, and extends to a $C^5$ function on the full polytope $P$, giving a very good approximation to the K\"ahler-Einstein metric. 
We will quantify this in considerably more detail in what follows but let us gain some intuition for the time being. 
Taking $N = 90$ and evaluating one derivative of the Monge–Ampère condition, $\partial_{x_1} F$, we find the non-rigorous $C^0$-bound of $\sim 10^{-36}$ over $P$. 
For each further derivative of the condition we lose several orders of magnitude such that third partial derivatives are bounded by $\sim 10^{-29}$.

\subsection{A posteriori geometric bounds}
\label{subsection:a-posteriori-geometric-bounds}
The rational Chebyshev approximation $h=h_{\mathrm{approx}}'$ constructed in the previous section determines the approximate symplectic potential
\[
    u=u_P+h
\]
and hence the approximate toric K\"ahler--Einstein metric of the form \eqref{toricKahlermetric}.
We require two types of estimates. First, the Sobolev estimates below (\cref{subsection:a-posteriori-error}) use componentwise bounds for $u^{ij}$ and its first two coordinate
derivatives. Second, the elliptic estimates of (\cref{GAE}) require invariant $C^0$-bounds for the Ricci and Riemann tensors and for the first two covariant derivatives of the Riemann tensor.

Although the entries $u_{ij}$ are singular at the boundary of the moment polytope, the entries of the inverse Hessian $u^{ij}$ extend smoothly to the closed polytope $P$. To obtain reasonable bounds near $\partial P$, it is important to perform the cancellation of the singular terms algebraically before applying interval arithmetic.
 
As described in \cref{prelim}, the entries of $u_{ij}$ diverge as one approaches a facet $l_a=0$. The entries of $u^{ij}$, however, remain smooth because the singular terms cancel when the matrix is inverted. Directly inverting $u_{ij}$ on a computer is poorly suited to rigorous estimation. The interval arithmetic treats the singular terms separately
and generally fails to recover these cancellations near the boundary. This difficulty becomes more severe as higher derivatives of the inverse matrix are taken. We therefore write $u^{ij}$ in the form  \eqref{inverse_metric} and cancel the $1/l_r$ terms explicitly.
Define
\[
D_0=U\prod_{|r|=1,2,3}l_r=\delta^{-1},
\qquad
(A^{ij})=
\left(\prod_{|r|=1,2,3}l_r\right)
\begin{pmatrix}
u_{22}&-u_{12}\\
-u_{21}&u_{11}
\end{pmatrix}.
\]
Then $u^{ij}=A^{ij}/D_0$, and $D_0$ and $A^{ij}$
are polynomials on the computational domain.
For higher derivatives of $u^{ij}$, we use this factorisation and apply quotient rule explicitly, i.e., $
    \partial^\alpha u^{ij}
    =
    \frac{N^{ij}_\alpha}{D_0^{|\alpha|+1}},
$
where the numerators $N^{ij}_\alpha$ are computed recursively. Constructing the numerator recursively avoids the
large intermediate expressions produced by repeatedly
differentiating a fully expanded quotient. This is especially
important for the third and fourth coordinate derivatives of
$u^{ij}$, which occur in the estimates for
$\nabla\Riem$ and $\nabla^2\Riem$, respectively.

\medskip

We briefly describe how the resulting rational functions are bounded.
The functions $D_0$ and $N^{ij}_\alpha$ are represented
in the tensor-product Chebyshev basis. Their coefficients are obtained
from the exact rational coefficients of $h$, so the initial interval
representations contain no error from floating-point conversion.
Products are truncated at a fixed Chebyshev degree, with the discarded
coefficient tails included in the resulting interval enclosures. The
computational square is then divided into smaller boxes. On each box,
directed-rounding interval arithmetic is used to certify $D_0>0$ and
to enclose the quotient $\frac{N^{ij}_\alpha}{D_0^{|\alpha|+1}}$.
Taking the largest upper endpoint over all boxes gives a global bound. The $D_6$-symmetry then extends the bounds from the computational region to the whole polytope $P$. The certified componentwise estimates for $u^{ij}$ are recorded in \cref{Data}.

The same factorisation is used to estimate the curvature. The formulae
in Appendix \ref{CDE} express $\Riem$, $\nabla\Riem$, and
$\nabla^2\Riem$ in terms of coordinate derivatives of $u^{ij}$ of
orders two, three, and four, respectively. For each curvature
quantity, the numerator of its squared pointwise norm is first
constructed in Chebyshev coefficient space. The resulting rational
function is then bounded on each subdivision box using the same
interval procedure as above. The certified curvature bounds are recorded in Appendix \cref{Data}.

For the inverse-Hessian estimates, we use
a Chebyshev truncation of degree $19$ and an $8\times8$
subdivision. The $\|{\Ric(g)-g}_{C^0_g}$ estimate uses degree $19$ and a
$30\times30$ subdivision. The direct Ricci estimate and the
estimates for $\Riem$, $\nabla\Riem$, and
$\nabla^2\Riem$ use degree $19$ and a $8\times8$
subdivision. These computations were implemented in Julia using the
\texttt{IntervalArithmetic.jl} package, with
\texttt{Interval\{BigFloat\}} endpoints evaluated at $256$-bit
precision. The input Chebyshev coefficients are read from their exact
rational representations and enclosed at the active precision using
directed rounding.

\subsection{A posteriori error}
\label{subsection:a-posteriori-error}

In this section, we will compute an explicit $H^3$-bound for the residual $\mathscr{E}$ from \cref{Taylordecomp}.
Recall that from \cref{KEMA} we have the explicit equation $F=\ln(U)-2(u_1x_1+u_2x_2-u)$ for the Ricci potential.
Thus, bounding $\|{\mathscr{E}}_{H^3}$ reduces to rigorously computing a bound for a known function.
Throughout, we write $\mathscr{E}=1-e^{F}=1-Ge^H$ for $G=D_0$ and $H=-2(h_1x_1+h_2x_2-h)$.
Although $\mathscr{E}$ is not a polynomial, $G$ and $H$
are polynomials on the computational domain.

\textbf{Bounds for Chebyshev polynomials.}
From \cref{eq:approx}, the approximate solution $h:P \rightarrow \R$ is given as a sum of Chebyshev polynomials, i.e. $h_{\mathrm{approx}}'=\sum_{m,n = 0}^{N-1} c_{m,n}' T_m(2 x_1 - 1) T_n(2 x_2 + 1)$.
Because $\|{T_i}_{C^0([0,1])}=1$ for $i \geq 0$, we have that
\begin{align}
    \label{eqn:chebyshev-bound}
    \|{h}_{C^0} \leq \sum_{m,n=0}^{N-1} |c_{m,n}'|.
\end{align}
The basic strategy is then to expand $\mathscr{E}$ in Chebyshev polynomials and apply that same bound.
This is not literally possible, because $\mathscr{E}$ contains an exponential, so is not a polynomial, so we will apply a short workaround.

\textbf{Euclidean derivatives of $\mathscr{E}$.}
We have $\partial_{x_k} \mathscr{E}=-\partial_{x_k} G e^H = -(\partial_{x_k}G+G\partial_{x_k}H)e^H$ for $k \in \{1,2\}$.
Here, $(\partial_{x_k}G+G\partial_{x_k}H)$ and $H$ are polynomials, so we can expand them in the Chebyshev basis and, using \cref{eqn:chebyshev-bound}, compute upper bounds $\|{\partial_{x_k}G+G\partial_{x_k}H}_{C^0}$ and $\|{H}_{C^0}$.
We obtain the upper bound
\[
\|{\partial_{x_k} \mathscr{E}}_{C^0}
\leq
\|{\partial_{x_k}G+G\partial_{x_k}H}_{C^0}
e^{\|{H}_{C^0}}.
\]
Using this, we can compute $C^0$-bounds for all partial derivatives of $\mathscr{E}$, but not for $\mathscr{E}$ itself.
The mean value theorem gives a $C^0$-bound for $\mathscr{E}$:
\begin{align*}
\|{\mathscr{E}}_{C^0}
&\leq
|\mathscr{E}(1/2,1/2)|+d((0,0),(1/2,1/2)) \cdot \|{|d \mathscr{E}|_{g_{\text{Eucl}}}}_{C^0}
\\
&\leq
|\mathscr{E}(1/2,1/2)|+\frac{1}{\sqrt{2}} \cdot \sqrt{\|{\partial_{x_1} \mathscr{E}}_{C^0}^2+\|{\partial_{x_2} \mathscr{E}}_{C^0}^2}.
\end{align*}

\textbf{$L^2$-bounds.}
So far, we computed $C^0$-bounds for the Euclidean derivatives of $\mathscr{E}$.
For our desired $H^3$-bound, we need to convert these into integral bounds, measuring vectors and the volume element in the approximate metric induced by $u$.
One estimate we use for this are uniform \emph{upper} bounds for $u^{ij}$, $\partial_k u^{ij}$, $\partial_l \partial_k u^{ij}$ on all of $[0,1]^2$.
Note that some eigenvalues of $u^{ij}$ go to zero at the boundary of $\partial P$, so no uniform \emph{lower} bound for these quantities hold.
This is a major difficulty occurring in \cref{subsection:non-sharp-eigenvalue-bound}, but not here.
The following proposition reduces the computation of the $H^2$-bound to Euclidean derivatives of $\mathscr{E}$ and $u^{ij}$:

\begin{proposition}
    We have
    \begin{align*}
        \|{\mathscr{E}}_{L^2}
        &\leq
        \sqrt{12}\,\pi\,\|{\mathscr{E}}_{C^0},
        \\
        \|{\nabla_g\mathscr{E}}_{L^2}
        &\leq
        \sqrt{12}\,\pi
        \left(
            \sum_{i,j=1}^2
            \|{u^{ij}}_{C^0}
            \|{\partial_{x_i}\mathscr{E}}_{C^0}
            \|{\partial_{x_j}\mathscr{E}}_{C^0}
        \right)^{1/2},
        \\
        \|{\nabla_g^2\mathscr{E}}_{L^2}
        &\leq
        \left(
            12\pi^2\|{\Delta_g\mathscr{E}}_{C^0}^2
            +
            K_1\|{\nabla_g\mathscr{E}}_{L^2}^2
        \right)^{1/2},
        \\
        \|{\nabla_g^3\mathscr{E}}_{L^2}
        &\leq
        \left(
            24\pi^2
            \sum_{k,l=1}^2
            \|{u^{kl}}_{C^0}
            \|{\partial_{x_k}\Delta_g\mathscr{E}}_{C^0}
            \|{\partial_{x_l}\Delta_g\mathscr{E}}_{C^0}
            +
            \Cr{D^3u-estimate-second-summand}
            \|{\mathscr{E}}_{L^2_2}^2
        \right)^{1/2}.
    \end{align*}
    Moreover,
    \[
        \|{\Delta_g\mathscr{E}}_{C^0}
        \leq
        \sum_{i,j=1}^2
        \left(
            \|{\partial_{x_i}u^{ij}}_{C^0}
            \|{\partial_{x_j}\mathscr{E}}_{C^0}
            +
            \|{u^{ij}}_{C^0}
            \|{\partial_{x_i}\partial_{x_j}\mathscr{E}}_{C^0}
        \right),
    \]
    and, for $k=1,2$,
    \begin{align*}
        \|{\partial_{x_k}\Delta_g\mathscr{E}}_{C^0}
        \leq
        \sum_{i,j=1}^2
        \bigl(
        &\|{\partial_{x_k}\partial_{x_i}u^{ij}}_{C^0}
         \|{\partial_{x_j}\mathscr{E}}_{C^0}
        +
        \|{\partial_{x_i}u^{ij}}_{C^0}
         \|{\partial_{x_k}\partial_{x_j}\mathscr{E}}_{C^0}
        \\
        &+
        \|{\partial_{x_k}u^{ij}}_{C^0}
         \|{\partial_{x_i}\partial_{x_j}\mathscr{E}}_{C^0}
        +
        \|{u^{ij}}_{C^0}
         \|{\partial_{x_k}\partial_{x_i}\partial_{x_j}
         \mathscr{E}}_{C^0}
        \bigr).
    \end{align*}
\end{proposition}

\begin{proof}
    The first estimate is immediate from $dV_g
        =
        dx_1\,dx_2\,d\theta_1\,d\theta_2$,
    the fact that the determinant of the metric is $1$, and
    $\vol_g(M)=12\pi^2$.

    Because $\mathscr{E}$ is $T^2$-invariant, we have
    \[
        |\nabla_g\mathscr{E}|_g^2
        =
        \sum_{i,j=1}^2
        u^{ij}
        \partial_{x_i}\mathscr{E}
        \partial_{x_j}\mathscr{E}.
    \]
    Taking coordinatewise absolute upper bounds for $u^{ij}$ and
    $\partial_{x_i}\mathscr{E}$ and integrating this gives the second estimate.

    By \cref{proposition:estimates-from-commutator-formula},
    \[
        \|{\nabla_g^2\mathscr{E}}_{L^2}^2
        \leq
        \|{\Delta_g\mathscr{E}}_{L^2}^2
        +
        K_1\|{\nabla_g\mathscr{E}}_{L^2}^2.
    \]
    Since $\Delta_g\mathscr{E}
        =
        -(\partial_{x_i}u^{ij})\partial_{x_j}\mathscr{E}
        +
        u^{ij}\partial_{x_i}\partial_{x_j}\mathscr{E}$,
    taking coordinatewise absolute upper bounds and using
    $\vol_g(M)=12\pi^2$ gives the third estimate.

        Again by
    \cref{proposition:estimates-from-commutator-formula},
    \[
        \|{\nabla_g^3\mathscr{E}}_{L^2}^2
        \leq
        2\|{\nabla_g\Delta_g\mathscr{E}}_{L^2}^2
        +
        \Cr{D^3u-estimate-second-summand}
        \|{\mathscr{E}}_{L^2_2}^2.
    \]
    Since $\Delta_g\mathscr{E}$ is $T^2$-invariant, $|\nabla_g\Delta_g\mathscr{E}|_g^2
        =
        \sum_{k,l=1}^2
        u^{kl}
        \partial_{x_k}\Delta_g\mathscr{E}
        \partial_{x_l}\Delta_g\mathscr{E}$.
    Taking coordinatewise absolute upper bounds and integrating gives
    \[
        \|{\nabla_g\Delta_g\mathscr{E}}_{L^2}^2
        \leq
        12\pi^2
        \sum_{k,l=1}^2
        \|{u^{kl}}_{C^0}
        \|{\partial_{x_k}\Delta_g\mathscr{E}}_{C^0}
        \|{\partial_{x_l}\Delta_g\mathscr{E}}_{C^0}.
    \]
    Moreover, differentiating the formula for $\Delta_g\mathscr{E}$ gives
    \begin{align*}
       - \partial_{x_k}\Delta_g\mathscr{E}
        ={}&
        (\partial_{x_k}\partial_{x_i}u^{ij})
        \partial_{x_j}\mathscr{E}
        +
        (\partial_{x_i}u^{ij})
        \partial_{x_k}\partial_{x_j}\mathscr{E}
        \\
        &+
        (\partial_{x_k}u^{ij})
        \partial_{x_i}\partial_{x_j}\mathscr{E}
        +
        u^{ij}
        \partial_{x_k}\partial_{x_i}\partial_{x_j}\mathscr{E}.
    \end{align*}
    Taking coordinatewise absolute upper bounds gives the stated estimate for
    $\|{\partial_{x_k}\Delta_g\mathscr{E}}_{C^0}$. Substituting these bounds gives the final estimate.
\end{proof}

Plugging in the numerical values of our particular approximate solution, we obtain:

\begin{corollary}\label{H3aposteriorierror}
    Let $u$ be the Kähler potential from \cref{subsection:numerical-approach}.
    Then $\|{\mathscr{E}}_{H^3}$ satisfies the bound from \cref{tab:numerical-constants}.
\end{corollary}

\subsection{A non-sharp first eigenvalue}
\label{subsection:non-sharp-eigenvalue-bound}

In this section we derive a lower bound for the smallest absolute value of an eigenvalue of the linearised operator $\mathscr{L}\phi=-\frac{1}{2} \Delta \phi + e^F \cdot \phi$ from \cref{subsection:einstein-condition-for-toric-kaehler} acting on $T^2 \rtimes D_6$-invariant functions (with the Laplacian taken with respect to the approximate Kähler-Einstein metric).
We achieve this by proving an explicit spectral gap for $\Delta$ that is bigger than $4$.
Because $F$ is close to zero, this implies a spectral gap for $\mathscr{L}$, say $|\lambda_{\min}(\mathscr{L})|$.
That is:
$|\lambda_{\min}(\mathscr{L})|$ is the smallest absolute value of any eigenvalue of $\mathscr{L}$.
The numerical values obtained from the result of this section are recorded in \cref{tab:numerical-constants}.

Throughout, we denote the smallest $T^2 \rtimes D_6$-invariant non-zero eigenvalue  of $\Delta$ by $\lambda_{\min}$.
We then have a Rayleigh-Ritz characterisation of $\lambda_{\min}$, namely
\begin{align}
\label{equation:rayleigh-ritz}
\lambda_{\min}
=
\inf_{f \in C^\infty_0(M)^{T^2 \rtimes D_6} \setminus \{0\}}
\frac{\int_{M} |\nabla f|_{g}^2dV_g}{\int_M f^2 dV_g}
=
\inf_{f\in C^\infty_0(P)^{D_6} \setminus \{0\}}
\frac{\int_{P} |\nabla f|_{g}^2 dx_1 dx_2}{\int_P f^2 dx_1 dx_2},
\end{align}
where the subscript $0$ in the function space indicates restriction to functions with mean zero (with respect to the appropriate measure). 
The second equality uses the $T^2$-invariance of the relevant metrics and functions to reduce the computation to integration over the hexagon $P$. In fact, if we let $Q=\operatorname{conv} (\{(0,0),(1,-1),(1,0)\})$ be a fundamental domain, then since our approximate K\"ahler-Einstein metric $g$ is $D_6$-invariant, we further obtain 
\begin{align}\label{Qformulation}
    \lambda_{\min}= \inf_{f\in C^\infty_0(P)^{D_6} \setminus \{0\}}
\frac{\int_{Q} |\nabla f|_{g}^2 dx_1 dx_2}{\int_Q f^2 dx_1 dx_2}.
\end{align}
Note that the only way in which the approximate metric $g$ enters is through $|\cdot|_{g}$. Indeed the volume form $dV_g$, as described in \cref{volumeform}, does not depend on the specific choice of $g$ coming from \cref{toricKahlermetric}, and consequently, the relevant volume form on $P$ and $Q$ is the ordinary Euclidean volume form. 

Bounding $\lambda_{\min}$ from below by $4$ is accomplished in three steps:
\begin{enumerate}
    \item 
    On an \emph{inset hexagon}, $P_\delta$, with $\delta>0$ describing closeness to $P$, compute a spectral gap for a finite element Laplacian with respect to some triangulation $\mathcal{T}$;

    \item 
    Adapt a result from \cite{Liu} to deduce from this a bound for the Laplacian acting on smooth functions on $P_\delta$;

    \item 
    Use quantitative estimates to bound the difference between the spectral gaps on $P_\delta$ and $P$.
\end{enumerate}

This approach using an inset hexagon $P_{\delta}$ is necessary because the metric $|\cdot|_{g}$ on $P$ is generated by the inverse matrix $u^{ij}$, which degenerates on the boundary (cf.  \cref{canonicalpotential} and \cref{perturbedpotential}). This is problematic, because adapting the \cite{Liu} method for this  finite-element approach relies on uniform metric lower bounds on any element; any element which includes any part of the boundary will have a useless lower bound. Performing a finite-element approach on $P_{\delta}$ for some $\delta>0$ instead of $P$ circumvents this issue. If $\delta>0$ is sufficiently small, it is possible to draw conclusions about the spectral gap on $P$ from finite-element information on $P_{\delta}$, but if $\delta$ is too small, then boundary degeneracy becomes an issue again; we need to balance these two competing issues cautiously. 

\subsubsection{Steps 1 and 2: FEM and smooth spectral gaps on $P_\delta$}

For $0<\delta<1$, define the inset hexagon
\[
P_\delta
:=
\operatorname{conv} (\{(-1+\delta,1-\delta),(-1+\delta,0),(0,-1+\delta),(1-\delta,-1+\delta),(1-\delta,0),(0,1-\delta)\}) \subset P
\]
and let $Q_\delta = \operatorname{conv} (\{(0,0),(1-\delta,-1+\delta),(1-\delta,0)\})$ be one of its fundamental sectors. 
Let $\sigma\in D_6$ be the reflection $$\sigma=\begin{pmatrix}1&0\\-1&-1\end{pmatrix},$$ the unique
non-trivial element of $D_6$ preserving $Q_\delta$; it exchanges the vertices $(1-\delta,0)$ and
$(1-\delta,-1+\delta)$ and fixes the median
$\ell_\delta=\{(t,-t/2):0\le t\le 1-\delta\}$ pointwise. 
Let $\mathcal{T}$ be some triangulation of $Q_\delta$ whose boundary edges are invariant under $\sigma$. 
The situation is shown in \cref{fig:triangulation}.

\begin{figure}
    \centering
    \begin{minipage}{0.45\textwidth}
        \includegraphics[width=\linewidth]{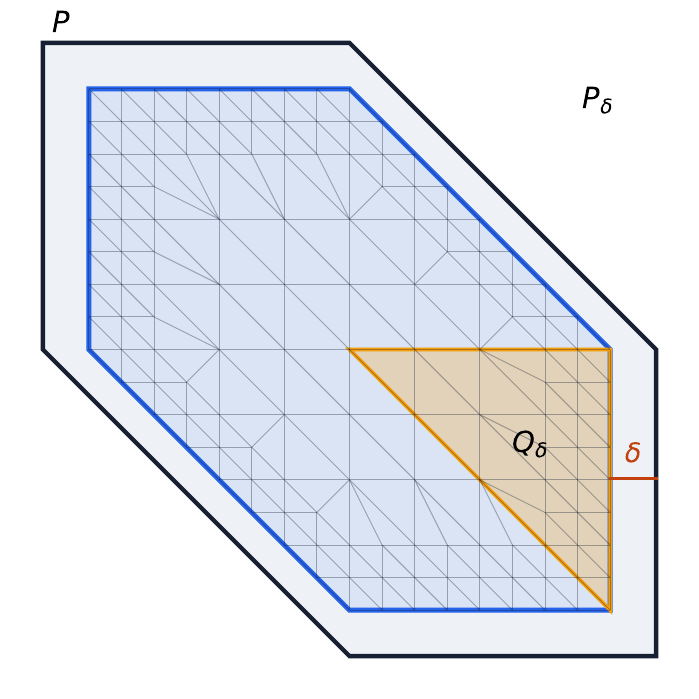}    
    \end{minipage}
    \begin{minipage}{0.45\textwidth}
        \includegraphics[width=\linewidth]{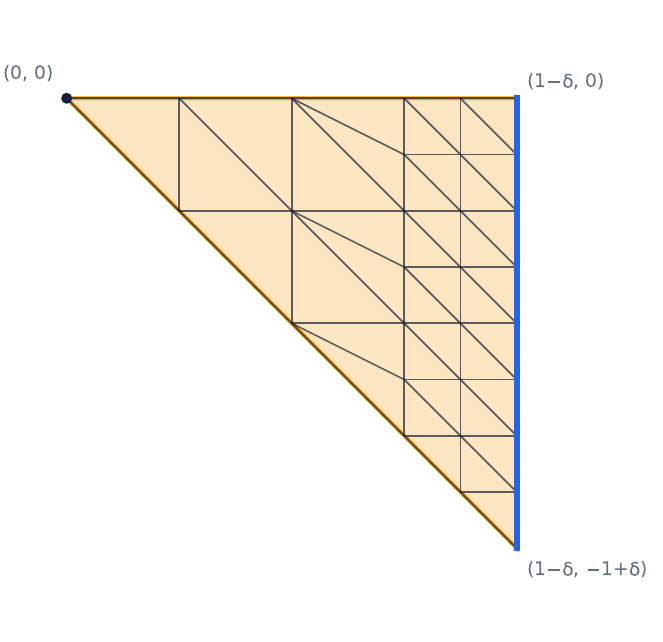}
    \end{minipage}
    \caption{The domains and triangulations in question.
    Left: the hexagon $P$ with inset hexagon $P_\delta$ in blue. The fundamental sector $Q_\delta$ is marked in orange.
    Right: the fundamental sector $Q_\delta$ with its triangulation. The triangulation is $D_6$-invariant and then by $D_6$-symmetry extended to all of $P_\delta$.
    Note that neither $\delta$ nor the triangulation are drawn to scale.
    Our chosen value of $\delta$ is too small to see in a figure, and the true triangulation has very small triangles near the boundary which are likewise too small to see, so we draw a coarser version of it.}
    \label{fig:triangulation}
\end{figure}

Let us now set up our function spaces. We define the Sobolev space 
\[
H^1(Q_\delta)^{D_6}
:=
\left\{
U|_{Q_\delta}:
U\in H^1(P_\delta)
\text{ is $D_6$-invariant}
\right\},
\]
as well as its mean zero version 
\begin{align}
\label{eq:V-space-def}
V
:=
\left\{
u\in H^1(Q_\delta)^{D_6}:
\int_{Q_\delta}u\,dx_1dx_2=0
\right\}.
\end{align}
The eigenvalue formulation in \cref{Qformulation} shows that $V$ should be highly-relevant to estimating the eigenvalue, but our finite element method also requires us to consider the Crouzeix--Raviart finite-element space of our triangulation $\mathcal{T}$, which is the space $\operatorname{CR}(Q_{\delta},\mathcal{T})$
of functions $v:Q_\delta\to\mathbb{R}$ (defined up to the edges of $\mathcal{T}$) such that
\begin{enumerate}
  \item $v|_T$ is affine, that is $v|_T(x)=a_T+b_T\cdot x$ for some $a_T\in\mathbb{R}$ and
        $b_T\in\mathbb{R}^2$, for every $T\in\mathcal{T}$; and
  \item for every interior edge $e=T_1\cap T_2$ of $\mathcal{T}$,
        \[
        \frac{1}{|e|}\int_e v|_{T_1}\,ds
        =
        \frac{1}{|e|}\int_e v|_{T_2}\,ds,
        \]
        equivalently $v|_{T_1}(m_e)=v|_{T_2}(m_e)$, where $m_e$ is the midpoint of $e$.
\end{enumerate}
Its $D_6$-invariant part is
\[
\operatorname{CR}(Q_\delta,\mathcal{T})^{D_6}
:=
\left\{
v\in\operatorname{CR}(Q_{\delta},\mathcal{T}):
v\circ\sigma=v
\right\}.
\]

We then define \begin{align}
\label{eq:fem-space-def}
V^h :=
\left\{
v_h\in\operatorname{CR}(Q_\delta,\mathcal T):
\begin{aligned}
&\int_{Q_\delta}v_h\,dx_1dx_2=0,\\
&v_h(m_e)=v_h(\sigma m_e)
\quad\text{whenever }e,\sigma e\text{ are edges of }\mathcal T
\end{aligned}
\right\},
\end{align}
and finally, define $V(h)=V+V^h$ (note that $V^h$ is \textit{not} necessarily a subset of $V$). 

To proceed with our finite element approach, we write $A=u^{ij}:Q_{\delta}\to \mathbb{R}^{2\times 2}$ for the positive and symmetric matrix function appearing in the stiffness form. Then for any $T \in \mathcal{T}$ we choose a positive-definite matrix $B_T \in \R^{2 \times 2}$ such that $A(x) \succeq B_T$ for all $x \in T$, where $\succeq$ denotes the Loewner order, i.e. $A(x)-B_T$ is positive semidefinite.
The matrices $B_T$ are auxiliary matrices and need not be $D_6$-invariant. Note that for any $f,g\in V(h)$, it is possible to make sense of the expressions 
\[
M_A(f,g)
:=
\int_{Q_{\delta}} \nabla f\cdot A\cdot\nabla g\,dx_1dx_2,
\qquad
M_B(f,g)
:=
\sum_{T\in\mathcal{T}}\int_T \nabla f\cdot B_T\cdot\nabla g\,dx_1dx_2,
\]
and
\[
N(f,g)
:=
\int_{Q_{\delta}}fg\,dx_1dx_2.
\]

The first positive $D_6$-invariant eigenvalue on the full inset hexagon is
\begin{equation}
    \label{equation:d6-sector-rayleigh-quotient}
    \lambda^{D_6}_{A,1}(Q_\delta)
    :=
    \inf_{u\in V\setminus\{0\}}
    \frac{M_A(u,u)}{N(u,u)}.
\end{equation}
 Since the mesh is connected and every $B_T$ is positive definite, $M_B(\cdot,\cdot)$ is an inner product on $V(h)$.
Let $\Pi:V(h)\to \operatorname{CR}(Q_\delta,\mathcal{T})^{D_6}$ be the Crouzeix--Raviart interpolation operator, defined by the requirement that for every edge $e$ with mid-point $m_e$, we have 
\begin{equation}
    \label{equation:cr-interpolation}
    (\Pi u)(m_e)
    =
    \frac{1}{|e|}\int_e u\,ds
\end{equation}
where $ds$ is the Euclidean length measure; $\Pi u$ is then uniquely determined by the requirement that it is affine on each triangle $T\in \mathcal{T}$. 
Consider now the function $p: V(h)\to CR(Q_{\delta},\mathcal{T})^{D_6}$ with 
\begin{equation}
    \label{equation:mean-zero-cr-projection}
    p u
    :=
    \Pi u
    -
    \frac{1}{|Q_\delta|}\int_{Q_\delta}\Pi u\,dx_1dx_2.
\end{equation}

\begin{lemma}
    \label{lemma:projection-equals-interpolation}
    The image of $p$ is contained in $V^h$. Furthermore, $p$ is the orthogonal projection with respect to $M_B(\cdot,\cdot)$; that is,
    \[
        M_B(u-pu,v_h)=0
        \qquad
        \text{for all $u\in V(h)$ and $v_h\in V^h$.}
    \]
\end{lemma}

\begin{proof}
The constant correction in \eqref{equation:mean-zero-cr-projection} gives the mean zero property required of elements of $V^h$, but does not change the gradient, so it is enough to prove that 
\begin{align*}
    M_B(u-\Pi u,v_h) = 0 \text{ for all $u\in V(h)$ and $v_h\in V^h$.}
\end{align*} On each triangle, both $B_T$ and $\nabla v_h$ are constant. Integration by parts therefore gives
\begin{align*}
M_B(u-\Pi u,v_h)
&=
\sum_{T\in\mathcal{T}}
\int_T \nabla(u-\Pi u)\cdot B_T\cdot\nabla v_h\,dx
\\
&=
\sum_{T\in\mathcal{T}}
\sum_{e\subset\partial T}
\bigl(B_T\nabla v_h\cdot n_{T,e}\bigr)
\int_e(u-\Pi u)\,ds.
\end{align*}
Every edge integral in the last expression vanishes by \eqref{equation:cr-interpolation}, and the result follows. 
\end{proof}

This orthogonal projection property is used in \cite{Liu} to find a convenient way to compare the eigenvalues on smooth functions with eigenvalues of finite elements:
\begin{theorem}
    \label{theorem:liu-theorem}
     Define
    \[
    \lambda_{B,1}
    :=
    \inf_{u\in V\setminus\{0\}}
    \frac{M_B(u,u)}{N(u,u)},
    \qquad
    \text{ and }
    \qquad
    \lambda_{B,\mathcal{T},1}
    :=
    \inf_{u_h\in V^h\setminus\{0\}}
    \frac{M_B(u_h,u_h)}{N(u_h,u_h)}.
    \]
    If $C(\mathcal{T})>0$ is any number satisfying 
    \begin{align}
        \label{equation:liu-necessary-error-estimate}
        \|{u-pu}_N
        \leq
        C(\mathcal{T})\|{u-pu}_{M_B} \text{ for all } u\in V,
    \end{align}
   then
    \begin{equation}
        \label{equation:combined-fem-lower-bound}
        \lambda^{D_6}_{A,1}(Q_\delta)
        \geq
        \lambda_{B,1}
        \geq
        \frac{\lambda_{B,\mathcal{T},1}}
        {1+\lambda_{B,\mathcal{T},1}C(\mathcal{T})^2}.
    \end{equation}
\end{theorem}

\begin{proof}
The first inequality follows immediately from $A(x) \succeq B_T$. 
The second is \cite[Thm. 2.1]{Liu}, applied to the inner products $M_B$ and $N$, the spaces $V$ and $V^h$, and the projection from \cref{lemma:projection-equals-interpolation}.
\end{proof}

Thus, reliable estimates using the finite-element method are achievable once we find a constant $C(\mathcal{T})$ satisfying \cref{equation:liu-necessary-error-estimate}. This is possible because the projection is, up to its mean-zero correction, the Crouzeix--Raviart interpolation operator.

\begin{corollary}
    \label{corollary:Lius-constant-for-A-norm}
    For each $T\in\mathcal{T}$, let $\alpha_T>0$ such that $\alpha_T\leq\lambda_{\min}(B_T)$.
    Then \eqref{equation:liu-necessary-error-estimate} holds with
    \begin{equation}
        \label{equation:liu-constant}
        C(\mathcal{T})
        =
        \max_{T\in\mathcal{T}}
        \frac{0.1893h(T)}{\sqrt{\alpha_T}},
    \end{equation}
    where $h(T)$ denotes the longest side of $T$.
\end{corollary}

\begin{proof}
Let $e=u-\Pi u$. By Section 4 of \cite{Liu}, on each triangle we have
\[
    \int_Te^2\,dx
    \leq
    \bigl(0.1893h(T)\bigr)^2
    \int_T|\nabla e|^2\,dx.
\]
Since $B_T\succeq\alpha_T I$, this gives
\[
    \int_Te^2\,dx
    \leq
    \frac{\bigl(0.1893h(T)\bigr)^2}{\alpha_T}
    \int_T\nabla e\cdot B_T\cdot\nabla e\,dx.
\]
For $u\in V$, formula \eqref{equation:mean-zero-cr-projection} gives
\[
    u-p u
    =
    e-\frac{1}{|Q_\delta|}\int_{Q_\delta}e\,dx_1dx_2.
\]
Subtracting the mean is the $L^2$-orthogonal projection onto the mean-zero functions, and hence
\[
    \|{u-p u}_N\leq\|{e}_N.
\]
Moreover, $u-p u$ and $e$ have the same elementwise gradient. Summing the local estimates over the triangulation and taking the largest local constant therefore proves \eqref{equation:liu-necessary-error-estimate} with the value in \eqref{equation:liu-constant}.
\end{proof}

\begin{remark}
    We impose no boundary condition at the hexagon edge $\operatorname{conv}(\{(1-\delta,0),(1-\delta,-1+\delta)\})$ of $Q_\delta$.
    This is the case for our smooth function space $V$ from \cref{eq:V-space-def} as well as for our finite element space from \cref{eq:fem-space-def}.
    Generally speaking, the minimum value of the Rayleigh-Ritz quotient on a space is equal to the smallest eigenvalue of the Laplace operator with Neumann boundary conditions, see \cite[Thm. 6.23]{Borthwick2020} for a description of this principle.
    This holds true in our situation, except that we impose the $D_6$-symmetry boundary condition on two of the edges of $Q_\delta$.
\end{remark}

\begin{remark}
    Computing the smallest finite element eigenvalue $\lambda_{B,\mathcal{T},1}$ from \cref{theorem:liu-theorem} is a standard finite-dimensional linear algebra problem.
    We compute rigorous eigenvalue enclosures using the \texttt{veigs} package from \url{https://github.com/yuuka-math/veigs} for INTLAB \cite{Ru99a}.
    Our implementation slightly differs from what is written in that we compute the \emph{second} smallest eigenvalue on $\operatorname{CR}(Q_\delta)^{D_6}$ rather than the smallest eigenvalue on $V^h \subset \operatorname{CR}(Q_\delta)^{D_6}$.
    The smallest eigenvalue is then zero.
    The reason for the discrepancy between implementation and the description in this section is that \cite[Thm. 2.1]{Liu} applies to positive definite forms, but not to positive semi-definite forms.
    Instead of this ad-hoc workaround, one could have used the more general framework from \cite{LIU2020112666} for semi-definite forms.
\end{remark}

\subsubsection{Step 3: comparing spectral gaps on $P$ and $P_\delta$}

We use $\mu$ to denote a positive number with $\text{Ric}(u,u)\ge \mu g(u,u)$ for all $u\in TM$, where $g$ denotes the approximate Kähler-Einstein metric. 

Let $M_{\delta}=P_{\delta}\times \mathbb{T}^2 \subset M$ be the corresponding subset of $M$. 
In \cref{eq:V-space-def}, a space of $D_6$-invariant functions on $Q_\delta$ was defined.
In this section, we need the analog function space on the hexagon $P_\delta$ for bookkeeping reasons, say $V'$ defined as
\begin{align}
V'
:=
\left\{
U \in H^1(P_\delta)^{D_6}:
\int_{P_\delta}U\,dx_1dx_2=0
\right\}.
\end{align}

\begin{theorem}\label{theorem:eigenvalue-inset-comparison}
   Let $\lambda_{\min}>0$ be the positive first eigenvalue of $\Delta$, restricted to toric and $D_6$-invariant scalar functions on $M$. 
   For any $0<\delta<3/7$, we have \begin{align}\label{deltaenergy}
       \lambda_{\min}^\delta
       :=
       \inf_{u\in V' \setminus \{0\}} \frac{\int_{M_{\delta}}\left|\nabla u\right|_{g}^2dV_g}{\int_{M_{\delta}}u^2 dV_g}\le \frac{\lambda_{\min}}{1-\sqrt{\frac{7\delta}{3}} \left(\frac{3\lambda_{\min}}{2\mu}+1\right)-\frac{7\delta}{3-7\delta}},
   \end{align}
    provided the right hand side is defined and positive. 
    In particular,
    \begin{align}
    \label{equation:lambda-min-lower-bound-final}
    \lambda_{\min}
    \ge
    \frac{
    \left(
    1-\sqrt{\frac{7\delta}{3}}-\frac{7\delta}{3-7\delta}
    \right)
    \lambda_{\min}^\delta
    }{
    1+
    \frac{3}{2\mu}\sqrt{\frac{7\delta}{3}}
    \displaystyle
    \lambda_{\min}^\delta
    }.
    \end{align}
\end{theorem}
The proof of this Theorem relies crucially on a few elementary observations. The first of these is a simple volume estimate.  
\begin{lemma}\label{basicvolume}
The Euclidean area of $P_{\delta}\subseteq \mathbb{R}^2$, denoted $A(P_{\delta})$, satisfies 
\begin{align*}
   3-7\delta\le A(P_{\delta})\le 3; 
\end{align*}
    the volume of $M_{\delta}$ satisfies 
    \begin{align*}
        12\pi^2-28\pi^2 \delta\le \vol(M_\delta)\le 12\pi^2. 
    \end{align*}
\end{lemma}
\begin{proof}
    Clearly the Euclidean area of the hexagon $P$ is $3$. The Euclidean area of the difference $P\setminus P_{\delta}$ is no greater than that of six thin strips with width $\delta$. The total Euclidean length of these six strips is no greater than the Euclidean length of the perimeter of the hexagon, which is $4+2\sqrt{2}$. 
    The result about the volume of $M_{\delta}$ then follows from the formula for the volume form $dV_g=dx_1dx_2d\theta_1d\theta_2$. 
\end{proof}
The second elementary observation concerns estimates of various integrals of $u$ and its gradient over $M_{\delta}$. The estimates are found by combining the definition of $\lambda_{\min}$, the previous volume estimates, and H\"older's inequality. 
\begin{lemma}\label{efuncest}
    Let $u \in H^1(M)^{T^2 \rtimes D_6}$ be an $L^2$-normalised mean-zero weak eigenfunction of $\Delta$ for the eigenvalue $\lambda_{\min}$, i.e. 
    \[
    \int_M |\nabla u|_g^2 dV_g
    =
    \lambda_{\min}
    \int_M u^2 dV_g
    =
    \lambda_{\min}.
    \]
    Then the restricted function $u:M_{\delta}\to \mathbb{R}$ satisfies 
    \begin{align*}
    \int_{M_{\delta}}\left|\nabla u\right|_{g}^2dV_g\le \lambda_{\min}, \qquad \int_{M_{\delta}}u^2 dV_g\ge 1-\sqrt{\frac{7\delta}{3}} \left(\frac{3\lambda_{\min}}{2\mu}+1\right), \qquad \left|\int_{M_{\delta}}udV_g\right|\le 2\pi \sqrt{7\delta}. 
    \end{align*}
\end{lemma}
\begin{proof}
    The first inequality holds since $M_{\delta}\subset M$ and $u$ is $L^2$-normalised on $M$, i.e. $\int_{M}u^2 dV_g=1$.
Next, the version of the Sobolev embedding theorem given on the first page of \cite{Ilias1983} combined with Lemma \ref{basicvolume} gives 
\begin{align*}
    \left|\left|u\right|\right|_{L^4}^2&\le \frac{3\left|\left|\nabla u\right|\right|_{L^2}^2}{2\mu\sqrt{\vol(M)}}+\frac{\left|\left|u\right|\right|_{L^2}^2}{\sqrt{\vol(M)}}
    =\frac{1}{2\sqrt{3}\pi}\left(\frac{3\lambda_{\min}}{2\mu}+1\right).
\end{align*}

    Applying Hölder's inequality to $u^2\cdot 1_{M\setminus M_{\delta}}$ then gives
    \begin{align*}
        \left|\left|u\right|\right|^2_{L^2 (M\setminus M_{\delta})}&= || u^2\cdot 1_{M\setminus M_{\delta}}||_{L^1(M)}\\&\le \left|\left| u^2 \right|\right|_{L^2(M)} \sqrt{\vol(M\setminus M_{\delta})}\\
&=\left|\left|u\right|\right|^2_{L^4}\sqrt{\vol(M)-\vol(M_{\delta})} \\
        &\le \frac{1}{2\sqrt{3}\pi}\left(\frac{3\lambda_{\min}}{2\mu}+1\right)\sqrt{28\pi^2\delta} \\
        &=\sqrt{\frac{7\delta}{3}} \left(\frac{3\lambda_{\min}}{2\mu}+1\right).
    \end{align*}

 Thus, the second inequality holds. Finally, H\"olders inequality applied to $\left|u\right|\cdot 1_{M\setminus M_{\delta}}$  gives 
 \begin{align*}
     \left|\int_{M\setminus M_{\delta}}udV_g\right|
     \leq
     \left|\left|u\cdot 1_{M\setminus M_{\delta}}\right|\right|_{L^1(M)}\le \left|\left|u\right|\right|_{L^2(M)}\sqrt{28\pi^2 \delta}=2\pi \sqrt{7\delta};
 \end{align*}
 the final inequality thus follows from the fact that $\int_{M}udV_g=0$. 
\end{proof}

\begin{proof}[Proof of Theorem \ref{theorem:eigenvalue-inset-comparison}]
    Let $u$ be the normalised eigenfunction from Lemma \ref{efuncest} and define $\tilde{u}:M_{\delta}\to \mathbb{R}$ to be $\tilde{u}(x)=u(x)+c$, where $c$ is chosen so that $\tilde{u}$ has mean-zero over $M_{\delta}$. Clearly $\int_{M_{\delta}}\left|\nabla \tilde{u}\right|_{g}^2dV_g=\int_{M_{\delta}}\left|\nabla u\right|_{g}^2dV_g\le \lambda_{\min}$. Furthermore, $c^2=\frac{\left(\int_{M_{\delta}}udV_g\right)^2}{\vol(M_\delta)^2}$, which gives  
    \begin{align}
    \label{eqn:delta-comparison-proof-step}
    \begin{split}
        \int_{M_{\delta}}\tilde{u}^2 dV_g&=\int_{M_{\delta}}u^2 dV_g+2c \int_{M_{\delta}}u dV_g+c^2\vol(M_\delta)\\
        &\ge  1-\sqrt{\frac{7\delta}{3}} \left(\frac{3\lambda_{\min}}{2\mu}+1\right)+\frac{\left(\int_{M_{\delta}} u d \mu_g\right)^2}{\vol(M_\delta)}-2\frac{\left(\int_{M_{\delta}} u d \mu_g\right)^2}{\vol(M_\delta)}\\
        &= 1-\sqrt{\frac{7\delta}{3}} \left(\frac{3\lambda_{\min}}{2\mu}+1\right)-\frac{\left(\int_{M_{\delta}} u d \mu_g\right)^2}{\vol(M_\delta)}\\
        &\ge 1-\sqrt{\frac{7\delta}{3}} \left(\frac{3\lambda_{\min}}{2\mu}+1\right)-\frac{ 7\delta }{3-7\delta}.
    \end{split}
    \end{align}
    Thus, 
    \begin{align*}
         \inf_{u\in V' \setminus \{0\}} \frac{\int_{M_{\delta}}\left|\nabla u\right|_{g}^2dV_g}{\int_{M_{\delta}}u^2 dV_g}&\le \frac{\int_{M_{\delta}}\left|\nabla \tilde{u}\right|_{g}^2dV_g}{\int_{M_{\delta}}\tilde{u}^2 dV_g}
         \le \frac{\lambda_{\min}}{1-\sqrt{\frac{7\delta}{3}} \left(\frac{3\lambda_{\min}}{2\mu}+1\right)-\frac{7\delta}{3-7\delta}}.
         \qedhere
    \end{align*}

    The final claim follows from
    \begin{align*}
    &\lambda_{\min}^{\delta}
    \left(
    1-\sqrt{\frac{7\delta}{3}}
    \left(\frac{3\lambda_{\min}}{2\mu}+1\right)
    -\frac{7\delta}{3-7\delta}
    \right)
    \le
    \lambda_{\min}^{\delta}
    \int_{M_\delta}\tilde u^2\,dV_g
    \le
    \int_{M_\delta}|\nabla\tilde u|_g^2\,dV_g
    \le \lambda_{\min},
    \end{align*}
    where we used \cref{eqn:delta-comparison-proof-step} in the first step, the definition of $\lambda_{\min}^\delta$ as the infimum over functions including $\widetilde{u}$, and in the last step we used the inequality from the beginning of the proof.
    Expanding the left-hand side, collecting the terms containing
    $\lambda_{\min}$, and dividing by the resulting positive factor gives the claimed lower bound.
\end{proof}

\begin{corollary}
    \label{corollary:L-spectral-gap}
    Let $u$ be the Kähler potential from \cref{subsection:numerical-approach}.
    Then the minimum absolute value of $T^2 \rtimes D_6$-invariant eigenvalues of $\mathscr{L}$, say $|\lambda_{\min}(\mathscr{L})|$, is bounded from below by the value in \cref{tab:numerical-constants}.
\end{corollary}

\begin{proof}
    Plugging the value for $C(\mathcal{T})$ from \cref{corollary:Lius-constant-for-A-norm} into \cref{equation:combined-fem-lower-bound} gives the lower bound for $\lambda^{D_6}_{A,1}$ recorded in \cref{table:quadratic-estimate-coefficients}.
    Plugging this into \cref{deltaenergy} gives the lower bound for $\lambda_{\min}$ provided by \cref{equation:combined-fem-lower-bound}.
    Here $\lambda_{\min}$ is the smallest positive
    $T^2\rtimes D_6$-invariant eigenvalue of $\Delta$.
    On the invariant $L^2$-space, we have
    \[
    \begin{aligned}
    \Spec(\Delta)
    &\subset \{0\}\cup[\lambda_{\min},\infty),\\
    \Spec\left(-\frac12\Delta+1\right)
    &\subset \{1\}\cup
    \left(-\infty,1-\frac12\lambda_{\min}\right].
    \end{aligned}
    \]
    Multiplication by $e^F-1$ is bounded and self-adjoint on this
    space, with operator norm at most
    $\|{e^F-1}_{C^0}=\|{\mathscr{E}}_{C^0}$.
    Thus \cite[Theorem V.4.10]{Kato1966} gives
    \[
    \Spec(\mathscr{L})
    \subset
    \left[1-\|{\mathscr{E}}_{C^0},1+\|{\mathscr{E}}_{C^0}\right]
    \cup
    \left(-\infty,
    1-\frac12\lambda_{\min}+\|{\mathscr{E}}_{C^0}\right].
    \]
    Consequently,
    \[
    |\lambda_{\min}(\mathscr{L})|
    \ge
    \min\left\{1,\frac12\lambda_{\min}-1\right\}
    -\|{\mathscr{E}}_{C^0}.
    \]
    Substituting the bounds from \cref{tab:numerical-constants}
    gives the claimed positive lower bound.
\end{proof}

\section{Perturbation}\label{perturb}

It is in this section that we will perturb the \emph{approximate solution} $(M,g,J,\omega_P)$ of the Einstein equation \cref{GRS} (more precisely, the complex Monge–Ampère equation \cref{equation:cx-Monge-Ampere}) that we constructed in \cref{approx} to a \emph{genuine solution}.
This is an application of the following simple consequence of the Banach fixed-point theorem.

\begin{theorem}
\label{theorem:fixed-point}
Let \(X,Y\) be Banach spaces, let \(u_0\in X\), and let \(r>0\). Let $B_r:=\{v\in X:\|{v}_X\le r\}$.
Let \(\mathscr{F}:u_0+B_r\to Y\) be a map admitting a decomposition $\mathscr{F}(u_0+v)=\mathscr{E}+\mathscr{L}v+\mathscr{N}(v)$ for $v \in B_r$,
where \(\mathscr{E} \in Y\), \(\mathscr{L}:X\to Y\) is an invertible linear operator and $\mathscr{N}(0)=0$.

Assume that there are constants \(\alpha_1,\alpha_2,\alpha_3\geq 0\) such that
\[
\|{\mathscr{E}}_{Y}\le \alpha_1,\qquad
\|{\mathscr{L}^{-1}}_{\mathcal{L}(Y,X)} \le \alpha_2,\qquad
\|{\mathscr{N}(v)-\mathscr{N}(w)}_{Y}
\le \alpha_3\bigl(\|{v}_{X}+\|{w}_{X}\bigr)\|{v-w}_{X}
\]
for all \(v,w\in B_r\). If
$\alpha_2 \alpha_1+\alpha_2 \alpha_3r^2\le r$,
$2\alpha_2 \alpha_3r<1$,
then there exists a unique \(v\in B_r\) such that
\[
\mathscr{F}(u_0+v)=0.
\]
Moreover, for every \(v^{(0)}\in B_r\), the iteration
\[
v^{(k+1)}
:=-\mathscr{L}^{-1}\bigl(\mathscr{E}+\mathscr{N}(v^{(k)})\bigr)
\]
converges in \(X\) to this solution.
\end{theorem}
Recall that the operator of interest $\mathscr{F}$ is defined in \cref{equation:cx-Monge-Ampere}. Furthermore, the decomposition of $\mathscr{F}$ required by \cref{theorem:fixed-point} has already been found with \cref{Taylordecomp}; it is conveniently broken down as $\mathscr{F}(\phi)=\mathscr{E}+\mathscr{L}(\phi)+\mathscr{N}(\phi)$,  
where
\begin{align*}
  \mathscr{E}
    &
    =
    1-e^F,
    \\
    \mathscr{L}\phi
    &
    =
    -\frac{1}{2} \Delta \phi + e^F \cdot \phi
    \text{ is a linear differential operator},
    \\
    \mathscr{N} \phi
    &
    =
    (i \partial \bar \partial \phi)^2/\omega_P^2
    -
    e^F \cdot (e^{-\phi}-1+\phi)
    =
    \frac12\star (i\partial\bar\partial \phi)^2
    -e^F(e^{-\phi}-1+\phi)
    \text{ are the higher order terms.}
\end{align*}

For reasons related to the dimension of the problem and the non-linear structure of this operator, we will choose 
\begin{align*}
    X=H^5(M,g)^{T^2 \rtimes D_6}, \qquad Y=H^3(M,g)^{T^2 \rtimes D_6}
\end{align*}
i.e. we are considering Sobolev spaces of $D_6$ and $T^2$-invariant scalar functions on $M$. Thus, the existence of a nearby toric K\"ahler-Einstein metric will follow by estimating the three numbers $\alpha_1,\alpha_2,\alpha_3$ (with norms evaluated with respect to the correct Banach spaces), and then showing that there is an $r>0$ for which $\alpha_2\alpha_1+\alpha_2\alpha_3 r^2\le r$, and $2\alpha_2\alpha_3r<1$. In this case, $r$ will tell us how close our approximate toric K\"ahler-Einstein metric is to the true solution.

\begin{proof}[Proof of \cref{maintheorem}]
    Choose the function $u:P\to \mathbb{R}$ from \cref{approx}, and consider the associated approximate K\"ahler-Einstein geometry $(M_P,g,\omega_P,J)$ described in \cref{othertoricKahler}. As we saw in \cref{subsection:einstein-condition-for-toric-kaehler} (especially  \cref{Taylordecomp}), the K\"ahler-Einstein condition for the perturbed K\"ahler geometry $(M_P,\omega,J)$ with $\omega=\omega_P+i\partial \overline{\partial \phi}$ is equivalent to 
    \[
0
=
\mathscr{E}+\mathscr{L}\phi+\mathscr{N}\phi.
\]

By \cref{H3aposteriorierror}, an appropriate choice for $\alpha_1$ is given in \cref{tab:numerical-constants}. By \cref{corollary:inj-estimate}, we can choose $\alpha_2$ to coincide with $\Cr{const:injectivity-estimate}=\left(
    \Cr{H^5-estimate-Laplace-term}
    +
    \frac{\Cr{H^5-estimate-u-term}}{\lambda_{\min}(\mathscr{L})}
    \right)$, after replacing $\lambda_{\min}(\mathscr{L})$ by its lower bound given in \cref{tab:numerical-constants}, as certified in \cref{corollary:L-spectral-gap}. Finally, for each $r>0$, we can choose $\alpha_3=\Cr{const:multiplication-theorem} \left[
        \frac{1}{2}+
        \left(
        1+
        \Cr{const:multiplication-theorem}\alpha_1
        \right)
        e^{2r\Cr{const:multiplication-theorem}
      }
        \right]$, which follows immediately from \cref{proposition:non-linear-estimate}. The existence and closeness of the required K\"ahler-Einstein metric then follows from applying \cref{theorem:fixed-point}, with the value of $r$ listed in \cref{tab:numerical-constants}.
\end{proof}

\section{Applications}\label{applications}

\subsection{Holomorphic sectional curvature}\label{HSC}

\cite[Conjecture 1.3]{broder2023some} suggests that if $\Ric >0$ for a Kähler manifold, then there exists a cohomologous metric with positive holomorphic sectional curvature.
We know by \cite{zhang2026positiveholomorphicsectionalcurvature} that this is the case for the third del Pezzo surface.
However, from general theory alone it is not known whether the Kähler-Einstein metric has positive sectional curvature.
In this section we prove that the sectional curvature is indeed \emph{not} positive everywhere.

We achieve this in first steps:
first, we prove a quantitative comparison theorem.
If the two Kähler metrics $g$ and $g^{KE}$ are close, then their holomorphic sectional curvatures are also close.
And closeness of the metrics is guaranteed by our construction \cref{maintheorem}.
Second, using certified numerics, we exhibit an explicit point at which the holomorphic sectional curvature with respect to $g$ is negative.
Taking both together, we deduce that also $g^{KE}$ must be negative near that point.
For computational reasons it is easier to certify that the integral of the curvature of $g^{KE}$ in a small open set is negative.
This implies that the curvature must be negative \emph{somewhere}, though we cannot say at which point exactly.

\subsubsection{Metric comparison lemmas}

Throughout the section, assume we are given Kähler forms $\omega_1,\omega_2$ with corresponding metrics $g_1,g_2$ on the same compact complex manifold $(M,J)$.
Furthermore, assume that $\omega_2=\omega_1+i \partial \bar \partial \varphi$ and 
\begin{align*}
    ||\varphi||_{H^5(M,g_1)}\le \varepsilon. 
\end{align*}
Later on, the role of $\omega_1$ will be played by $\omega_P$, and $\omega_2$ is the unknown true Kähler-Einstein metric for which we obtained the estimate for its Kähler potential from our application of the fixed point theorem.

The goal of this section is to compare the curvatures of $\omega_1$ and $\omega_2$, which will allow us to make statements about the curvature of the true Kähler-Einstein metric, even though it is only known implicitly.

\begin{lemma}
    \label{lemma:metric-comparison}
    For all $x \in M$ and $X,Y \in T_xM$ we have
    $|(g_2-g_1)(X,Y)| \leq \Cl{const:metric-difference}\varepsilon|X|_{g_1} |Y|_{g_1} := 4\Cr{const:emb-C4}
        \Cr{const:emb-C3}
        \Cr{const:emb-C2}
        \Cr{const:emb-C1}\varepsilon |X|_{g_1} |Y|_{g_1}$.
\end{lemma}
\begin{proof}
   For any non-zero $X,Y$, 
    \begin{align}
    \label{equation:metric-difference-phi-identity}
        g_2(X,Y)-g_1(X,Y)
        =(i \partial \overline{\partial}\varphi)(X,JY)
        =\frac{1}{2} \left( \nabla^2 \varphi(X,Y)+\nabla^2 \varphi(JX,JY) \right)
    \end{align}
    where in the second step we used \cite[Exercise 7.50]{Ballmann2006}.
    Since $g_1$ is Kähler, the complex structure $J$ is an isometry.
    Then:
    \begin{align*}
        \frac{|g_2(X,Y)-g_1(X,Y)|}{|X|_{g_1} |Y|_{g_1}}
        \le ||\nabla^2_{g_1}\varphi||_{L^{\infty}(M,g_1)}
        \le 2
        \Cr{const:emb-C4}
        \Cr{const:emb-C3}
        \Cr{const:emb-C2}
        \Cr{const:emb-C1} ||\nabla^2_{g_1} \varphi ||_{L_3^2(M,g_1)}
        \le 2
        \Cr{const:emb-C4}
        \Cr{const:emb-C3}
        \Cr{const:emb-C2}
        \Cr{const:emb-C1}||\varphi ||_{L_5^2(M,g_1)},
    \end{align*}
    where in the second step we used \cref{equation:L-infty-embedding}.
    Here, the constants quietly depend on the metric $g_1$, and the inequality was stated for functions, but the same proof shows it for tensors and so we may apply it to the tensor $\nabla^2_{g_1} \varphi$ here.
    This proves the result.
\end{proof}

\begin{lemma}
\label{lemma:curvature-comparison}
    Let $R_i \in \Omega^{1,1}(\End(T^{1,0} M))$ be the Chern curvature operator of $g_i$ for $i \in \{1,2\}$.
    If $\Cr{const:metric-difference} \varepsilon<1$, then
    \[
    \|{\, |R_2-R_1|_{\text{op},g_1} \,}_{L^2(M,g_1)}
    \leq
    \frac{\varepsilon}{1-\Cr{const:metric-difference} \varepsilon}
    +
    \frac{\Cr{const:emb-C1}^2 \varepsilon^2}{(1-\Cr{const:metric-difference} \varepsilon)^2}.
    \]
\end{lemma}

\begin{proof}
    Using indices for holomorphic coordinates, and writing $h=g_2-g_1$, we obtain:
    \begin{align*}
        (\nabla^1_i h)_{j \bar l}
        =
        (\nabla^1_i g_2)_{j \bar l}
        =
        Z_i(g_2(Z_j, \bar Z_l))
        -g_2(\nabla^1_{Z_i}Z_j, \bar Z_l)
        -g_2(Z_j, \nabla^1_{Z_i} \bar Z_l)
        =
        \partial_i \left( (g_2)_{j \bar l} \right)
        -
        \Gamma(g_1)^q_{ij} (g_2)_{q \bar l},
    \end{align*}
    where we used $\nabla^1 g_1=0$ in the first step;
    the Leibniz rule for covariant derivatives of tensors in the second step;
    $\nabla^1_{Z_i} Z_j=\Gamma(g_1)^q_{ij}Z_q$ by \cite[Eqn. 4.37]{Ballmann2006} and $\nabla^1_{Z_i} \bar Z_l=0$ by \cite[Eqn. 4.36]{Ballmann2006} in the last step.
    Defining $\Psi^k_{ij} := \Gamma(g_2)^k_{ij}-\Gamma(g_1)^k_{ij}$, we get:
    \[
    g_2^{k \bar l} \nabla^1_i h_{j \bar l}
    =
    g^{k \bar l}_2 (\partial_i (g_2) _{j \bar l} - \Gamma(g_1) ^q_{ij} (g_2) _{q \bar l} )
    =
    \Gamma(g_2)^k_{ij}-\Gamma(g_1)^q_{ij} \delta^k_q
    =
    \Gamma(g_2)^k_{ij}-\Gamma(g_1)^k_{ij}= \Psi^k_{ij},
    \]
    where we plugged in the above expression for $\nabla^1_i h_{j \bar l}$ in the first step,
    and in the second step we used $\Gamma(g)^k_{ij}=g^{k \bar l} \partial_i g_{j \bar l}$ from \cite[Eqn. 4.39]{Ballmann2006}.
    \cite[Eqn. 3.96]{Boucksom2013} gives 
    $(\nabla^2_{\bar b} \Psi)^k_{ip}
    =R(g_1)^k_{i \bar b p}-R(g_2)^k_{i \bar b p}$.
    Since $\Psi \in T^{1,0}M \otimes ((T^*)^{1,0}M)^{\otimes 2}$, i.e. only holomorphic indices, we have by \cite[Eqn. 4.36]{Ballmann2006} and the Leibniz rule for covariant derivatives of tensors $(\nabla^1_{\bar b} \Psi)^k_{ip}=\partial_{\bar b} \Psi^k_{ip} = (\nabla^2_{\bar b} \Psi)^k_{ip}$.
    Altogether:
    \begin{align*}
        (R(g_2)-R(g_1))^k_{i \bar b p}
        &=
        - (\nabla^1_{\bar b} \Psi)^k_{ip}
        =
        -(\nabla^1_{\bar b} g_2^{k \bar l})(\nabla^1_i h)_{p \bar l}
        -
        g^{k \bar l}_2(\nabla^1 _{\bar b} \nabla^1_i h)_{p \bar l},
    \end{align*}
    Now $g_2^{k \bar l}(g_2)_{r \bar l}=\delta^k_r$ and taking $\nabla^1_{\bar b}$ on both sides and multiplying by $g_2^{-1}$
    \[
    \nabla^1_{\bar b} g_2^{k \bar l}
    =
    -
    g_2^{k \bar q}
    (\nabla^1_{\bar b} g_2)_{r \bar q} g_2^{r \bar l}
    =
    -
    g_2^{k \bar q}
    (\nabla^1_{\bar b} h)_{r \bar q} g_2^{r \bar l}.
    \]
    Plugging this into the previous equation gives
    \begin{align}
    \label{eqn:curvature-difference-identity}
    (R(g_2)-R(g_1))^k_{i \bar b p}
    =
    g_2^{k \bar q} (\nabla^1_{\bar b} h)_{r \bar q} g_2^{r \bar l} (\nabla^1_i h)_{p \bar l}
    -
    g^{k \bar l}_2(\nabla^1_{\bar b} \nabla^1_i h)_{p \bar l}.
    \end{align}
    By \cref{lemma:metric-comparison}, $\|{ \,|h|_{\text{op},g_1}\,}_{L^\infty(M,g_1)} \leq \Cr{const:metric-difference} \varepsilon$, and hence $(1-\Cr{const:metric-difference} \varepsilon)g_1 \leq g_2 \leq (1+\Cr{const:metric-difference} \varepsilon)g_1$ as quadratic forms.
    Hence $|g_2^{-1}|_{\text{op},g_1} \leq \frac{1}{1-\Cr{const:metric-difference} \varepsilon}$ everywhere.
    Thus, \cref{eqn:curvature-difference-identity} gives the pointwise estimate $|R_2-R_1|_{\text{op},g_1}
    \leq
    \frac{|\nabla_1^2 h|_{g_1}}{1-\Cr{const:metric-difference} \varepsilon}
    +
    \frac{|\nabla_1 h|_{g_1}^2}{(1-\Cr{const:metric-difference} \varepsilon)^2}$ which yields:
    \[
    \|{\, |R_2-R_1|_{\text{op},g_1} \,}_{L^2}
    \leq
    \frac{\|{(\nabla^1)^2 h}_{L^2}}{1-\Cr{const:metric-difference} \varepsilon}
    +
    \frac{\|{\nabla^1 h}_{L^4}^2}{(1-\Cr{const:metric-difference} \varepsilon)^2}
    \leq
    \frac{\|{(\nabla^1)^2 h}_{L^2}}{1-\Cr{const:metric-difference} \varepsilon}
    +
    \Cr{const:emb-C1}^2
    \frac{\|{\nabla^1 h}_{L^2_1}^2}{(1-\Cr{const:metric-difference} \varepsilon)^2}
    \leq
    \frac{\|{\varphi}_{L^2_4}}{1-\Cr{const:metric-difference} \varepsilon}
    +
    \Cr{const:emb-C1}^2
    \frac{\|{\varphi}_{L^2_4}^2}{(1-\Cr{const:metric-difference} \varepsilon)^2},
    \]
    where in the last step we used \cref{equation:metric-difference-phi-identity}.
    This proves the claim.
\end{proof}

\begin{definition}
    \label{definition:HSC}
    Let $(M,g,J)$ be Kähler.
    For $p \in M$ and $0 \neq X \in T_p M$ the number
    \[
    H_p(X)
    :=
    \frac{\Riem (X,JX,JX,X)}{|X|^4}
    =
    \frac{g(R(X',X'')X',X'')}{|X|^4},
    \quad
    X'=\frac{1}{\sqrt{2}}(X-iJX) \in T^{1,0}M,
    X''=\frac{1}{\sqrt{2}}(X+iJX) \in T^{0,1}M
    \]
    is called \emph{holomorphic sectional curvature}.
    Here, $R \in \Omega^{1,1}(\End(T^{1,0}M))$ denotes the Chern curvature of $g$ and the identifications $X \mapsto X'$, $X \mapsto X''$ of the tangent bundle and holomorphic/anti-holomorphic tangent bundles where chosen to be isometries, compare with the usual factor $\frac{1}{2}$ in \cite[Eqn. 2.17]{Ballmann2006}.
\end{definition}

\begin{proposition}
\label{proposition:negative_HSC_at_point}
    Let $0\leq \eta < 1$ and $\rho \geq 0$ such that
    \[
    |(g_2-g_1)(X,Y)|
    \leq
    \eta |X|_{g_1} |Y|_{g_1} \text{ for all }
    X,Y \in TM 
    \quad\text{ and }\quad
    \|{\, |R_2-R_1|_{\text{op},g_1}\,}_{L^2(M,g_1)} \leq \rho.
    \]
    Let $U \subset M$ be open and $V \in \Gamma(U,TM)$ be a $g_1$-unit vector field.
    Set $Q_i(V) := \Riem_i(V,JV,JV,V)$.
    If
    \[
    \int_U Q_1(V) \vol_{g_1}
    <
    -\vol_{g_1}(U)^{1/2} (\eta \|{R_1}_{L^2(U,g_1)}+(1+\eta)\rho),
    \]
    then there exists $p \in U$ such that $H^{g_2}_p(V) < 0$.
\end{proposition}

\begin{proof}
    Let $V'$, $V''$ as in \cref{definition:HSC}, then:
    \begin{align*}
        |Q_2(V)-Q_1(V)|
        &\leq
        |(g_2-g_1) (R_1(V',V'')V',V'')|
        +|g_2((R_2-R_1)(V',V'')V',V'')|
        \\
        &\leq
        \eta |R_1|_{g_1}+(1+\eta)|R_2-R_1|_{\text{op},g_1}.
    \end{align*}
    Integration and Cauchy-Schwarz give
    \begin{align*}
        \int_U Q_2(V) \vol_{g_1}
        &\leq
        \int_U Q_1(V) \vol_{g_1}
        +\vol_{g_1}(U)^{1/2}
        \left(
        \eta \|{R_1}_{L^2(U,g_1)}+
        (1+\eta)\|{\,|R_2-R_1|_{\text{op},g_1}\,}_{L^2(U,g_1)}
        \right)
        \\
        &\leq
        \int_U Q_1(V) \vol_{g_1}
        +\vol_{g_1}(U)^{1/2}
        \left(
        \eta \|{R_1}_{L^2(U,g_1)}+
        (1+\eta)\rho
        \right)
        <0.
    \end{align*}
    Thus, $Q_2(V)<0$ at some point.
    Because $Q_2$ is the numerator in the definition of $H_2$ and the denominator is positive, this implies $H_2(V)<0$ at the same point.
\end{proof}

\subsubsection{A point with certified negative holomorphic sectional curvature}

We now exhibit a point and a direction in which the holomorphic sectional curvature with respect to the approximate Kähler-Einstein metric $g$ is negative.
We begin by writing the curvature in coordinates:

\begin{lemma}
    Let $(M,\omega, J)$ be a toric K\"ahler metric with symplectic potential $u$, $p\in M^0$ and $X \in TM$. Then, the holomorphic sectional curvature of $g$ in the direction $X$ at $p$ is 
    \[
    H(\xi)
    =
    -\frac{
    \displaystyle
    \sum_{a,b,i,j,k,l=1}^{2}
    u^{ai}u^{kb}
      u^{jl}_{ab}
    \,\xi^i\overline{\xi^j}\xi^k\overline{\xi^l}
    }{
    \displaystyle
    2\left(
    \sum_{r,s=1}^{2}u^{rs}\xi^r\overline{\xi^s}
    \right)^2
    }.
    \]
    where $X = \xi^i\frac{\partial}{\partial z_i} +\bar\xi^i\frac{\partial}{\partial \bar z_i}$, $\xi^i \in \mathbb{C}$.
    
\end{lemma}

\begin{proof}
    Since $X \in TM$, there exists $Z \in T^{1,0}M$ such that $X = Z + \bar Z$. Substituting this into the definition of holomorphic sectional curvature (\ref{definition:HSC}) implies $H_p(X) = \frac{\langle R(Z,\bar Z) Z, \bar Z\rangle}{(\langle Z,\bar Z\rangle)^2}$. The formula follows from equation \eqref{eq:Riemannian_coordinates} and $g_{i\bar j} = \frac{1}{2}u^{ij}$.
\end{proof}

We first use the approximate metric $g=g_1$ to locate a candidate.
A beam search tests real constant torus directions represented by
primitive integer pairs and refines the most promising
cell--direction pairs. These preliminary evaluations are used 
to select a direction and region. The final proof is a separate
interval calculation. The selected candidate and search
parameters are recorded in \cref{Data}.

For the rigorous calculation, write $u^{ij}=A^{ij}/D_0$ as in \eqref{inverse_metric} and
$u^{jl}_{ab}=N^{jl}_{ab}/D_0^3$, using the quotient
numerators constructed in
Subsection~\ref{subsection:a-posteriori-geometric-bounds}. For the
fixed real direction $\xi$, define
\begin{equation}
\label{eq:hsc-coefficient-contractions}
\begin{aligned}
S&=\sum_{i,j=1}^{2}A^{ij}\xi^i\xi^j,
&
W_a&=\sum_{i=1}^{2}A^{ai}\xi^i,\\
H_{ab}&=\sum_{j,l=1}^{2}N^{jl}_{ab}\xi^j\xi^l,
&
P&=-\frac12\sum_{a,b=1}^{2}W_aW_bH_{ab},
\qquad
Q_1(V)=\frac{P}{D_0^3S^2}.
\end{aligned}
\end{equation}
Here $V$ is the $g_1$-unit normalization of the constant direction
$\xi$, and $S/D_0$ is its unnormalised squared length. The verifier
proves $D_0>0$ and $S>0$ on every integration box, so all
denominators in \eqref{eq:hsc-coefficient-contractions} are strictly
positive.

The base region is subdivided into rectangles. On each
rectangle, directed-rounding interval arithmetic encloses $Q_1(V)$
and $|\Riem(g_1)|^2$. These enclosures are multiplied by the exact
box areas, summed, and multiplied by a rigorous enclosure of the
torus factor $4\pi^2$. Denote the resulting quantities by
$I_Q=\int_UQ_1(V)\,dV_{g_1}$ and
$R_U=(\int_U|\Riem(g_1)|^2\,dV_{g_1})^{1/2}$. The program then
encloses
\begin{equation}
\label{eq:hsc-certificate-margin}
\mathcal M
=
I_Q+\operatorname{vol}_{g_1}(U)^{1/2}
\bigl(\eta R_U+(1+\eta)\rho\bigr),
\end{equation}
where $\eta$ and $\rho$ are the a posteriori comparison constants
obtained from the fixed-point verification. Certification requires
the upper endpoint of the interval for $\mathcal M$ to be strictly
negative. Proposition~\ref{proposition:negative_HSC_at_point} then
implies that the exact K\"ahler--Einstein metric has negative
holomorphic sectional curvature at some point of $U$.

Figure~\ref{fig:negative-hsc-region} shows the selected region and
the numerical profile of $Q_1(V)$.

\begin{figure}[t]
   \centering
   \includegraphics[width=0.6\linewidth]
       {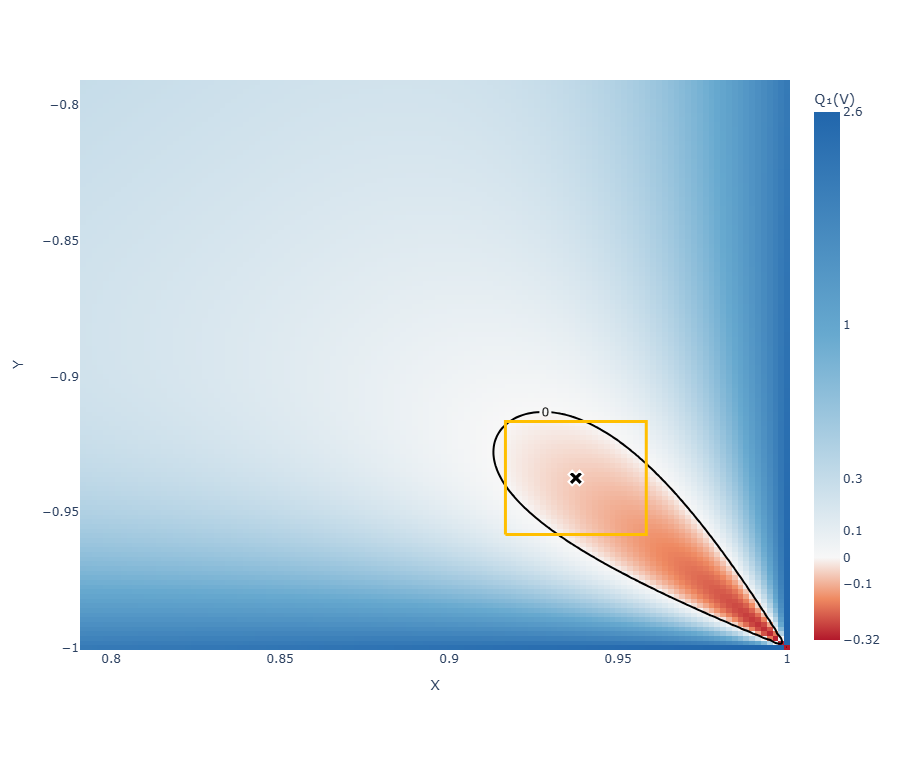}
   \caption{Numerical plot of $Q_1(V)$ for $\xi=(1,1)$. The black
   contour is its zero set, the marked point is the beam-search seed,
   and the highlighted rectangle is the region used in the rigorous
   interval certification. The coordinates $x$ and $y$ are coordinates on the hexagon.}
   \label{fig:negative-hsc-region}
\end{figure}

\subsection{First eigenvalue}\label{Firsteigenvalueestimate}

In \cref{subsection:non-sharp-eigenvalue-bound} we computed a rigorous lower bound for the first $T^2 \rtimes D_6$-invariant eigenvalue of the Laplace operator, $\lambda_{\min}$, with respect to the approximate Kähler-Einstein metric $(M_P,\omega_P,J)$.
In this section we will first compute a rigorous \emph{upper} bound for $\lambda_{\min}$.
Second, we will convert both bounds into rigorous bounds for the first $T^2 \rtimes D_6$-invariant eigenvalue with respect to the genuine Kähler-Einstein metric $(M_P,\omega^{KE},J)$, say $\lambda_{\min}^{KE}$.

\begin{lemma}\label{approxmineigenvalue}
    The first $T^2 \rtimes D_6$-invariant eigenvalue $\lambda_{\min}$ of $(M_P,\omega_P,J)$ satisfies the bound recorded in \cref{tab:numerical-constants}.
\end{lemma}
\begin{proof}
We have the Rayleigh-Ritz characterisation of $\lambda_{\min}$, namely
\[
\lambda_{\min}
=
\inf_{f}
\frac{\int_{M} |\nabla f|_{\omega_P}^2dV_g}{\int_M f^2 dV_g},
\]
where the infimum is taken over smooth and $T^2 \rtimes D_6$-invariant functions $f$ that have mean zero, but are not identically zero. 
In particular, any such function $f$ can be used to find an upper bound for this infimum.
Letting $q=x_1^2+x_1x_2+x_2^2$ and $f_0=q-\frac{5}{12}$, we have
\[
\int_M f_0 =
3 \cdot 4\pi^2 \int_{[0,1] \times [-1,0]} f_0 dx_1dx_2=0,
\quad
\int_M f_0^2 dV_g = 
3 \cdot 4\pi^2 \int_{[0,1] \times [-1,0]} f_0^2dx_1dx_2=
3 \cdot 4\pi^2 \cdot \frac{43}{720}.
\]
Here, the factor $3$ appears because the area of the moment polytope is three times the area of the square $[0,1] \times [-1,0]$, and the factor $4 \pi^2$ follows from the formula for the volume form $dV_g$ given in \cref{volumeform}. 
We estimate the integral of $|\nabla f_0|_{\omega_P}^2$ over $M$ using the formula 
\begin{align*}
    \int_{M} |\nabla f_0|^2_{g}dV_g=12\pi^2\int_{[0,1]\times [-1,0]} u^{ij}\frac{\partial f_0}{\partial x_i}\frac{\partial f_0}{\partial x_j}dx_1dx_2,
\end{align*} combined with straightforward interval arithmetic. 
\end{proof}

In \cref{HSC}, we showed that the approximate Kähler-Einstein metric and the true Kähler-Einstein metric are close in the $C^0$ topology. 
In this subsection, we show how this also implies closeness of the first eigenvalues.
A similar estimate was used in the proof of \cref{corollary:L-spectral-gap}.

\begin{proposition}\label{proposition:eigenvalue-comparison}
    If $\Cr{const:metric-difference}\varepsilon<1$, then
    \[
    \frac{(1-\Cr{const:metric-difference}\varepsilon)^2}{(1+\Cr{const:metric-difference}\varepsilon)^3}
    \lambda_{\min}
    \leq
    \lambda_{\min}^{KE}
    \leq
    \frac{(1+\Cr{const:metric-difference}\varepsilon)^2}{(1-\Cr{const:metric-difference}\varepsilon)^3}
    \lambda_{\min}.
    \]
\end{proposition}

\begin{proof}
    As in the previous section, let $g_1$ denote the approximate Kähler-Einstein metric induced by $\omega_P$ and $g_2$ be the genuine Kähler-Einstein metric in the same Kähler class.
    We write
    \[
    N_i(f)
    :=
    \int_M |\nabla f|^2_{g_i} dV_{g_i},
    \quad
    D_i(f)
    :=
    \inf_{c \in \R}
    \int_M (f-c)^2 dV_{g_i}
    \quad
    \text{ for }
    \quad
    i \in \{1,2\}
    \]
    for the numerator and denominator of the Rayleigh-Ritz quotient, respectively.
    In particular, 
    \[
    \lambda_{\min}
    =
    \inf_{f}
    \frac{N_1(f)}{D_1(f)},
    \quad
    \lambda_{\min}^{KE}
    =
    \inf_{f}
    \frac{N_2(f)}{D_2(f)},
    \]
    where this time, the infimum is taken over smooth and $T^2 \rtimes D_6$-invariant smooth functions that are simply \textit{non-constant}; the cost of removing the condition that $f$ has mean zero with the infimum expression for $D_i(f)$ involving the constat $c$.
    From \cref{lemma:metric-comparison}, we get
    $(1-\Cr{const:metric-difference}\varepsilon)g_1 \leq g_2 \leq (1+\Cr{const:metric-difference}\varepsilon) g_1$ as bilinear forms; two immediate consequences are that 
    \begin{align*}
        \frac{|\nabla_{g_1}f|_{g_1}^2}{1+\Cr{const:metric-difference}\varepsilon}  \leq |\nabla_{g_2}f|_{g_2}^2 \leq \frac{|\nabla_{g_1}f|_{g_1}^2}{1-\Cr{const:metric-difference}\varepsilon} 
    \end{align*}
    for any smooth functions $f$, and
    \begin{align*}
        (1-\Cr{const:metric-difference}\varepsilon)^2 dV_{g_1} \leq dV_{g_2} \leq (1+\Cr{const:metric-difference}\varepsilon)^2 dV_{g_1}
    \end{align*}
    when applied to any positively oriented frame.
    Hence,
    \begin{align*}
    \frac{(1-\Cr{const:metric-difference}\varepsilon)^2}{1+\Cr{const:metric-difference}\varepsilon} N_1(f)
    \leq
    N_2(f)
    \leq
    \frac{(1+\Cr{const:metric-difference}\varepsilon)^2}{1-\Cr{const:metric-difference}\varepsilon} N_1(f)
    \end{align*}
    and 
    \begin{align*}
    (1-\Cr{const:metric-difference}\varepsilon)^2
    \int_M (f-c)^2 dV_{g_1}
    \leq
    \int_M (f-c)^2 dV_{g_2}
    \leq
    (1+\Cr{const:metric-difference}\varepsilon)^2
    \int_M (f-c)^2 dV_{g_1}
    \quad
    \text{ for all }
    \quad
    c \in \R.
    \end{align*}
    Taking the infimum over $c$ proves the claim.
\end{proof}

By plugging in the numerical values for $\Cr{const:metric-difference}$ and the Kähler potential bound $\|{\varphi}_{H^5(M,\omega_P)}$ for the Kähler potential of the true Kähler-Einstein metric from \cref{tab:numerical-constants} and the known bounds for $\lambda_{\min}$ from \cref{approxmineigenvalue}, we thus prove the enclosure for $\lambda_{\min}^{KE}$ given in \cref{tab:numerical-constants}.
This is consistent with the numerical values from the literature:
6.32 in \cite[Section 4.4]{Doran2008}, 6.322 in \cite[Section 6.1]{Doran2008}, 6.3288 in \cite{HallMurphy}.

\newpage 
\appendix

\section{Toric K\" ahler Curvature with Complex Geometry}\label{CDE}

The aim of this appendix is to prove the formulae for the curvature quantities.
The idea will be to use complex coordinates to simplify the computations. A-priori, the curvature expressions have $\operatorname{Hess}(u)$ terms. However, $\operatorname{Hess}(u)$ blows up at the boundary which causes the numerical estimation scheme to fail. As such, we factorise these terms out of the expressions.  
To discuss the curvature of K\"ahler metrics \eqref{toricKahlermetric}, we choose holomorphic 
coordinates $z_i = \xi_i + i\theta_i$ on  $M^0$ so that
$T_{\mathbb{C}}M = \operatorname{span}_{\mathbb{C}} \left\{\frac{\partial}{\partial z_1}, \frac{\partial}{\partial z_2}, \frac{\partial}{\partial\bar z_1}, \frac{\partial}{\partial \bar z_2}\right\}$ (see \cite{Abreu} for more details).
For instance, \eqref{toricKahlermetric} extended complex linearly to $T_{\mathbb{C}}^*M\otimes T_{\mathbb{C}}^*M$ is given by 

\begin{equation}
    g= \sum_{\alpha,\beta=1}^2\frac{1}{2}u^{\alpha\beta}( d z^{\alpha}\otimes d\bar z^{\beta} +  d\bar z^\beta\otimes d z^{\alpha})
\end{equation}

(c.f. \cite{Doran2008}). Denote by 
\begin{align}
    \label{eq:Cmplx_projection_operators}
    S^{\alpha\ldots}_{\ldots} = \frac{1}{2}\left(S^{a\ldots}_{\ldots} - iJ^{a}_jS^{j\ldots}_{\ldots} \right), \quad  S^{\bar \alpha \ldots}_{\ldots} = \frac{1}{2}\left(S^{a\ldots}_{\ldots} + iJ^{a}_{j}S^{j\ldots}_{\ldots}\right), \\
    S^{\ldots}_{\alpha\ldots} = \frac{1}{2}\left(S^{\ldots}_{a\ldots} - iJ^{j}_aS^{\ldots}_{j\ldots} \right), \quad S^{ \ldots}_{\bar\alpha\ldots} = \frac{1}{2}\left(S^{\ldots}_{a\ldots} + iJ^{j}_{a}S^{\ldots}_{j\ldots}\right),
\end{align}

the contravariant and covariant projections of the tensor field $S^{\dots}_{\dots}$ onto the holomorphic and anti-holomorphic sub-bundles (see for instance \cite{Joyce}). We adopt the coordinate convention
\begin{equation}
    R\left(\frac{\partial}{\partial z_j},\frac{\partial}{\partial \bar z_k}\right)\frac{\partial}{\partial z_l} = R_{j\bar k l}{}^{\mu} \frac{\partial}{\partial z_{\mu}}
\end{equation}
so that the  Riemannian curvature tensor $R_{bcd}{}^{a}$ as a smooth section of $\wedge^2 T_{\mathbb{C}}M \otimes End(T_{\mathbb{C}} M)$ is given by

\begin{equation}
    \label{eq:Riem curvature}
    R_{bcd}{}^a = R_{\beta \bar \gamma \delta}{}^{\alpha} + R_{\bar \beta \gamma  \delta}{}^{\alpha} + R_{ \beta \bar \gamma \bar \delta}{}^{\bar \alpha} + R_{\bar \beta \gamma \bar \delta}{}^{\bar\alpha}
\end{equation}

(c.f. \cite{Joyce}[Section 4.5]). We also denote by $\zeta= x_i + \theta_i$ the symplectic coordinates on the polytope. Hence, for $u: P \rightarrow \R$, 

\begin{equation}
    \frac{\partial u }{\partial z_i} = \frac{1}{2}\frac{\partial}{\partial \xi_i}u = \frac{u^{ij}}{2}\frac{\partial u}{\partial x_j}
\end{equation}

since $\frac{\partial}{\partial \xi_i} = \frac{u^{ij}}{\partial x_j}$ \cite{Doran2008}.

\begin{lemma}
\label{lemma:Riem_complex_coordiante_formula}
    Let $u^{ij}_{kl} = \frac{\partial^2}{\partial x_k\partial x_l}u^{ij}$ where $u^{ij}$ is the inverse of the euclidean Hessian of $u: P \rightarrow \R$. Then component $R_{\beta \bar \gamma \delta}{}^{\alpha}$  of the Riemann curvature tensor is given by 

    \begin{equation}
        \label{eq:Riemannian_coordinates}
        R = -\frac{1}{4}u^{\gamma \tau}u^{\delta \beta}_{\tau \alpha} dz^{\beta} \wedge d\bar z^{\gamma}\otimes  dz^{\delta}\otimes  \frac{\partial}{\partial z_{\alpha}}.
    \end{equation}

    Denote by $\nabla$ the complex linear extension of the Levi-Civita connection to $T_{\C}M$ and $\nabla_{\bar\epsilon}$ as the natural contraction $\frac{\partial}{\partial \bar{z}_{\epsilon}}\cdot \nabla$. Then the first covariant derivatives are given by

    \begin{equation}
        \nabla_{\bar\epsilon}R_{\alpha}{}^{\beta}{}_{\gamma}{}^\mu = -\frac{1}{4}u^{\epsilon \lambda}u^{\gamma \alpha}_{\beta \mu \lambda}
    \end{equation}
    and 
    \begin{equation}
    \nabla_{\epsilon}R_{\alpha}{}^{\beta}{}_{\gamma}{}^{\mu}{} =-\frac{1}{4}(u^{\epsilon\lambda}u^{\gamma\alpha}_{\beta\mu\lambda}+u^{\gamma\alpha}_{\lambda\mu}u^{\epsilon\lambda}_\beta-u^{\lambda\alpha}_{\beta\mu}u^{\epsilon\gamma}_\lambda+u^{\gamma\alpha}_{\beta\lambda}u^{\epsilon\lambda}_\mu-u^{\gamma\lambda}_{\beta\mu}u^{\epsilon\alpha}_\lambda).
\end{equation}

Finally, the relevant components of the second covariant derivative are given by

\begin{equation}
        \nabla_{\bar \omega}\nabla_{\epsilon}R_{\alpha}{}^{\beta}{}_{\gamma}{}^{\mu} = -\frac{1}{8}u^{\omega \eta}(u^{\epsilon\lambda}u^{\gamma\alpha}_{\beta\mu\lambda}+u^{\gamma\alpha}_{\lambda\mu}u^{\epsilon\lambda}_\beta-u^{\lambda\alpha}_{\beta\mu}u^{\epsilon\gamma}_\lambda+u^{\gamma\alpha}_{\beta\lambda}u^{\epsilon\lambda}_\mu-u^{\gamma\lambda}_{\beta\mu}u^{\epsilon\alpha}_\lambda)_\eta
\end{equation}

and 
\begin{equation}
    \nabla_{ \omega}\nabla^{\epsilon} R^{\alpha}{}_{\beta}{}^{\gamma}{}_{\mu}{} = -\frac{1}{4}(u^{\omega\tau}u^{\mu\beta}_{\alpha\gamma\epsilon\tau} - u^{\tau \beta}_{\alpha\gamma\epsilon}u^{\omega\mu}_{\tau} +u^{\mu\beta}_{\tau\gamma\epsilon}u^{\omega\tau}_{\alpha} - u^{\mu\tau}_{\alpha\gamma\epsilon}u^{\omega\beta}_{\tau} + u^{\mu\beta}_{\alpha\tau\epsilon}u^{\omega\tau}_{\gamma}+u^{\mu\beta}_{\alpha\gamma\tau}u^{\omega\tau}_{\epsilon})
\end{equation}

where the subscript $\eta$ denotes $\frac{\partial}{\partial x_\eta}$.

\end{lemma}


\begin{proof}
     In the coordinates defined on $M^0$, $R_{\beta \bar \gamma \delta}{}^{\alpha} = -\bar \partial_\gamma \Gamma^{\alpha}_{\beta \delta}$  where $\Gamma^{\alpha}_{\beta\delta} = g^{\bar \nu \alpha}\frac{\partial g_{\delta \bar \nu}}{\partial z_\beta}$ (c.f. \cite{Ballmann2006}). Since $g_{\alpha \bar \beta} = \frac{1}{2}u^{a b}$,

\begin{align}
    \Gamma^\alpha_{\beta\delta} &= g^{\bar \nu\alpha}\frac{\partial g_{\delta \bar \nu}}{\partial z_\beta} \\
    &= \frac{u^{\beta\lambda}}{2}\frac{\partial}{\partial x_\lambda}\frac{1}{2}u^{\delta \nu}2u_{\nu\alpha} \\
    &=-\frac{1}{2}u_{\nu\alpha}u^{\beta \lambda }u^{\delta \mu}u_{\mu \tau\lambda}u^{\tau\nu} \\
    &=-\frac{1}{2}u^{\beta\lambda}u_{\mu \alpha \lambda}u^{\delta\mu}
    =\frac{1}{2}u^{\delta\beta}_{\alpha}
\end{align}

where in the final step we commuted $u_{\mu \alpha \lambda} = u_{\mu\lambda \alpha}$ and used the symmetry $u^{\beta \lambda} = u^{\lambda \beta}$.
Hence,

\begin{align}
    R_{\beta \bar \gamma \delta}{}^{\alpha} &= -\frac{\partial \Gamma^{\alpha}_{\beta \delta}}{\partial \bar z^\gamma} \\
    &=-\frac{1}{2}\frac{\partial}{\partial\bar z^\gamma}\bigg(u^{\delta\beta}_{\alpha}\bigg) \\
    &=-\frac{u^{\gamma\tau}}{2^2}u^{\delta\beta}_{\alpha\tau} = -\frac{1}{4}u^{\gamma\tau}u^{\delta\beta}_{\tau \alpha}
\end{align}

where in the last step we again swapped the order of the Euclidean partial derivatives. Next, 

  \begin{equation}
        \nabla_{\bar \epsilon}R_{\alpha}{}^{\beta}{}_{\gamma}{}^{\mu} =\nabla_{\bar \epsilon}\left(-\frac{1}{2}u_{\beta\mu}^{\gamma\alpha} dz^{\alpha}\otimes \frac{\partial}{\partial z_{\beta}}\otimes dz^{\gamma} \otimes \frac{\partial}{\partial z_{\mu}}\right).
    \end{equation} 
    Recall that $\nabla_{\bar\alpha}dz^{\beta} = \nabla_{\alpha}dz^{\bar \beta} = \nabla_{\bar \alpha} \frac{\partial}{\partial z_{ \beta}} = \nabla_{\alpha}\frac{\partial}{\partial \bar z_{\beta}}=0$ and hence
    \begin{equation}
        \nabla_{\bar \epsilon}\left(-\frac{1}{2}u_{\beta\mu}^{\gamma\alpha} dz^{\alpha}\otimes \frac{\partial}{\partial z_{\beta}}\otimes dz^{\gamma} \otimes \frac{\partial}{\partial z_{\mu}}\right) = -\frac{1}{4}u^{\epsilon\lambda}u_{\beta\mu\lambda}^{\gamma\alpha}dz^{\alpha}\otimes \frac{\partial}{\partial z_{\beta}}\otimes dz^{\gamma} \otimes \frac{\partial}{\partial z_{\mu}}.
    \end{equation}

    To compute $\nabla_{\epsilon}R_{\alpha}{}^{\beta}{}_{\gamma}{}^{\mu}$, recall that $\nabla_{\epsilon}dz^{\alpha} = -\Gamma_{\epsilon \lambda}^{\alpha}dz^{\lambda} = -\sum_{\lambda}\frac{1}{2}u^{\lambda \epsilon}_\alpha dz^{\lambda}$ and $\nabla_{\epsilon}\frac{\partial}{\partial z_{\alpha}} = \Gamma^{\lambda}_{\epsilon\alpha}\frac{\partial}{\partial z_{\lambda}} = \sum_{\lambda} \frac{1}{2}u^{\alpha\epsilon}_{\lambda}\frac{\partial}{\partial z_{\lambda}}$ and apply the Leibniz rule. The formula for $\nabla_{\bar \omega}\nabla_{\epsilon}R^{\alpha}{}_{\beta}{}^{\gamma}{}_{\mu}$ follows immediately since $\nabla_{\epsilon}R^{\alpha}{}_{\beta}{}^{\gamma}{}_{\mu}$ has no anti-holomorphic indices. Finally, $\nabla_{\omega}\nabla^{\epsilon}R^{\alpha}{}_{\beta}{}^{\gamma}{}_{\mu}$ follows from the symmetry $R^{\alpha}{}_{\beta}{}^{\gamma}{}_{\mu} =R_{\beta}{}^{\alpha}{}_{\mu}{}^{\gamma}$. Then $\nabla^\epsilon R_{\beta}{}^{\alpha}{}_{\mu}{}^{\gamma} = g^{\epsilon\bar \lambda}\nabla_{\bar\lambda}R_{\beta}{}^{\alpha}{}_{\mu}{}^{\gamma} = -\frac{1}{2}u^{\mu\beta}_{\alpha\gamma\epsilon}$. The formula for $\nabla_{\omega}\nabla^{\epsilon}R^{\alpha}{}_{\beta}{}^{\gamma}{}_\mu$ follows from the application of the Leibniz rule.
\end{proof}

By taking the trace of \eqref{eq:Riem curvature} with respect to \ref{toricKahlermetric} yields the following formula for the Ricci endomorphism.

\begin{corollary}
    The Ricci endomorphism $R^{a}{}_{b}$  is given by 
    \begin{equation}
        \Ric^{\#} = -\frac{1}{2}u^{\beta \gamma}_{\alpha \gamma}\left(\frac{\partial}{\partial z^{\alpha}}\otimes dz^{\beta} + \frac{\partial}{\partial \bar z^{\alpha}}\otimes d\bar z^{\beta}\right)
    \end{equation}
    Since $R_{cb} = R_{\bar\lambda \gamma \bar\beta}{}^{\bar \lambda} + R_{\lambda\bar \gamma \beta}{}^{\lambda} $, and raising one index with $g^{ac}$ gives $R^{a}{}_{b} = R_{\bar \lambda}{}^{\bar \alpha}{}_{\bar \beta}{}^{\bar \lambda} + R_{\lambda}{}^{\alpha}{}_{\beta}{}^{\lambda}$ and $R_{\bar \lambda}{}^{\bar \alpha}{}_{\bar \beta}{}^{\bar \lambda} = R_{\lambda}{}^{\alpha}{}_{\beta}{}^{\lambda} = -\frac{1}{2}u^{\beta \lambda}_{\alpha \lambda}$ by \ref{lemma:Riem_complex_coordiante_formula}.
\end{corollary}

It is now straightforward to compute the required curvature bounds. For convenience, we collect them in the following lemma.

\begin{lemma}
   The norm of the Riemann curvature tensor is given by 
   \begin{align*}
       |\Riem|_g^2= \sum_{i,j,k,l=1}^2 u^{ij}_{kl}u^{kl}_{ij}. 
   \end{align*}
   The norm of the first covariant derivative is given by
   \begin{align*}
       |\nabla \Riem|_g^2 = \sum_{i,j,k,l,n,m=1}^2u^{jl}_{ikm}(u^{mn}u^{ik}_{jln}+u^{ik}_{nl}u^{mn}_j-u^{nk}_{jl}u^{mi}_n+u^{ik}_{jn}u^{mn}_l-u^{in}_{jl}u^{mk}_n).
   \end{align*}
   For the second covariant derivative, set
    \[
    \begin{aligned}
    B_b^{ik}{}_{jl}=\sum_{m=1}^2\big(
    &u^{bm}u^{ik}_{jlm}
    +u^{bm}_j u^{ik}_{ml}-u^{bi}_m u^{mk}_{jl}\\
    &+u^{bm}_l u^{ik}_{jm}-u^{bk}_m u^{im}_{jl}\big).
    \end{aligned}
    \]
    Then
    \[
    \begin{aligned}
    |\nabla^2\Riem|_g^2
    ={}&\frac12\sum_{a,b,i,j,k,l,m=1}^2
    u^{jl}_{ikab}\\
    &\quad\cdot\big(
    u^{am}\partial_m B_b^{ik}{}_{jl}
    +u^{am}_j B_b^{ik}{}_{ml}
    -u^{ai}_m B_b^{mk}{}_{jl}\\
    &\qquad
    +u^{am}_l B_b^{ik}{}_{jm}
    -u^{ak}_m B_b^{im}{}_{jl}
    -u^{ab}_m B_m^{ik}{}_{jl}\big)\\
    &+\frac12\sum_{a,b,i,j,k,l,m=1}^2
    (\partial_a B_b^{ik}{}_{jl})\\
    &\quad\cdot\big(
    u^{am}u^{jl}_{ikbm}
    +u^{am}_i u^{jl}_{mkb}
    -u^{aj}_m u^{ml}_{ikb}\\
    &\qquad
    +u^{am}_k u^{jl}_{imb}
    -u^{al}_m u^{jm}_{ikb}
    +u^{am}_b u^{jl}_{ikm}\big),
    \end{aligned}
    \]
    where $\partial_a=\partial/\partial x_a$, and subscripts on
    $u^{ij}$ denote ordinary derivatives.
\end{lemma}
\begin{proof}

Recall that $|\Ric|^2_g= g^{a_1b_1}g^{a_2b_2}R_{a_1a_2}\overline{R_{b_1b_2}} $ and $\overline{R_{b_1b_2}} = \overline{R^{\lambda}{}_{b_1 \lambda \bar{b_2}} + R^{\bar \lambda}{}_{\bar b_1 \bar \lambda b_2}} = {R^{\lambda}{}_{b_1 \lambda \bar{b_2}} + R^{\bar \lambda}{}_{\bar b_1 \bar \lambda b_2}}$. Hence

\begin{align*}
    |\Ric|^2_g &= g^{a_1b_1}g^{a_2b_2}R_{a_1a_2}{R_{b_1b_2}} \\ 
    &= R^{b_1}{}_{a_2}R_{b_1}{}^{a_2} \\ 
    &= R^{\beta_1}{}_{\alpha_2}R_{\beta_1}{}^{\alpha_2} + R^{\bar \beta_1}{}_{\bar \alpha_2}R_{\bar \beta_1}{}^{\bar \alpha_2} \\ 
    &= 2R^{\beta}{}_{\alpha}R_{\beta}{}^{\alpha} = \frac{1}{2}u^{ak}_{bk}u^{bl}_{al}
\end{align*}
    
The Riemannian curvature formula is then given by 
\begin{align*}
    |\Riem|^2 &= g^{a_1b_1}g^{a_2b_2}g^{a_3b_3}g^{a_4b_4}R_{a_1a_2a_3a_4}{R_{b_1b_2b_3b_4}} \\
     &= R^{b_1}{}_{a_2}{}^{b_3}{}_{a_4} R_{b_1}{}^{a_2}{}_{b_3}{}^{a_4} \\
     &= R^{\beta_1^{\pm_1}}{}_{\alpha_2^{\pm_1}}{}^{\beta_3^\pm}{}_{\alpha_4^\pm} R_{\beta_1^{\pm_1}}{}^{\alpha_2^{\pm_1}}{}_{\beta_3^\pm}{}^{\alpha_4^{\pm}} \\
     &= 4 R_{\alpha_2}{}^{\beta_1}{}_{\alpha_4}{}^{\beta_3}R_{\beta_1}{}^{\alpha_2}{}_{\beta_3}{}^{\alpha_4} \\&= u^{a_4a_2}_{b_1b_3}u^{b_3b_1}_{a_2a_4} = u^{ij}_{kl}u^{kl}_{ij}
\end{align*}
where ${\pm}$ denotes barred and unbarred indices respectively from the projections \ref{eq:Cmplx_projection_operators} and the curvature formula \eqref{eq:Riem curvature}. 

For the covariant derivatives, we note that  $\overline{\nabla_{\epsilon}R_{\alpha\bar\beta\gamma}{}^{\mu}} = \nabla_{\bar \epsilon}R_{\bar \alpha \beta \bar \gamma}{}^{\bar \mu}$.

Using the formula \eqref{eq:Riem curvature} and the projections \ref{eq:Cmplx_projection_operators}, 
\begin{align*}
    |\nabla \Riem|^2 &= \nabla_eR_{a}{}^{b}{}_{c}{}^{d}\nabla^eR^a{}_{b}{}^{c}{}_d \\
&= \nabla_e R_{\alpha^{\pm_1}}{}^{\beta^{\pm_1}}{}_{\gamma^{\pm}}{}^{\delta^\pm} \nabla^eR^{\alpha^{\pm_1}}{}_{\beta^{\pm_1}}{}^{\gamma^{\pm}}{}_{\delta^\pm} \\      
&= \nabla_{\epsilon^{\pm_2}} R_{\alpha^{\pm_1}}{}^{\beta^{\pm_1}}{}_{\gamma^{\pm}}{}^{\delta^\pm} \nabla^{\epsilon^{\pm_2}}R^{\alpha^{\pm_1}}{}_{\beta^{\pm_1}}{}^{\gamma^{\pm}}{}_{\delta^\pm}  \\
\end{align*}
and the formula follows by observing that the $2^3$ summands are equal. Finally, the curvature symmetries and complex conjugation give
\[
\begin{aligned}
|\nabla^2\Riem|_g^2
={}&8\sum_{\omega,\epsilon,\alpha,\beta,\gamma,\mu=1}^2
(\nabla_\omega\nabla_\epsilon
R_\alpha{}^\beta{}_\gamma{}^\mu)
(\nabla^\omega\nabla^\epsilon
R^\alpha{}_\beta{}^\gamma{}_\mu)\\
&+8\sum_{\omega,\epsilon,\alpha,\beta,\gamma,\mu=1}^2
(\nabla_{\bar\omega}\nabla_\epsilon
R_\alpha{}^\beta{}_\gamma{}^\mu)
(\nabla^{\bar\omega}\nabla^\epsilon
R^\alpha{}_\beta{}^\gamma{}_\mu).
\end{aligned}
\]
Here the iterated covariant derivatives denote components of
the covariant Hessian, including the connection on the
$\epsilon$ index. The two sums are the pure and mixed
derivative contributions, respectively. Substituting the
component formulae and applying the Leibniz rule gives the
stated expression. The connection on the derivative index
contributes $-u^{ab}_m B_m^{ik}{}_{jl}$ in the first sum and
$+u^{am}_b u^{jl}_{ikm}$ in the second.
\end{proof}

\section{Elementary Estimates from Geometric Analysis}\label{GAE}
The purpose of this Appendix is to detail a number of basic estimates in geometric analysis that are crucial for our argument. The fact that such inequalities hold in some form will not be surprising to experts in PDE, but we emphasise here that for our purposes, it is essential that the inequality constants are known \textit{explicitly}, and it is preferable if they are as close to optimal as practical. 
Throughout this subsection, we will let $(M,g)$ be a closed Riemannian manifold of dimension $n\ge 2$, with volume $V$ and Ricci curvature satisfying the positive lower bound $\Ric(u,u) \geq \mu g(u,u)$ for all $u\in TM$, and some $\mu\in \mathbb{R}^+$. We use the norm convention $\|{u}_{W^{k,p}}=\|{u}_{L^{p}}+\sum_{i=1}^{k}\|{\nabla^i u}_{L^{p}}$, for any positive integer $k$.

\subsection{Sobolev embedding theorems}
For $1\le q<n$, we define the following constants (adapted from \cite{Aubin} and \cite{Ilias1983}), which use the $\Gamma$ function and the volume  $\omega_n$ of the unit-radius round sphere $\mathbb{S}^n$:
\begin{align*}
    K(n,q)
    =
    \frac{q-1}{n-q}
    \left(\frac{n-q}{n(q-1)}\right)^{1/q}
    \left[
    \frac{q\,\Gamma(n+1)}
    {(n-q)\,\Gamma\!\left(\frac{n}{q}-1\right)\,\Gamma\!\left(n+1-\frac{n}{q}\right)\,\omega_{n-1}}
    \right]^{1/n}
    \text{ for }
    1<q<n,\\
    \Cl{cost:said-isoperimetric-const}= \lim_{q \rightarrow 1^+}
    K(n,q)^{-1}=n \left( \frac{\omega_{n-1}}{n} \right)^{1/n},  \qquad \Cl{const:said-isoperimetric-const-of-M} V^{-1/n}
    =
    2^{-1/n}
    \sqrt{\frac{\mu}{n-1}}
    \left(
    \int_0^{\frac{\pi}{2}}
    (\cos t)^{n-1}
    dt
    \right)^{-1}.
\end{align*}
We obtain Morrey's inequality. 
\begin{theorem}[Théorème 7 in \cite{Ilias1983}]
   For $q >n$ and any $u \in W^q_1(M)$ we have:
    \[
    \|{u}_{L^\infty}
    \leq
    \left(
    \frac{q-1}{q-n}
    \right)^{1-1/q}
    \left(
    \frac{\Cr{cost:said-isoperimetric-const}}{\Cr{const:said-isoperimetric-const-of-M}}
    \right)
    \frac{V^{1/n-1/q}}{n^{1/q-1/n} \omega^{1/n}_{n-1}}
    \|{\nabla u}_{L^q}
    +
    2 V^{-1/q} \|{u}_{L^q}.
    \]
\end{theorem}
We also obtain a general Sobolev inequality. 

\begin{theorem}\label{SETp}

    Let $2 \le  q < n$, $p=\frac{nq}{n-q}$, and $\alpha=\frac{q(n-2)}{2(n-q)}=\frac{p}{2^*}$. 
    Then, for $u \in W^q_1(M)$ we have
    \[
    \|{u}_{L^p}
    \leq
    \frac{K(n,2)\sqrt{n-1}}{\sqrt{\mu}} \cdot
    \left(
    \frac{\omega_n}{V}
    \right)^{1/n}
    \cdot
    \alpha
    \cdot
    \|{\nabla u}_{L^q}
    +
    V^{-1/n} \|{u}_{L^q}. 
    \]
\end{theorem}

\begin{proof}
For any smooth $u:M\to \mathbb{R}$, the function $v=|u|^{\alpha}$ is in $L_{1}^2(M)$, so by writing $2^*=\frac{2n}{n-2}$ and using \cite[Théorème 3]{Ilias1983}, we obtain 
    \[
    \|{v}^2_{L^{2^*}}
    \leq
    \frac{K(n,2)^2(n-1)}{\mu}
    \left(
   \frac{\omega_n}{V}
    \right)^{2/n}
    \|{\nabla v}^2_{L^2}
    +
    V^{-2/n} \|{v}^2_{L^2}.
    \] 
    Taking square roots and using $\sqrt{a+b} \leq \sqrt{a}+\sqrt{b}$ for positive numbers $a$ and $b$ we get
    \[
    \|{v}_{L^{2^*}}
    \leq
    \frac{K(n,2)\sqrt{n-1}}{\sqrt{\mu}}
    \left(
    \frac{\omega_n}{V}
    \right)^{1/n}
    \|{\nabla v}_{L^2}
    +
    V^{-1/n} \|{v}_{L^2}.
    \]
   Since $|\nabla v|=\alpha |u|^{\alpha-1}|\nabla u|$ almost everywhere, we obtain 

    \begin{align}\label{penultimateSobolevinequality}
        \|{u}_{L^p}^\alpha=\|{v}_{L^{2^*}}&\leq 
        \frac{K(n,2)\sqrt{n-1}}{\sqrt{\mu}}
        \left(
        \frac{\omega_n}{V}
        \right)^{1/n}
        \cdot \alpha \cdot
        \|{ \, |u|^{\alpha-1} \nabla u \,}_{L^2}+V^{-1/n} \|{u}^\alpha_{L^{2\alpha}}. 
    \end{align}
    If $q=2$, then $p=2^*$ and $\alpha=1$, so the claimed estimate is
    precisely the preceding $W^{1,2}(M)\to L^{2^*}(M)$ estimate. We may
    therefore assume that $2<q<n$.
    Using Hölder's inequality with indices $\frac{p}{2(\alpha-1)}$ and $\frac{q}{2}$ (which is permissible since $(\alpha-1)\frac{2q}{q-2}=p$), we obtain 
    \begin{align*}
        \|{ \, |u|^{\alpha-1} \nabla u \,}_{L^2}^2=\|{ \, |u|^{2(\alpha-1)} |\nabla u|^2 \,}_{L^1}\leq
        \|{u}_{L^p}^{2(\alpha-1)} \cdot \|{\nabla u}_{L^q}^2
    \end{align*} and
    \begin{align*}
        \|{u}_{L^{2\alpha}}^{2 \alpha} =\|{|u|^{2\alpha-2}\cdot |u|^{2}}_{L^{1}} \leq
        \|{u}_{L^p}^{2(\alpha-1)} \cdot \|{u}_{L^q}^2;
    \end{align*}putting these into \eqref{penultimateSobolevinequality} gives 
    \begin{align*}
      \|{u}_{L^p}^\alpha  \leq
        \frac{K(n,2)\sqrt{n-1}}{\sqrt{\mu}}
        \left(
        \frac{\omega_n}{V}
        \right)^{1/n}
        \cdot \alpha \cdot
        \|{u}_{L^p}^{\alpha-1} \cdot \|{\nabla u}_{L^q}
        +V^{-1/n} \|{u}_{L^p}^{\alpha-1} \|{u}_{L^q}. 
    \end{align*}
    If $u=0$, then the required estimate holds trivially, and if $u \neq 0$, then we can divide both sides by $\|{u}_{L^p}^{\alpha-1}$ and obtain the claim. 
\end{proof}

For the rest of this subsection, we specialise to the case $n=4$. 
To simplify notation in this article, we also use less sharp versions of the above inequalities, that use the constants 
\begin{align*}
\Cl{const:emb-C1}=
\max \left\{ \sqrt{3}K(4,2) \cdot
    \left(
   \frac{\omega_4}{V\mu^{2}}
    \right)^{1/4}
    ,
    V^{-1/4}
    \right\}, 
    \qquad 
    \Cl{const:emb-C3}= 
    \max \left\{ 3\sqrt{3}\cdot K(4,2) \cdot
    \left(
    \frac{\omega_4}{V\mu^{2}}
    \right)^{1/4}
    ,
    V^{-1/4} \right\},\\
     \Cl{const:emb-C4}=\max\{\left(
    \frac{11}{8}
    \right)^{1-1/12}
    \left(
    \frac{\Cr{cost:said-isoperimetric-const}}{\Cr{const:said-isoperimetric-const-of-M}}
    \right)
    \frac{(4V)^{1/6}}{ \omega^{1/4}_{3}}
    ,
    2 V^{-1/12}\}.
\end{align*}

\begin{corollary}
\label{corollary:explicit-embeddings}
   If $n=4$, then for all functions $u$ in the appropriate function spaces:
    \begin{align*}
        \|{u}_{L^4}
        &\leq
        \Cr{const:emb-C1} \|{u}_{H^1},
        \\
        \|{u}_{L^{12}}
        &\leq
        \Cr{const:emb-C3} \|{u}_{W^{1,3}},
        \\
        \|{u}_{L^\infty}
        &\leq
        \Cr{const:emb-C4} \|{u}_{W^{1,12}}.
    \end{align*}
\end{corollary}

Also for later use we record the embedding constant of $L^4$ into $L^3$.
The statement is a direct consequence of Hölder's inequality.

\begin{proposition}
\label{proposition:L4-L3-embedding}
    Let $M$ be a manifold with volume $V$.
    Then:
    \[
    \|{u}_{L^3}
    \leq
    \Cl{const:emb-C2}
    \|{u}_{L^4}
    \]
    for $\Cr{const:emb-C2}=V^{1/12}$.
\end{proposition}

This then implies the following Sobolev multiplication theorem with an explicit constant:

\begin{corollary}
    \label{corollary:sobolev-multiplication-thm}
    Let $M$ be closed and let $u, v \in H^3(M)$.
    Then:
    \[
    \|{uv}_{H^3}
    \leq
    \Cl{const:multiplication-theorem}
    \|{u}_{H^3}
    \|{v}_{H^3},
    \]
    where $\Cr{const:multiplication-theorem}$ is defined in \cref{equation:multiplication-theorem-constant-def}.
\end{corollary}

\begin{proof}
    We have
    \begin{align*}
        \|{\nabla^3(uv)}_{L^2}
        &\leq
        \|{(\nabla^3 u)v}_{L^2}
        +
        3\|{(\nabla^2 u)\otimes\nabla v}_{L^2}
        +
        3\|{(\nabla u)\otimes\nabla^2 v}_{L^2}
        +
        \|{u \nabla^3 v}_{L^2}
        =:
        I+II+III+IV.
    \end{align*}
    Note that
    \begin{align}
    \label{equation:L-infty-embedding}
    \begin{split}
        \|{v}_{L^\infty}
        &\leq 
        \Cr{const:emb-C4}
        \|{v}_{W^{1,12}}
        =
        \Cr{const:emb-C4}
        (\|{v}_{L^{12}}
        +
        \|{\, |\nabla v| \,}_{L^{12}}
        )
        \leq
        \Cr{const:emb-C4}
        \Cr{const:emb-C3}
        (\|{v}_{W^{1,3}}
        +
        \|{\, |\nabla v| \,}_{W^{1,3}}
        )
        \\
        &
        \leq
        2\Cr{const:emb-C4}
        \Cr{const:emb-C3}
        \|{v}_{L^3_2}
        \leq 2
        \Cr{const:emb-C4}
        \Cr{const:emb-C3}
        \Cr{const:emb-C2}
        \|{v}_{L^4_2}
        \leq 4
        \Cr{const:emb-C4}
        \Cr{const:emb-C3}
        \Cr{const:emb-C2}
        \Cr{const:emb-C1}
        \|{v}_{H^3},
    \end{split}
    \end{align}
    where in the first and third step we used \cref{corollary:explicit-embeddings};
    in the fourth step we used Kato's inequality $|\nabla | \nabla v|| \leq |\nabla^2 v|$;
    in the fifth step we used \cref{proposition:L4-L3-embedding};
    in the last step we again combined \cref{corollary:explicit-embeddings} with Kato's inequality like before.

    Hence:
    \[
    I
    \leq
    \|{\nabla^3 u}_{L^2}
    \|{v}_{L^\infty}
    \leq 4
    \Cr{const:emb-C4}
    \Cr{const:emb-C3}
    \Cr{const:emb-C2}
    \Cr{const:emb-C1}
    \|{\nabla^3 u}_{L^2}
    \|{v}_{H^3}.
    \]
    The summand $IV$ is treated identically, while $II$ is estimated with $\|{(\nabla^2 u)\otimes \nabla v}_{L^2}\le \|{(\nabla^2 u)}_{L^4}\|{\nabla v}_{L^4}$ and \cref{corollary:explicit-embeddings}, and similarly for $III$, which all together gives:
    \begin{align*}
        \|{\nabla^3(uv)}_{L^2}
        &\leq 4
        \Cr{const:emb-C4}
        \Cr{const:emb-C3}
        \Cr{const:emb-C2}
        \Cr{const:emb-C1}
        \|{\nabla^3 u}_{L^2}
        \|{v}_{H^3}
        +
        3\Cr{const:emb-C1}^2
        \|{\nabla^2 u}_{H^1}
        \|{\nabla v}_{H^1}
        \\
        &
        \quad
        +
        3\Cr{const:emb-C1}^2
        \|{\nabla u}_{H^1}
        \|{\nabla^2 v}_{H^1}
        +
        4\Cr{const:emb-C4}
        \Cr{const:emb-C3}
        \Cr{const:emb-C2}
        \Cr{const:emb-C1}
        \|{u}_{H^3}
        \|{\nabla^3 v}_{L^2}.
    \end{align*}
    Similarly,
    \begin{align*}
        \|{\nabla^2(uv)}_{L^2}
        &\leq
        \Cr{const:emb-C1}^2
        \|{\nabla^2 u}_{H^1}
        \|{v}_{H^1}
        +
        2\Cr{const:emb-C1}^2
        \|{\nabla u}_{H^1} \|{\nabla v}_{H^1}
        +
        \Cr{const:emb-C1}^2
        \|{u}_{H^1} \|{\nabla^2 v}_{H^1},
        \\
        \|{\nabla(uv)}_{L^2}
        &\leq
        \Cr{const:emb-C1}^2 \|{\nabla u}_{H^1} \|{v}_{H^1}
        +
        \Cr{const:emb-C1}^2 \|{u}_{H^1} \|{\nabla v}_{H^1},
        \\
        \|{uv}_{L^2}
        &\leq
        \Cr{const:emb-C1}^2 \|{u}_{H^1} \|{v}_{H^1}.
    \end{align*}
    We sum up the four right hand sides and record the coefficients of the different products $\|{\nabla^k u}_{L^2} \|{\nabla^l v}_{L^2}$ in \cref{table:quadratic-estimate-coefficients}.
    \begin{table}[htbp]
        \centering
        \begin{tabular}{ccccc}
             &$u$&$\nabla u$&$\nabla^2 u$&$\nabla^3 u$
             \\
             \hline
             $v$&$\Cr{const:emb-C1}^2$&$2\Cr{const:emb-C1}^2$&$2\Cr{const:emb-C1}^2$&$4\Cr{const:emb-C4}
            \Cr{const:emb-C3}
            \Cr{const:emb-C2}
            \Cr{const:emb-C1}+\Cr{const:emb-C1}^2$
             \\
             $\nabla v$&$2\Cr{const:emb-C1}^2$&$5\Cr{const:emb-C1}^2$&$7\Cr{const:emb-C1}^2$&$4\Cr{const:emb-C4}
            \Cr{const:emb-C3}
            \Cr{const:emb-C2}
            \Cr{const:emb-C1}+4\Cr{const:emb-C1}^2$
             \\
             $\nabla^2 v$&$2\Cr{const:emb-C1}^2$&$7\Cr{const:emb-C1}^2$&$8\Cr{const:emb-C1}^2$&$4\Cr{const:emb-C4}
            \Cr{const:emb-C3}
            \Cr{const:emb-C2}
            \Cr{const:emb-C1}+3\Cr{const:emb-C1}^2$
             \\
             $\nabla^3 v$&$4\Cr{const:emb-C4}
            \Cr{const:emb-C3}
            \Cr{const:emb-C2}
            \Cr{const:emb-C1}+\Cr{const:emb-C1}^2$&$4\Cr{const:emb-C4}
            \Cr{const:emb-C3}
            \Cr{const:emb-C2}
            \Cr{const:emb-C1}+4\Cr{const:emb-C1}^2$&$4\Cr{const:emb-C4}
            \Cr{const:emb-C3}
            \Cr{const:emb-C2}
            \Cr{const:emb-C1}+3\Cr{const:emb-C1}^2$&$8\Cr{const:emb-C4}
            \Cr{const:emb-C3}
            \Cr{const:emb-C2}
            \Cr{const:emb-C1}$
        \end{tabular}
        \caption{Coefficients of the different products $\|{\nabla^k u}_{L^2} \|{\nabla^l v}_{L^2}$ in the proof of \cref{corollary:sobolev-multiplication-thm}.}
        \label{table:quadratic-estimate-coefficients}
    \end{table}
    
    The value of the constant $\Cr{const:multiplication-theorem}$ is then the largest entry of that table, i.e.
    \begin{align}
        \label{equation:multiplication-theorem-constant-def}
        \Cr{const:multiplication-theorem}
        =
        \max \{
        8\Cr{const:emb-C4}
        \Cr{const:emb-C3}
        \Cr{const:emb-C2}
        \Cr{const:emb-C1},
        \,4
        \Cr{const:emb-C4}
        \Cr{const:emb-C3}
        \Cr{const:emb-C2}
        \Cr{const:emb-C1}+4\Cr{const:emb-C1}^2,
        \,
        8\Cr{const:emb-C1}^2
        \}.
    \end{align}
\end{proof}

\begin{corollary}
    \label{corollary:sobolev-low-order-multiplication}
    Let $M$ be closed and define
    \begin{align*}
        \Cl{const:sobolev-low-order-multiplication-M0}
        &=
        2\Cr{const:emb-C4}
        \Cr{const:emb-C3}
        \Cr{const:emb-C2}
        \Cr{const:emb-C1},
        \quad
        \Cl{const:sobolev-low-order-multiplication-M1}
        =
        2\Cr{const:emb-C4}
        \Cr{const:emb-C3}
        \Cr{const:emb-C2}
        \Cr{const:emb-C1}
        +
        \Cr{const:emb-C1}^2,
        \quad
        \Cl{const:sobolev-low-order-multiplication-M2}
        =
        2\Cr{const:emb-C4}
        \Cr{const:emb-C3}
        \Cr{const:emb-C2}
        \Cr{const:emb-C1}
        +
        4\Cr{const:emb-C1}^2.
    \end{align*}
    If $u\in H^3(M)$, then
    \begin{align*}
        \|{uv}_{L^2}
        &\leq
        \Cr{const:sobolev-low-order-multiplication-M0}
        \|{u}_{H^3}
        \|{v}_{L^2},
        \\
        \|{uv}_{H^1}
        &\leq
        \Cr{const:sobolev-low-order-multiplication-M1}
        \|{u}_{H^3}
        \|{v}_{H^1},
        \\
        \|{uv}_{H^2}
        &\leq
        \Cr{const:sobolev-low-order-multiplication-M2}
         \|{u}_{H^3}
        \|{v}_{H^2},
    \end{align*}
    for $v\in L^2(M),W^2_1(M),W^2_2(M)$, respectively.
\end{corollary}

\begin{proof}
Hölder's inequality together with \cref{corollary:explicit-embeddings} and Kato's inequality gives, for tensor fields $S,T$,
\[
    \|{S\otimes T}_{L^2}
    \leq
    \Cr{const:emb-C1}^2
    \|{S}_{H^1}\|{T}_{H^1}.
\]
For $1\leq i\leq k\leq r\leq 2$, this implies
\[
    \|{\nabla^i u\otimes\nabla^{k-i}v}_{L^2}
    \leq
    \Cr{const:emb-C1}^2
    \|{u}_{H^3}\|{v}_{H^r}.
\]
We thus have
\begin{align*}
    \|{uv}_{H^r}
    &=
    \sum_{k=0}^r
    \|{\nabla^k(uv)}_{L^2}
    \\
    &\leq
    \sum_{k=0}^r
    \sum_{i=0}^k
    \binom{k}{i}
    \|{\nabla^i u\otimes\nabla^{k-i}v}_{L^2}
    \\
    &=
    \sum_{k=0}^r
    \|{u\nabla^k v}_{L^2}
    +
    \sum_{k=1}^r
    \sum_{i=1}^k
    \binom{k}{i}
    \|{\nabla^i u\otimes\nabla^{k-i}v}_{L^2}
    \\
    &\leq
    \|{u}_{L^\infty}
    \|{v}_{H^r}
    +
    \Cr{const:emb-C1}^2
    \|{u}_{H^3}
    \|{v}_{H^r}
    \sum_{k=1}^r
    \sum_{i=1}^k
    \binom{k}{i}
    \\
    &\leq
    2\Cr{const:emb-C4}
    \Cr{const:emb-C3}
    \Cr{const:emb-C2}
    \Cr{const:emb-C1}
    \|{u}_{H^3}
    \|{v}_{H^r}
    +
    \Cr{const:emb-C1}^2
    \|{u}_{H^3}
    \|{v}_{H^r}
    \sum_{k=1}^r
    \sum_{i=1}^k
    \binom{k}{i}
    \\
    &=
    \left(
    2\Cr{const:emb-C4}
    \Cr{const:emb-C3}
    \Cr{const:emb-C2}
    \Cr{const:emb-C1}
    +
    (2^{r+1}-r-2)
    \Cr{const:emb-C1}^2
    \right)
    \|{u}_{H^3}
    \|{v}_{H^r},
\end{align*}
where we used \cref{equation:L-infty-embedding} in the fifth step, and we used $\sum_{k=1}^r
    \sum_{i=1}^k
    \binom{k}{i}=2^{r+1}-r-2$ in the last step.
Taking $r=0,1,2$ proves the three claims.
\end{proof}

\subsection{Elliptic estimates}
\label{subsection:elliptic-estimates}

In this subsection, we use Bochner's formula to estimate a tensor field $T$ in terms of its connection Laplacian $\Delta T$. 

\begin{proposition}
    Let $M$ be an $n$-dimensional manifold and let $E \rightarrow M$ be a vector bundle with a metric connection $\nabla$.
    Let $s$ be a section of $E$.
    Denote the connection Laplacian as $\Delta = \nabla^* \nabla$. Then $[\Delta,\nabla]s$ is a section of $T^* M\otimes E$ which satisfies the following for any $x \in M$ and $X \in T_x M$:
    \begin{align}
    \label{equation:Delta-commutator}
    [\Delta, \nabla]s(X)
    =
    -\nabla_{\Ric(X)}s
    -
    2\sum_{i=1}^n
    R(e_i, X) \nabla_{e_i} s
    -
    \sum_{i=1}^n
    (\nabla_{e_i}R)(e_i, X)s,
    \end{align}
    where $e_i$ is an orthonormal basis of $T_x M$ and $R(U,V)s=\nabla_U \nabla_V s-\nabla_V \nabla_U s-\nabla_{[U,V]} s$ denotes the curvature endomorphism of $E$.
\end{proposition}
\begin{proof}
    Let $e_i$ be a coordinate basis of normal coordinates centred at a point $x \in M$ and extend $X$ to have constant coefficients in the coordinate frame given by normal coordinates centred at $x$ so that $\nabla X=0$ and $\nabla e_i=0$ at $x$ and $[e_i,X]=0$ in a neighbourhood of $x$.
    If we let $\tilde{e}_i$ be an orthonormal basis, formed by applying Gram-Schmidt to $e_i$, then $\tilde{e}_i-e_i=0$ and $\nabla(\tilde{e}_i-e_i)=0$ at $x$, so,  
    \begin{align}
    \begin{split}
        \label{equation:basic-laplace-formula}
        (\nabla^* \nabla \sigma)
        &=
        -\text{trace}(\nabla^2 \sigma)\\
        &=-\sum_{i}(\nabla^2 \sigma (\tilde{e}_i,\tilde{e}_i))\\
        &=-\sum_{i}(\nabla^2 \sigma (e_i,e_i))+\mathcal{E}\\
        & =
        -\sum_i (\nabla_{e_i} \nabla_{e_i} \sigma-\nabla_{\nabla_{e_i} e_i} \sigma)+\mathcal{E}
        \end{split}
    \end{align}
    by \cite[Eqn. 8.3]{Lawson2016}. Here, $\mathcal{E}(x)=0$ and $\nabla \mathcal{E}(x)=0$.  We will use \eqref{equation:basic-laplace-formula} with $\sigma=\nabla s$ to compute $\Delta (\nabla s)$ at $x$. 
To this end, observe that the Leibniz rule for the induced connection $\nabla$ on $T^*M\otimes E$ gives 
    \begin{align*}
        (\nabla_{e_i} \nabla s)(X)
        &=\nabla_{e_i}(\nabla s(X))-\nabla s (\nabla_{e_i}X)\\
        &=\nabla_{e_i}\nabla_X s-\nabla_{\nabla_{e_i}X}s,
        \\
        (\nabla_{e_i} \nabla_{e_i} \nabla s)(X)
        &=\nabla_{e_i}((\nabla_{e_i} \nabla s)(X))-(\nabla_{e_i} \nabla s)(\nabla_{e_i}X)\\
        &=\nabla_{e_i}\nabla_{e_i}\nabla_X s-\nabla_{e_i}\nabla_{\nabla_{e_i}X}s-\nabla_{e_i}\nabla_{\nabla_{e_i}X} s+\nabla_{\nabla_{e_i}\nabla_{e_i}X}s\\
        &=\nabla_{e_i}\nabla_{e_i}\nabla_X s-2\left(\nabla_{\nabla_{e_i}\nabla_{e_i}X}s+\nabla_{\nabla_{e_i}X}\nabla_{e_i}s\right)+\nabla_{\nabla_{e_i}\nabla_{e_i}X}s\\
        &=\nabla_{e_i}\nabla_{e_i}\nabla_X s-\nabla_{\nabla_{e_i}\nabla_{e_i}X}s
        \end{align*}
at $x$ since $\nabla_{e_i}X=0$ here. 
  But since $\nabla_{e_i}e_i$ also vanishes at $x$, we can use the above expression and \cref{equation:basic-laplace-formula} with $\sigma=\nabla s$ to conclude
    \begin{align*}
        \Delta \nabla s (X)
        &=
        \left(
        -\sum_i (\nabla_{e_i} \nabla_{e_i} \nabla s-\nabla_{\nabla_{e_i} e_i} \nabla s)
        \right)(X)
        \\
        &=
        -\left(
        \sum_i (\nabla_{e_i} \nabla_{e_i} \nabla s)
        \right)(X)
        \\
       &=-\nabla_{e_i}\nabla_{e_i}\nabla_X s+\nabla_{\nabla_{e_i}\nabla_{e_i}X}s
    \end{align*}
at $x$.

    To compute $\nabla \Delta s$, we apply $\nabla$ to \cref{equation:basic-laplace-formula} and obtain:
    \begin{align*}
        \nabla_X \Delta s
        &=
        -\sum_i \nabla_X \nabla_{e_i} \nabla_{e_i} s
        +
        \sum_i \nabla_X \nabla_{\nabla_{e_i} e_i} s
        \\
        &=
        -\sum_i \nabla_X \nabla_{e_i} \nabla_{e_i} s
        +
        \sum_i
        \nabla_X (\nabla s) \underbrace{(\nabla_{e_i} e_i)}_{=0}+\nabla_{\nabla_X \nabla_{e_i} e_i} s
    \end{align*}
    by the Leibniz rule, which implies 
    \begin{align}\label{LBFF}
        ([\Delta,\nabla]s)(X)
        &=
        \sum_i -\nabla_{e_i} \nabla_{e_i} \nabla_X s
        +\nabla_X \nabla_{e_i} \nabla_{e_i} s
        +\sum_i \nabla_{\nabla_{e_i} \nabla_{e_i} X-\nabla_X \nabla_{e_i} e_i} s.
    \end{align}
    Now we can use $R(e_i, X)=\nabla_{e_i} \nabla_X-\nabla_X \nabla_{e_i} -\nabla_{[e_i,X]}$ twice to commute covariant derivatives to obtain
    \begin{align}\label{triplecommute}
    \begin{split}
        -\nabla_{e_i} \nabla_{e_i} \nabla_X s
        +\nabla_X \nabla_{e_i} \nabla_{e_i} s&=-\nabla_{e_i} (\nabla_{e_i} \nabla_X s-\nabla_{X}\nabla_{e_i}s)
        +(\nabla_X \nabla_{e_i} -\nabla_{e_i}\nabla_X)\nabla_{e_i} s
        \\
        &=
        -\nabla_{e_i} (R^E(e_i,X)s+\nabla_{[e_i,X]} s)
        -
        R^E(e_i,X) \nabla_{e_i} s- \underbrace{\nabla_{[e_i,X]}}_{=0} \nabla_{e_i} s
        \\
        &=
        -(\nabla_{e_i} R^E(e_i,X))s-R^E(e_i,X) \nabla_{e_i} s
        -\nabla_{e_i}(\nabla s)(\underbrace{[e_i,X]}_{=0})-\nabla_{\nabla_{e_i} [e_i,X]} s
        -R^E(e_i,X) \nabla_{e_i} s
    \end{split}
    \end{align}
    and the last summand of $([\Delta,\nabla]s)(X)$ can be 
    rewritten using the following relation:
    \begin{align}\label{Riemanncurvaturecompute}
        \nabla_{e_i} \nabla_{e_i} X-\nabla_X \nabla_{e_i} e_i
        &=
        \nabla_{e_i} \nabla_X e_i+\nabla_{e_i} [e_i,X]-\nabla_X \nabla_{e_i} e_i
        =
        R^M(e_i,X) e_i + \underbrace{\nabla_{[e_i,X]}}_{=0} e_i + \nabla_{e_i} [e_i,X].
    \end{align}
    After putting  \eqref{triplecommute} and \eqref{Riemanncurvaturecompute} into \eqref{LBFF}, we obtain
    \begin{align*}
        [\Delta, \nabla]s(X)&=\sum_i -\nabla_{e_i} \nabla_{e_i} \nabla_X s
        +\nabla_X \nabla_{e_i} \nabla_{e_i} s
        +\sum_i \nabla_{\nabla_{e_i} \nabla_{e_i} X-\nabla_X \nabla_{e_i} e_i} s\\
        &=\sum_{i}
        -(\nabla_{e_i} R^E(e_i,X))s-2R^E(e_i,X) \nabla_{e_i} s
        -\nabla_{\nabla_{e_i} [e_i,X]} s
       \\
       &+\sum_{i} \nabla_{ R^M(e_i,X) e_i} s+\nabla_{ \nabla_{e_i} [e_i,X]} s\\
       &=
        \sum_i
        -2R^E(e_i,X) \nabla_{e_i} s
        -(\nabla_{e_i} R^E(e_i,X))s
        +\nabla_{R^M(e_i,X)e_i}
    \end{align*}
    The claim follows by using $-\Ric(X)=\sum_i R^M(e_i,X)e_i$. 
\end{proof}

Now we move from general vector bundles to tensor bundles. On each of these tensor bundles, we will use the metric and connection induced from the Riemannian metric, and the Levi-Civita connection. 
\begin{proposition}
\label{proposition:commutator-estimates}
    For a $(0,k)$-tensor $T \in C^\infty(T^* M ^{\otimes k})$ we have
    \[
    |[\Delta, \nabla] T|
    \leq
    n(2k+1) |\Riem| \cdot |\nabla T|
    +
    kn |\nabla \Riem| \cdot |T|.
    \]
\end{proposition}

\begin{proof}
    We begin choosing an arbitrary $x\in M$, and estimating $|R(U,V)T|$ for unit length vectors $U,V \in T_x M$. 
    
    Denote the curvature endomorphism of $T^* M ^{\otimes k}$ by $R$ and the curvature endomorphism of $TM$ by $R^M$.
    Then, for $U,V,X_1,\dots,X_k \in T_x M$:
    \begin{align}
    \label{equation:curvature-action}
    (R(U,V)T)(X_1,\dots,X_k)
    =
    -\sum_{i=1}^k
    T(X_1,\dots,X_{i-1}, R^M(U,V)X_i, X_{i+1}, \dots, X_k).
    \end{align}
    To see this, we extend  $U,V,X_1,\dots,X_k$ to vector fields in a neighbourhood of $x$ that are parallel at $x$ itself, expand the definition of $R$ and $R^M$, and use $\nabla_U T(X_1,\dots,X_k)=U(T(X_1,\dots,X_k))-T(\nabla_U X_1,\dots,x_k)-\dots$. Indeed, this gives 
    \begin{align*}
        (R(U,V)T)(X_1,\cdots,X_k)&=(\nabla_{U}\nabla_{V} T-\nabla_{V}\nabla_{U}T-\nabla_{[U,V]}T)(X_1,\cdots,X_k)\\
        &=U(\nabla_{V}T(X_1,\cdots,X_k))-\nabla_V T(\nabla_{U}X_1,\cdots,X_k)-\cdots -\nabla_{V}T(X_1,\cdots,\nabla_{U} X_k)\\
        &-V(\nabla_{U}T(X_1,\cdots,X_k))+\nabla_U T(\nabla_{V}X_1,\cdots,X_k)+\cdots +\nabla_{U}T(X_1,\cdots,\nabla_{V} X_k)\\
        &-[U,V](T(X_1,\cdots,X_k))+T(\nabla_{[U,V]}X_1,\cdots,X_k)+T(X_1,\cdots,\nabla_{[U,V]}X_k)\\
        &=U(V(T(X_1,\cdots,X_k))-T(\nabla_{V}X_1,\cdots,X_k)-\cdots- T(X_1,\cdots,\nabla_{V}X_k))\\
        &-V(T(\nabla_{U}X_1,\cdots,X_k))+T(\nabla_{V}\nabla_{U}X_1,\cdots,X_k)+\cdots +T(\nabla_{U}X_1,\cdots,\nabla_{V}X_k)+\cdots\\
        &-V(T(X_1,\cdots,\nabla_{U}X_k)+T(\nabla_{V}X_1,\cdots,\nabla_{U}X_k)+\cdots+T(X_1,\cdots,\nabla_{V}\nabla_{U}X_k)\\
        &-V(U(T(X_1,\cdots,X_k))-T(\nabla_{U}X_1,\cdots,X_k)-\cdots -T(X_1,\cdots,\nabla_{U}X_k))\\
        &+U(T(\nabla_{V}X_1,\cdots,X_k))-T(\nabla_{U}\nabla_{V}X_1,\cdots,X_k)-\cdots -T(\nabla_{V}X_1,\cdots,\nabla_{U}X_k)+\cdots\\
        &+U(T(X_1,\cdots,\nabla_{V}X_k))-T(\nabla_{U}X_1,\cdots,\nabla_{V}X_k)-\cdots-T(X_1,\cdots,\nabla_{U}\nabla_{V}X_k)\\
         &-[U,V](T(X_1,\cdots,X_k))+T(\nabla_{[U,V]}X_1,\cdots,X_k)+T(X_1,\cdots,\nabla_{[U,V]}X_k)
    \end{align*}
    and 
     \Cref{equation:curvature-action} follows. We therefore find 
    \begin{align}
    \label{equation:R-action-bound}
        \begin{split}
        |R(U,V)T|
        &\leq
        k
        |T| \cdot |\Id \otimes \dots \otimes \Id \otimes R^M(U,V)^* \otimes \Id \otimes \dots \otimes \Id|_{op}
        \\
        &=
        k
        |T| \cdot |\Id|_{op} \cdot \dots \cdot | \Id|_{op} \cdot  |R^M(U,V)^*|_{op} \cdot |\Id|_{op} \cdot \dots \cdot |\Id|_{op}
        \\
        &\leq
        k |T| \cdot |\Riem|,
        \end{split}
    \end{align}
    where in the first step we used the definition of the operator norm $|\cdot|_{op}$ for endomorphisms acting on $T^* M^{\otimes k}$;
    in the second step we computed the operator norm acting on a tensor product;
    in the last step we used $|\Id|_{op}=1$ and $|R^M(U,V)^*|_{op}=|R^M(U,V)|_{op} \leq |R^M(U,V)| \leq |\Riem|$.
    This holds because the operator norm is bounded by the Hilbert-Schmidt norm (with constant equal to one), and the last step follows from writing $U,V$ in an orthonormal basis.

    We now estimate $|(\nabla_X) R(U,V)T|$ for unit length vectors $X,U,V \in T_x M$. We again extend $X,U,V,X_1,\cdots,X_k$ to vector fields defined in a neighbourhood of $x$ that are parallel at $x$ itself. 
    Differentiating \cref{equation:curvature-action} with respect to $X$ gives:
    \begin{align}
    \label{equation:curvature-action-differentiated}
    (\nabla_X R(U,V))T(X_1,\dots,X_k)
    =
    -\sum_{i=1}^k
    T(X_1,\dots,X_{i-1}, (\nabla_XR^M)(U,V)X_i, X_{i+1}, \dots, X_k).
    \end{align}
    As above, we get
    \begin{align}
        \label{equation:nabla-R-action}
        \begin{split}
        |\nabla_X R(U,V)T|
        &\leq
        k
        |T| \cdot |\Id \otimes \dots \otimes \Id \otimes \nabla_X R^M(U,V)^* \otimes \Id \otimes \dots \otimes \Id|_{op}
        \\
        &=
        k
        |T| \cdot |\Id|_{op} \cdot \dots \cdot | \Id|_{op} \cdot  |\nabla_X R^M(U,V)^*|_{op} \cdot |\Id|_{op} \cdot \dots \cdot |\Id|_{op}
        \\
        &\leq
        k |T| \cdot |\nabla_X \Riem|.
        \end{split}
    \end{align}
    For $j \in \{1,\dots,n\}$ we write
    \[
    A_j
    :=
    \nabla_{\Ric(e_j)}T,
    \qquad
    B_j
    :=
    2 \sum_{i=1}^n
    R(e_i,e_j) \nabla_{e_i} T,
    \qquad
    C_j
    :=
    \sum_{i=1}^n
    (\nabla_{e_i} R)(e_i,e_j) T
    \]
    and \cref{equation:Delta-commutator} gives:
    \begin{align}
    \label{equation:commutator-with-Bj-Cj}
        |[\Delta, \nabla]T|
        &=
        \left(
        \sum_{j=1}^n
        |A_j+B_j+C_j|^2
        \right)^{1/2}
        \leq
        \left(
        \sum_{j=1}^n
        |A_j|^2
        \right)^{1/2}
        +
        \left(
        \sum_{j=1}^n
        |B_j|^2
        \right)^{1/2}
        +
        \left(
        \sum_{j=1}^n
        |C_j|^2
        \right)^{1/2}.
    \end{align}
    Now:
    \begin{align*}
        \sum_{j=1}^n
        |A_j|^2
        &\leq
        \sum_{j=1}^n |\Ric(e_j)|^2 |\nabla T|^2
        =
        |\Ric|^2 |\nabla T|^2
        \leq
        n^2 |\Riem|^2 |\nabla T|^2
    \end{align*}
    and    
    \begin{align*}
        \sum_{j=1}^n
        |B_j|^2
        \leq
        n \cdot 4k^2 |\Riem|^2
        \left(
        \sum_{i=1}^n |\nabla_{e_i} T|
        \right)^2
        \leq
        n^2 \cdot 4k^2 |\Riem|^2
        \left(
        \sum_{i=1}^n |\nabla_{e_i} T|^2
        \right)
        \leq
        n^2 \cdot 4k^2 |\Riem|^2
        |\nabla T|^2,
    \end{align*}
    where in the first step we used \cref{equation:R-action-bound};
    in the second step we used the norm equivalence between $1$-norm and $2$-norm.
    Furthermore:
    \begin{align*}
        \sum_{j=1}^n |C_j|^2
        &\leq
        n \left(
        \sum_{i=1}^n
        |\nabla_{e_i} \Riem| \cdot k \cdot |T|
        \right)^2
        \leq
        n^2
        \sum_{i=1}^n
        |\nabla_{e_i} \Riem|^2 \cdot k^2 \cdot |T|^2
        \leq
        n^2 |\nabla \Riem|^2 k^2 |T|^2,
    \end{align*}
    where in the second step we used the equivalence of $1$-norm and $2$-norm.

    Taking the square root of the last three lines and plugging them into \cref{equation:commutator-with-Bj-Cj} gives the claim.
\end{proof}

\begin{proposition}
\label{proposition:nabla-commutator-estimates}
    For a $(0,l)$-tensor $T \in C^\infty(T^* M ^{\otimes l})$ we have
    \[
    |\nabla [\Delta, \nabla] T|
    \leq
    n(3l+n) |\nabla \Riem| \cdot |\nabla T|
        +
        n(2l+n) |\Riem| \cdot |\nabla^2 T|
        +
        nl | \nabla^2 \Riem| \cdot |T|.
    \]
\end{proposition}

\begin{proof}
    We fix $x \in M$ and choosing an orthonormal basis $e_i$ with $\nabla e_i=0$ at $x$ as before.
    Setting $X=e_k$ in \cref{equation:Delta-commutator} and differentiating with respect $e_j$ gives:
    \begin{align*}
        \nabla [\Delta, \nabla] T(e_j,e_k)
        &=
        -\nabla_{(\nabla_{e_j}\Ric)(e_k)}T
        -(\nabla^2T)(e_j,\Ric(e_k))
        \\
        &\quad
        -
        2
        \sum_i (\nabla_{e_j} R)(e_i,e_k) \nabla_{e_i} T
        -
        2
        \sum_i
        R(e_i,e_k) \nabla_{e_j} \nabla_{e_i} T
        -
        \sum_i (\nabla_{e_j} \nabla_{e_i} R)(e_i, e_k) T
        -
        \sum_i (\nabla_{e_i} R)(e_i, e_k)(\nabla_{e_j} T)
        \\
        &=
        E_{jk}+F_{jk}+
        A_{jk}+B_{jk}+C_{jk}+D_{jk},
    \end{align*}
    where for the first summand we used
    \[
    \nabla \nabla_{\Ric(\cdot)} T(e_j,e_k)
    =
    \nabla_{e_j} \nabla_{\Ric(e_k)} T-
    \underbrace{\nabla_{\Ric(\nabla_{e_j}e_k)} T}_{=0 \text{ at } x}
    =
    \nabla^2 T(e_j, \Ric(e_k))+\nabla_{\nabla_{e_j}(\Ric(e_k))}T
    =
    \nabla^2 T(e_j, \Ric(e_k))+\nabla_{(\nabla_{e_j}\Ric)(e_k)}T.
    \]
    For the individual summands we find
    \begin{align*}
        |A_{jk}|^2
        &=
        4
        \left| 
        \sum_i 
        (\nabla_{e_j} R)(e_i,e_k) \nabla_{e_i} T
        \right|^2
        \\
        &\leq
        4n
        \sum_i 
        \left| 
        (\nabla_{e_j} R)(e_i,e_k) \nabla_{e_i} T
        \right|^2
        \\
        &\leq
        4n l^2 \sum_i |\nabla_{e_j} \Riem|^2 |\nabla_{e_i} T|^2
        \\
        &\leq
        4n l^2 |\nabla_{e_j} \Riem|^2 |\nabla T|^2,
    \end{align*}
    where in the second step we used the equivalence of $1$-norm and $2$-norm;
    and in the third step we used \cref{equation:nabla-R-action}.
    For the other summands we find analogously:
    \begin{align*}
        |B_{jk}|^2
        &\leq
        4
        \left|
        \sum_i
        R(e_i,e_k) \nabla_{e_j} \nabla_{e_i} T
        \right|^2
        \\
        &\leq
        4n 
        \sum_i
        \left|
        R(e_i,e_k) \nabla_{e_j} \nabla_{e_i} T
        \right|^2
        \\
        &\leq
        4n l^2
        |\Riem|^2
        \sum_i
        |\nabla_{e_j} \nabla_{e_i} T|^2,
        \\
        |C_{jk}|^2
        &\leq
        \left|
        \sum_i (\nabla_{e_j} \nabla_{e_i} R)(e_i, e_k) T
        \right|^2
        \\
        &\leq
        n
        \sum_i 
        \left|
        (\nabla_{e_j} \nabla_{e_i} R)(e_i, e_k) T
        \right|^2
        \\
        &\leq
        n l^2 |T|^2 \sum_i |\nabla_{e_j} \nabla_{e_i} \Riem|^2,
        \\
        |D_{jk}|^2
        &\leq
        \left|
        \sum_i (\nabla_{e_i} R)(e_i, e_k)(\nabla_{e_j} T)
        \right|^2
        \\
        &\leq
        n 
        \sum_i 
        \left|
        (\nabla_{e_i} R)(e_i, e_k)(\nabla_{e_j} T)
        \right|^2
        \\
        &\leq
        n l^2
        |\nabla \Riem|^2 |\nabla_{e_j} T|^2,
        \\
        |E_{jk}|^2
        &\leq
        |\nabla_{(\nabla_{e_j}\Ric)(e_k)}T|^2
        \leq
        n^2|\nabla \Riem|^2 \cdot |\nabla T|^2,
        \\
        |F_{jk}|^2
        &\leq
        |(\nabla^2T)(e_j,\Ric(e_k))|^2
        \leq
        n^2|\Riem|^2 \cdot |\nabla^2 T|^2.
    \end{align*}
    Thus:
    \begin{align*}
        |\nabla [\Delta,\nabla] T|
        &
        \leq
        \left(
        \sum_{j,k}
        |E_{jk}|^2
        \right)^{1/2}
        +
        \left(
        \sum_{j,k}
        |F_{jk}|^2
        \right)^{1/2}
        +
        \left(
        \sum_{j,k}
        |A_{jk}|^2
        \right)^{1/2}
        +
        \left(
        \sum_{j,k}
        |B_{jk}|^2
        \right)^{1/2}
        +
        \left(
        \sum_{j,k}
        |C_{jk}|^2
        \right)^{1/2}
        +
        \left(
        \sum_{j,k}
        |D_{jk}|^2
        \right)^{1/2}
        \\
        &\leq
        (n^2 \cdot n^2|\nabla \Riem|^2 \cdot |\nabla T|^2)^{1/2}
        +
        (n^2 \cdot n^2|\Riem|^2 \cdot |\nabla^2 T|^2)^{1/2}
        \\
        &\qquad+
        \left(
        4n^2 l^2 |\nabla T|^2 \sum_k |\nabla_{e_k} \Riem|^2
        \right)^{1/2}
        +
        \left(
        4n^2 l^2 | \Riem|^2 \sum_{i,k} |\nabla_{e_k} \nabla_{e_i} T|^2
        \right)^{1/2}
        \\
        &\qquad
        +
        \left(
        n^2 l^2 |T|^2 \sum_{i,k} |\nabla_{e_k} \nabla_{e_i} \Riem|^2
        \right)^{1/2}
        +
        \left(
        n^2 l^2 |\nabla \Riem|^2 \sum_k |\nabla_{e_k} T|^2
        \right)^{1/2}
        \\
        &=
        (2nl+n^2) |\nabla \Riem| \cdot |\nabla T|
        +
        (2nl+n^2) |\Riem| \cdot |\nabla^2 T|
        +
        nl | \nabla^2 \Riem| \cdot |T|
        +
        nl |\nabla \Riem| \cdot |\nabla T|.
        \qedhere
    \end{align*}
\end{proof}

\begin{proposition}
    \label{proposition:estimates-from-commutator-formula}
    Assume $\|{\Ric}_{C^0} \leq K_1$, $\|{\Riem}_{C^0} \leq K_2$, $\|{\nabla \Riem}_{C^0} \leq K_3$, $\|{\nabla^2 \Riem}_{C^0} \leq K_4$.
    Then
    \begin{align*}
        \|{\nabla u}_{L^2}^2
        &\leq
        \frac{1}{2}\|{\Delta u}_{L^2}^2
        +
        \frac{1}{2} \|{u}^2_{L^2},
        \\
        \|{\nabla ^2 u}_{L^2}^2
        &\leq
        \|{\Delta u}_{L^2}^2
        +
        K_1 \|{\nabla u}_{L^2}^2,
        \\
        \|{\nabla^3 u}_{L^2}^2
        &\leq
        2\|{\nabla \Delta u}_{L^2}^2
        +
        \Cr{D^3u-estimate-second-summand}
        \|{u}_{H^2}^2,
        \\
        \|{\nabla^4 u}_{L^2}^2
        &\leq
        3\|{\nabla^2 \Delta u}_{L^2}^2
        +
        \Cr{D^4u-estimate-second-summand}
        \|{u}_{H^3}^2,
        \\
        \|{\nabla^5 u}_{L^2}^2
        &\leq
        4\|{\nabla^3 \Delta u}_{L^2}^2
        +
        \Cr{D^5u-estimate-second-summand}
        \|{u}_{L^2_4}^2.
    \end{align*}
    where
    \begin{align*}
        \Cl{D^3u-estimate-second-summand}
        &=
        2K_1^2
        +
        3nK_2
        +
        nK_3,
        \Cl{D^4u-estimate-second-summand}
        =3(K_1+nK_3)^2
        +
        3(3nK_2+nK_3)^2
        +
        5nK_2+2nK_3,
        \\
        \Cl{D^5u-estimate-second-summand}
        &=
        4(K_1+2nK_3+nK_4)^2
        +
        4\bigl(n(n+3)K_3+n(n+2)K_2+nK_4\bigr)^2
        +
        4(5nK_2+2nK_3)^2
        +
        7nK_2+3nK_3.
    \end{align*}
\end{proposition}

\begin{proof}
    Estimate for $\nabla u$:
    we have $\|{\nabla u}_{L^2}^2 \leq \|{\Delta u}_{L^2} \|{u}_{L^2} \leq \frac{1}{2}(\|{\Delta u}_{L^2}^2+
    \|{u}_{L^2}^2)$, where we used the Cauchy-Schwarz inequality in the first step and $2ab \leq a^2+b^2$ for real numbers $a,b$ in the second step.

    Estimate for $\nabla^2 u$:
    this is the integrated Bochner formula.

    Estimate for $\nabla^3 u$:
    we have
    \begin{align*}
        \|{\nabla^3 u}_{L^2}^2
        &=
        \langle \Delta \nabla^2 u,\nabla^2 u\rangle
        \\
        &=
        \langle \nabla \Delta \nabla u,\nabla^2 u\rangle
        +
        \langle [\Delta,\nabla]\nabla u,\nabla^2 u\rangle
        \\
        &=
        \|{\Delta\nabla u}_{L^2}^2
        +
        \langle [\Delta,\nabla]\nabla u,\nabla^2 u\rangle
        \\
        &=
        \|{\nabla\Delta u+[\Delta,\nabla]u}_{L^2}^2
        +
        \langle [\Delta,\nabla]\nabla u,\nabla^2 u\rangle
        \\
        &\leq
        2\|{\nabla\Delta u}_{L^2}^2
        +
        2\|{[\Delta,\nabla]u}_{L^2}^2
        +
        \|{[\Delta,\nabla]\nabla u}_{L^2}
        \|{\nabla^2 u}_{L^2}
        \\
        &\leq
        2\|{\nabla\Delta u}_{L^2}^2
        +
        2K_1^2\|{\nabla u}_{L^2}^2
        +
        \bigl(
        3n\|{\Riem}_{C^0}
        +
        n\|{\nabla\Riem}_{C^0}
        \bigr)
        \|{u}_{H^2}\|{\nabla^2 u}_{L^2}
        \\
        &\leq
        2\|{\nabla\Delta u}_{L^2}^2
        +
        \bigl(
        2K_1^2
        +
        3nK_2
        +
        nK_3
        \bigr)
        \|{u}_{H^2}^2,
    \end{align*}
    where in the sixth step we used \cref{equation:Delta-commutator} to estimate the $[\Delta,\nabla]u$ term,
    and we used \cref{proposition:commutator-estimates} to estimate the $[\Delta,\nabla]\nabla u$ term.

    Estimate for $\nabla^4 u$:
    \begin{align*}
        \|{\nabla^4 u}_{L^2}^2
        &=
        \langle \Delta\nabla^3 u,\nabla^3 u\rangle
        \\
        &=
        \langle \nabla \Delta \nabla^2 u,\nabla^3 u\rangle
        +
        \langle [\Delta,\nabla]\nabla^2 u,\nabla^3 u\rangle
        \\
        &=
        \|{\Delta\nabla^2 u}_{L^2}^2
        +
        \langle [\Delta,\nabla]\nabla^2 u,\nabla^3 u\rangle
        \\
        &=
        \|{\nabla^2\Delta u+\nabla[\Delta,\nabla]u+[\Delta,\nabla]\nabla u}_{L^2}^2
        +
        \langle [\Delta,\nabla]\nabla^2 u,\nabla^3 u\rangle
        \\
        &\leq
        3\|{\nabla^2\Delta u}_{L^2}^2
        +
        3\|{\nabla[\Delta,\nabla]u}_{L^2}^2
        +
        3\|{[\Delta,\nabla]\nabla u}_{L^2}^2
        +
        \|{[\Delta,\nabla]\nabla^2 u}_{L^2}
        \|{\nabla^3 u}_{L^2}
        \\
        &\leq
        3\|{\nabla^2\Delta u}_{L^2}^2
        +
        3(K_1+nK_3)^2\|{u}_{H^2}^2
        +
        3(3nK_2+nK_3)^2\|{u}_{H^2}^2
        \\
        &\quad
        +
        (5nK_2+2nK_3)\|{u}_{H^3}^2
        \\
        &\leq
        3\|{\nabla^2\Delta u}_{L^2}^2
        +
        \Bigl(
        3(K_1+nK_3)^2
        +
        3(3nK_2+nK_3)^2
        +
        5nK_2+2nK_3
        \Bigr)
        \|{u}_{H^3}^2 .
    \end{align*}
    In the sixth step we used \cref{equation:Delta-commutator}, which gives $\nabla [\Delta,\nabla] u=\nabla (-\nabla_{\Ric(\cdot)} u)$, and bounded this as in the proof of \cref{proposition:nabla-commutator-estimates};
    for the other estimates in this step we used \cref{proposition:commutator-estimates}.
    
    Estimate for $\nabla^5 u$:
    using \cref{proposition:commutator-estimates,proposition:nabla-commutator-estimates}, the steps are essentially the same as before and we obtain:
    \begin{align*}
        \|{\nabla^5u}_{L^2}^2
        &=
        \|{\Delta\nabla^3u}_{L^2}^2
        +
        \langle [\Delta,\nabla]\nabla^3u,\nabla^4u\rangle
        \\
        &=
        \|{\nabla^3\Delta u
        +
        \nabla^2[\Delta,\nabla]u
        +
        \nabla[\Delta,\nabla]\nabla u
        +
        [\Delta,\nabla]\nabla^2u}_{L^2}^2
        +
        \langle [\Delta,\nabla]\nabla^3u,\nabla^4u\rangle
        \\
        &\leq
        4\|{\nabla^3\Delta u}_{L^2}^2
        +
        4\|{\nabla^2[\Delta,\nabla]u}_{L^2}^2
        +
        4\|{\nabla[\Delta,\nabla]\nabla u}_{L^2}^2
        +
        4\|{[\Delta,\nabla]\nabla^2u}_{L^2}^2
        \\
        &\quad
        +
        \|{[\Delta,\nabla]\nabla^3u}_{L^2}
        \|{\nabla^4u}_{L^2}
        \\
        &\leq
        4\|{\nabla^3\Delta u}_{L^2}^2
        +
        4(K_1+2nK_3+nK_4)^2\|{u}_{H^3}^2
        \\
        &\quad
        +
        4\bigl(n(n+3)K_3+n(n+2)K_2+nK_4\bigr)^2\|{u}_{H^3}^2
        \\
        &\quad
        +
        4(5nK_2+2nK_3)^2\|{u}_{H^3}^2
        +
        (7nK_2+3nK_3)\|{u}_{L^2_4}^2
        \\
        &\leq
        4\|{\nabla^3\Delta u}_{L^2}^2
        \\
        &\quad
        +
        \Bigl(
        4(K_1+2nK_3+nK_4)^2
        +
        4\bigl(n(n+3)K_3+n(n+2)K_2+nK_4\bigr)^2
        \\
        &\qquad
        +
        4(5nK_2+2nK_3)^2
        +
        7nK_2+3nK_3
        \Bigr)
        \|{u}_{L^2_4}^2,
    \end{align*}
    where we used
    \begin{align*}
    \begin{split}
        |\nabla^2[\Delta,\nabla]u|
        &=
        |\nabla^2(-\nabla_{\Ric(\cdot)}u)|
        \\
        &\leq
        |\nabla^2\Ric|\,|\nabla u|
        +
        2|\nabla\Ric|\,|\nabla^2u|
        +
        |\Ric|\,|\nabla^3u|
        \\
        &\leq
        n|\nabla^2\Riem|\,|\nabla u|
        +
        2n|\nabla\Riem|\,|\nabla^2u|
        +
        K_1|\nabla^3u|,
        \\
        \|{\nabla^2[\Delta,\nabla]u}_{L^2}
        &\leq
        (K_1+2nK_3+nK_4)\|{u}_{H^3}.
    \end{split}
    \end{align*}
    by \cref{equation:Delta-commutator}.
    The other estimates are analogous to the previous steps.
\end{proof}

\begin{corollary}
\label{corollary:a-priori-estimate}
    Let $E\in H^3(M)$, and $\epsilon=\|{E}_{H^3}$, and $\mathscr{L}
        =
        -\frac{1}{2}\Delta
        +
        (1+E)\Id$.
    Then, for every $u\in C^\infty(M)$,
    \[
        \|{u}_{H^5}
        \leq
        \Cr{H^5-estimate-Laplace-term}
        \|{\mathscr{L}u}_{H^3}
        +
        \Cr{H^5-estimate-u-term}
        \|{u}_{L^2},
    \]
    where
    $\Cr{H^5-estimate-Laplace-term}$
    and
    $\Cr{H^5-estimate-u-term}$
    are defined at the bottom of
    \cref{equation:a-priori-derivation}.
\end{corollary}

\begin{proof}
By
\cref{corollary:sobolev-multiplication-thm,corollary:sobolev-low-order-multiplication},
we have
\begin{align*}
    \|{(1+E)u}_{L^2}
    &\leq
    \left(
    1+
    \Cr{const:sobolev-low-order-multiplication-M0}
    \epsilon
    \right)
    \|{u}_{L^2},
    \\
    \|{\nabla((1+E)u)}_{L^2}
    &\leq
    \|{(1+E)u}_{H^1}
    \leq
    \left(
    1+
    \Cr{const:sobolev-low-order-multiplication-M1}
    \epsilon
    \right)
    \|{u}_{H^1},
    \\
    \|{\nabla^2((1+E)u)}_{L^2}
    &\leq
    \|{(1+E)u}_{H^2}
    \leq
    \left(
    1+
    \Cr{const:sobolev-low-order-multiplication-M2}
    \epsilon
    \right)
    \|{u}_{H^2},
    \\
    \|{\nabla^3((1+E)u)}_{L^2}
    &\leq
    \|{(1+E)u}_{H^3}
    \leq
    \left(
    1+
    \Cr{const:multiplication-theorem}
    \epsilon
    \right)
    \|{u}_{H^3}.
\end{align*}

Taking square roots in
\cref{proposition:estimates-from-commutator-formula}
gives
\begin{align*}
    \|{\nabla u}_{L^2}
    &\leq
    \frac{1}{\sqrt{2}}
    \|{\Delta u}_{L^2}
    +
    \frac{1}{\sqrt{2}}
    \|{u}_{L^2},
    \\
    \|{\nabla^2u}_{L^2}
    &\leq
    \|{\Delta u}_{L^2}
    +
    \sqrt{K_1}
    \|{\nabla u}_{L^2},
    \\
    \|{\nabla^3u}_{L^2}
    &\leq
    \sqrt{2}
    \|{\nabla\Delta u}_{L^2}
    +
    \sqrt{\Cr{D^3u-estimate-second-summand}}
    \|{u}_{H^2},
    \\
    \|{\nabla^4u}_{L^2}
    &\leq
    \sqrt{3}
    \|{\nabla^2\Delta u}_{L^2}
    +
    \sqrt{\Cr{D^4u-estimate-second-summand}}
    \|{u}_{H^3},
    \\
    \|{\nabla^5u}_{L^2}
    &\leq
    2
    \|{\nabla^3\Delta u}_{L^2}
    +
    \sqrt{\Cr{D^5u-estimate-second-summand}}
    \|{u}_{L^2_4}.
\end{align*}

Since
\[
    \Delta u
    =
    -2\mathscr{L}u
    +
    2(1+E)u,
\]
we obtain
\begin{align}
\label{equation:a-priori-derivation}
\begin{split}
    \|{u}_{H^1}
    &\leq
    \frac{1}{\sqrt{2}}
    \|{\Delta u}_{L^2}
    +
    \left(
    1+\frac{1}{\sqrt{2}}
    \right)
    \|{u}_{L^2}
    \\
    &\leq
    \sqrt{2}
    \|{\mathscr{L}u}_{L^2}
    +
    \left(
    1+\frac{1}{\sqrt{2}}
    +
    \sqrt{2}
    \left(
    1+
    \Cr{const:sobolev-low-order-multiplication-M0}
    \epsilon
    \right)
    \right)
    \|{u}_{L^2}
    \\
    &=
    \Cl{H^1-estimate-Laplace-term}
    \|{\mathscr{L}u}_{L^2}
    +
    \Cl{H^1-estimate-u-term}
    \|{u}_{L^2},
    \\
    \|{u}_{H^2}
    &\leq
    \left(
    1+\sqrt{K_1}
    \right)
    \|{u}_{H^1}
    +
    \|{\Delta u}_{L^2}
    \\
    &\leq
    2
    \|{\mathscr{L}u}_{L^2}
    +
    \left(
    1+\sqrt{K_1}
    \right)
    \|{u}_{H^1}
    +
    2
    \left(
    1+
    \Cr{const:sobolev-low-order-multiplication-M0}
    \epsilon
    \right)
    \|{u}_{L^2}
    \\
    &\leq
    \left(
    2
    +
    \left(
    1+\sqrt{K_1}
    \right)
    \Cr{H^1-estimate-Laplace-term}
    \right)
    \|{\mathscr{L}u}_{L^2}
    \\
    &\quad
    +
    \left(
    2
    \left(
    1+
    \Cr{const:sobolev-low-order-multiplication-M0}
    \epsilon
    \right)
    +
    \left(
    1+\sqrt{K_1}
    \right)
    \Cr{H^1-estimate-u-term}
    \right)
    \|{u}_{L^2}
    \\
    &=
    \Cl{H^2-estimate-Laplace-term}
    \|{\mathscr{L}u}_{L^2}
    +
    \Cl{H^2-estimate-u-term}
    \|{u}_{L^2},
    \\
    \|{u}_{H^3}
    &\leq
    \left(
    1+
    \sqrt{\Cr{D^3u-estimate-second-summand}}
    \right)
    \|{u}_{H^2}
    +
    \sqrt{2}
    \|{\nabla\Delta u}_{L^2}
    \\
    &\leq
    2\sqrt{2}
    \|{\mathscr{L}u}_{H^1}
    +
    \left(
    1+
    \sqrt{\Cr{D^3u-estimate-second-summand}}
    +
    2\sqrt{2}
    \left(
    1+
    \Cr{const:sobolev-low-order-multiplication-M1}
    \epsilon
    \right)
    \right)
    \|{u}_{H^2}
    \\
    &\leq
    \Bigg(
    2\sqrt{2}
    +
    \left(
    1+
    \sqrt{\Cr{D^3u-estimate-second-summand}}
    +
    2\sqrt{2}
    \left(
    1+
    \Cr{const:sobolev-low-order-multiplication-M1}
    \epsilon
    \right)
    \right)
    \Cr{H^2-estimate-Laplace-term}
    \Bigg)
    \|{\mathscr{L}u}_{H^1}
    \\
    &\quad
    +
    \left(
    1+
    \sqrt{\Cr{D^3u-estimate-second-summand}}
    +
    2\sqrt{2}
    \left(
    1+
    \Cr{const:sobolev-low-order-multiplication-M1}
    \epsilon
    \right)
    \right)
    \Cr{H^2-estimate-u-term}
    \|{u}_{L^2}
    \\
    &=
    \Cl{H^3-estimate-Laplace-term}
    \|{\mathscr{L}u}_{H^1}
    +
    \Cl{H^3-estimate-u-term}
    \|{u}_{L^2},
    \\
    \|{u}_{L^2_4}
    &\leq
    \left(
    1+
    \sqrt{\Cr{D^4u-estimate-second-summand}}
    \right)
    \|{u}_{H^3}
    +
    \sqrt{3}
    \|{\nabla^2\Delta u}_{L^2}
    \\
    &\leq
    2\sqrt{3}
    \|{\mathscr{L}u}_{H^2}
    +
    \left(
    1+
    \sqrt{\Cr{D^4u-estimate-second-summand}}
    +
    2\sqrt{3}
    \left(
    1+
    \Cr{const:sobolev-low-order-multiplication-M2}
    \epsilon
    \right)
    \right)
    \|{u}_{H^3}
    \\
    &\leq
    \Bigg(
    2\sqrt{3}
    +
    \left(
    1+
    \sqrt{\Cr{D^4u-estimate-second-summand}}
    +
    2\sqrt{3}
    \left(
    1+
    \Cr{const:sobolev-low-order-multiplication-M2}
    \epsilon
    \right)
    \right)
    \Cr{H^3-estimate-Laplace-term}
    \Bigg)
    \|{\mathscr{L}u}_{H^2}
    \\
    &\quad
    +
    \left(
    1+
    \sqrt{\Cr{D^4u-estimate-second-summand}}
    +
    2\sqrt{3}
    \left(
    1+
    \Cr{const:sobolev-low-order-multiplication-M2}
    \epsilon
    \right)
    \right)
    \Cr{H^3-estimate-u-term}
    \|{u}_{L^2}
    \\
    &=
    \Cl{L^2_4-estimate-Laplace-term}
    \|{\mathscr{L}u}_{H^2}
    +
    \Cl{L^2_4-estimate-u-term}
    \|{u}_{L^2},
    \\
    \|{u}_{H^5}
    &\leq
    \left(
    1+
    \sqrt{\Cr{D^5u-estimate-second-summand}}
    \right)
    \|{u}_{L^2_4}
    +
    2
    \|{\nabla^3\Delta u}_{L^2}
    \\
    &\leq
    4
    \|{\mathscr{L}u}_{H^3}
    +
    \left(
    1+
    \sqrt{\Cr{D^5u-estimate-second-summand}}
    +
    4
    \left(
    1+
    \Cr{const:multiplication-theorem}
    \epsilon
    \right)
    \right)
    \|{u}_{L^2_4}
    \\
    &\leq
    \Bigg(
    4
    +
    \left(
    1+
    \sqrt{\Cr{D^5u-estimate-second-summand}}
    +
    4
    \left(
    1+
    \Cr{const:multiplication-theorem}
    \epsilon
    \right)
    \right)
    \Cr{L^2_4-estimate-Laplace-term}
    \Bigg)
    \|{\mathscr{L}u}_{H^3}
    \\
    &\quad
    +
    \left(
    1+
    \sqrt{\Cr{D^5u-estimate-second-summand}}
    +
    4
    \left(
    1+
    \Cr{const:multiplication-theorem}
    \epsilon
    \right)
    \right)
    \Cr{L^2_4-estimate-u-term}
    \|{u}_{L^2}
    \\
    &=
    \Cl{H^5-estimate-Laplace-term}
    \|{\mathscr{L}u}_{H^3}
    +
    \Cl{H^5-estimate-u-term}
    \|{u}_{L^2}.
\end{split}
\end{align}
\end{proof}

\begin{corollary}
\label{corollary:inj-estimate}
    Let $E\in H^3(M)$ be $T^2\rtimes D_6$-invariant, and $\mathscr{L}
        =
        -\frac{1}{2}\Delta
        +
        (1+E)\Id$.
    Let $\lambda_{\min}(\mathscr{L})$ denote the smallest absolute
    value of an eigenvalue of $\mathscr{L}$ on
    $T^2\rtimes D_6$-invariant functions.
    If $\lambda_{\min}(\mathscr{L})>0$, then 
    for all $D_6 \times T^2$-invariant $u \in H^5(M)$, we have
    \[
    \|{u}_{H^5}
    \leq
    \Cl{const:injectivity-estimate}
    \|{\mathscr{L} u}_{H^3},
    \]
    where $\Cr{const:injectivity-estimate}=\left(
    \Cr{H^5-estimate-Laplace-term}
    +
    \frac{\Cr{H^5-estimate-u-term}}{\lambda_{\min}(\mathscr{L})}
    \right)$
    for the constants $\Cr{H^5-estimate-Laplace-term}, \Cr{H^5-estimate-u-term}$ from \cref{corollary:a-priori-estimate}.
\end{corollary}

\begin{proof}
    Let $\phi_1,\phi_2,\dots$ be an orthonormal basis of the
    $T^2\rtimes D_6$-invariant subspace of $L^2(M)$ consisting
    of eigenfunctions of $\mathscr{L}$, with corresponding
    eigenvalues ordered so that
    \[
    0<\lambda_{\min}(\mathscr{L})=|\lambda_1|
    \leq|\lambda_2|\leq\cdots.
    \]
    Then
    \begin{align*}
        \|{u}_{L^2}^2
        &=
        \sum_{i=1}^\infty
        |\< u, \phi_i \>_{L^2}|^2
        \|{\phi_i}_{L^2}^2
        =
        \sum_{i=1}^\infty
        |\< u, \phi_i \>_{L^2}|^2
        \frac{1}{\lambda_i^2}
        \|{\mathscr{L} \phi_i}_{L^2}^2
        \leq
        \frac{1}{\lambda_1^2}
        \sum_{i=1}^\infty
        |\< u, \phi_i \>_{L^2}|^2
        \|{\mathscr{L} \phi_i}_{L^2}^2
        =
        \frac{1}{\lambda_1^2}
        \|{\mathscr{L} u}_{L^2}^2,
    \end{align*}
    where in the first and last step we used that the $\phi_i$ are orthogonal, so the sum can be pulled out of the norm.
    Plugging this into \cref{corollary:a-priori-estimate} gives the claim.
\end{proof}

\subsection{Quadratic estimate}
\label{subsection:quadratic-estimate}

For our later application of a fixed-point theorem to the complex Monge--Ampère equation we need an estimate for the linearised operator as well as for the remaining non-linear terms.
The linear estimate has been achieved in the previous section, and in this section we prove the estimate for the non-linear terms from \cref{Taylordecomp}.

\begin{proposition}
\label{proposition:non-linear-estimate}
    Let $e^F=1+E$, where $E\in H^3(M)$ is real-valued, and set
    $\epsilon=\|{E}_{H^3}$.
    The map $\mathscr{N}:H^5(M)\rightarrow H^3(M)$ defined by $\mathscr{N}(u)
        =
        \frac{1}{2}\star(i\partial\bar\partial u)^2
        -
        e^F(e^{-u}-1+u)$
    satisfies
    \begin{align*}
        \|{\mathscr{N}(u_1)-\mathscr{N}(u_2)}_{H^3}
        &\leq
        \Cr{const:multiplication-theorem}
        \left(
        \|{u_1}_{H^5}
        +
        \|{u_2}_{H^5}
        \right)
        \\
        &\quad\cdot
        \left[
        \frac{1}{2}+
        \left(
        1+
        \Cr{const:multiplication-theorem}\epsilon
        \right)
        e^{\Cr{const:multiplication-theorem}
        \left(
        \|{u_1}_{H^3}
        +
        \|{u_2}_{H^3}
        \right)}
        \right]
        \|{u_1-u_2}_{H^5}
    \end{align*}
    for all $u_1,u_2\in H^5(M)$.
\end{proposition}

\begin{proof}
    \emph{Step 1: a tensor version of \cref{corollary:sobolev-multiplication-thm}.}
    We claim that for all $2$-forms $\eta,\mu\in H^3(M;\Lambda^2 T^*M)$,
    \begin{align}
    \label{equation:tensor-multiplication}
        \|{\star(\eta\wedge\mu)}_{H^3}
        \leq
        \Cr{const:multiplication-theorem}
        \|{\eta}_{H^3}
        \|{\mu}_{H^3},
    \end{align}
    with the same constant as in
    \cref{equation:multiplication-theorem-constant-def}.
    As $\star$ is a pointwise isometry of $\Lambda^2$ and
    $\star(\eta\wedge\mu)=\<\eta,\star\mu\>$, we have by Cauchy-Schwarz $|\star(\eta\wedge\mu)| = |\<\eta,\star\mu\>| \leq |\eta|\cdot|\mu|$
    pointwise. Since $\star$ and the metric are parallel, the bilinear map
    $(\eta,\mu)\mapsto\star(\eta\wedge\mu)$ is parallel as well, so the Leibniz
    rule gives, schematically,
    $\nabla^k\!\left(\star(\eta\wedge\mu)\right)
    =\sum_{i=0}^k\binom{k}{i}\star(\nabla^i\eta\wedge\nabla^{k-i}\mu)$,
    and each summand obeys the pointwise bound
    $|\star(\nabla^i\eta\wedge\nabla^{k-i}\mu)|
    \leq|\nabla^i\eta|\cdot|\nabla^{k-i}\mu|$.
    The proof of \cref{corollary:sobolev-multiplication-thm} now applies
    verbatim with $u,v$ replaced by $\eta,\mu$.

    \emph{Step 2: the Monge--Amp\`ere part.}
    Write $w_\pm=u_1\pm u_2$. Since $2$-forms commute,
    \[
        \frac{1}{2}\star(i\partial\bar\partial u_1)^2
        -
        \frac{1}{2}\star(i\partial\bar\partial u_2)^2
        =
        \frac{1}{2}\star\!\left(
        i\partial\bar\partial w_+
        \wedge
        i\partial\bar\partial w_-
        \right).
    \]
    By \cite[Exercise 7.50]{Ballmann2006}, we have
    $i\partial\bar\partial w
    =\tfrac12\left(\nabla^2 w(J\cdot,\cdot)-\nabla^2 w(\cdot,J\cdot)\right)$,
    and differentiating it $m$ times and using that $J$ is parallel yields
    $\nabla^m(i\partial\bar\partial w)
    =\tfrac12\left(\nabla^{m+2} w(J\cdot,\cdot)-\nabla^{m+2}w(\cdot,J\cdot)\right)$, hence
    $|\nabla^m(i\partial\bar\partial w)|\leq|\nabla^{m+2}w|$
    for every $m\geq0$ and therefore $\|{i\partial\bar\partial w}_{H^3} \leq \|{w}_{H^5}$.
    Combining this bound with \cref{equation:tensor-multiplication} gives
    \begin{align*}
        \|{\frac{1}{2}\star(i\partial\bar\partial u_1)^2
        -
        \frac{1}{2}\star(i\partial\bar\partial u_2)^2}_{H^3}
        &\leq
        \frac{1}{2}\Cr{const:multiplication-theorem}
        \|{i\partial\bar\partial w_+}_{H^3}
        \|{i\partial\bar\partial w_-}_{H^3}
        \leq
        \Cr{const:multiplication-theorem}
        \left(
        \|{u_1}_{H^5}
        +
        \|{u_2}_{H^5}
        \right)
        \|{u_1-u_2}_{H^5}.
    \end{align*}

    \emph{Step 3: the zeroth order part.}
    Set $\xi(u)=e^{-u}-1+u=\sum_{k=2}^\infty\frac{(-1)^k}{k!}u^k$.
    For each $k\geq2$ we factor
    $u_1^k-u_2^k=(u_1-u_2)\sum_{j=0}^{k-1}u_1^j u_2^{k-1-j}$;
    each summand is a product of $k$ functions, so $k-1$ applications of
    \cref{corollary:sobolev-multiplication-thm} give
    \[
        \|{u_1^k-u_2^k}_{H^3}
        \leq
        \Cr{const:multiplication-theorem}^{k-1}
        \|{u_1-u_2}_{H^3}
        \sum_{j=0}^{k-1}
        \|{u_1}_{H^3}^j\|{u_2}_{H^3}^{k-1-j}
        \leq
        \Cr{const:multiplication-theorem}^{k-1}
        \|{u_1-u_2}_{H^3}
        \left(
        \|{u_1}_{H^3}+\|{u_2}_{H^3}
        \right)^{k-1}.
    \]
    Summing the series (which converges absolutely in $H^3$) we obtain
    \begin{align*}
        \|{\xi(u_1)-\xi(u_2)}_{H^3}
        &\leq
        \Cr{const:multiplication-theorem}
        \|{u_1-u_2}_{H^3}
        \sum_{k=2}^{\infty}
        \frac{
        \Cr{const:multiplication-theorem}^{k-2}
        }{k!}
        \left(
        \|{u_1}_{H^3}
        +
        \|{u_2}_{H^3}
        \right)^{k-1}
        \\
        &\leq
        \Cr{const:multiplication-theorem}
        \|{u_1-u_2}_{H^3}
        \left(
        \|{u_1}_{H^3}
        +
        \|{u_2}_{H^3}
        \right)
        \cdot
        e^{\Cr{const:multiplication-theorem}
        \left(
        \|{u_1}_{H^3}
        +
        \|{u_2}_{H^3}
        \right)},
    \end{align*}
    where in the last step we used $1/k!\leq1/(k-2)!$.
    Since $e^F=1+E$, another application of
    \cref{corollary:sobolev-multiplication-thm}
    gives
    \begin{align*}
        \|{e^F(\xi(u_1)-\xi(u_2))}_{H^3}
        \leq
        \|{\xi(u_1)-\xi(u_2)}_{H^3}
        +
        \|{E(\xi(u_1)-\xi(u_2))}_{H^3}
        \leq
        \left(
        1+
        \Cr{const:multiplication-theorem}\epsilon
        \right)
        \|{\xi(u_1)-\xi(u_2)}_{H^3}.
    \end{align*}
    Combining Steps 2 and 3 and using
    $\|{u}_{H^3}\leq\|{u}_{H^5}$
    proves the claim.
\end{proof}

\section{The Data}\label{Data}
Numerical values for constants:

\begin{table}[htbp]
    \centering
    \begin{tabular}{c|c|c}
         $\uparrow\|{\mathscr{E}}_{H^3}$ & $1.634476\times10^{-27}$ & \cref{subsection:a-posteriori-error} \\
         $\Cr{const:metric-difference}$ & $10.925657$ & \cref{lemma:metric-comparison} \\
         inset $\delta$ & $1/5120000$ & \cref{subsection:non-sharp-eigenvalue-bound} \\
         $\downarrow\lambda^{D_6}_{A,1}(Q_\delta)$ & $6.2807793467039440$ & \cref{subsection:non-sharp-eigenvalue-bound} \\
         $\lambda_{\min}$ & $[6.2366561589878447,\,6.5975991681974387]$ & \cref{subsection:non-sharp-eigenvalue-bound,Firsteigenvalueestimate} \\
         $\downarrow\min|\operatorname{spec}(\mathscr{L})|$ & $1-2.110667\times10^{-36}$ & \cref{corollary:L-spectral-gap} \\
         $r=\uparrow\|{\varphi}_{H^5(M,\omega_P)}$ & $1.711300\times10^{-18}$ & \cref{perturb} \\
         $\lambda_{\min}^{KE}$ & $[6.2366561589878442,\,6.5975991681974393]$ & \cref{Firsteigenvalueestimate} \\
         $\eta$ & $[1.869706,\,1.869707]\times10^{-17}$ & \cref{HSC} \\
         $\rho$ & $[1.711299,\,1.711300]\times10^{-18}$ & \cref{HSC} \\
         $U$ & $\left[\frac{11}{12},\frac{23}{24}\right] \times \left[-\frac{23}{24},-\frac{11}{12}\right] \times\mathbb T^2$ & \cref{HSC} \\
         $\mathcal M$ & $[-4.318638188974,\,-1.081153782238]\times10^{-3}$ & \cref{HSC} \\
         $\xi$ & $(1,1)$ & \cref{HSC} \\
         $\varepsilon_{\Ric}:=\uparrow\|{\Ric(g)-g}_{C^0_g}$ & $5.393183\times10^{-3}$ & \cref{subsection:elliptic-estimates} \\
         $\downarrow\mu$ & $0.994606817$ & \cref{subsection:elliptic-estimates,subsection:non-sharp-eigenvalue-bound} \\
         $\uparrow\|{\Ric(g)}_{C^0_g}$ & $4.085698$ & \cref{subsection:elliptic-estimates} \\
         $\uparrow\|{\Riem(g)}_{C^0_g}$ & $10.325726$ & \cref{subsection:elliptic-estimates} \\
         $\uparrow\|{\nabla\Riem(g)}_{C^0_g}$ & $27.787753$ & \cref{subsection:elliptic-estimates} \\
         $\uparrow\|{\nabla^2\Riem(g)}_{C^0_g}$ & $93.177426$ & \cref{subsection:elliptic-estimates} \\
    \end{tabular}
    \caption{Chosen parameter values, rigorous enclosures, and rigorous upper ("$\uparrow$") and lower ("$\downarrow$") bounds. Here $\Ric(g)\geq\mu g$ and $\mu=1-\varepsilon_{\Ric}$. The displayed endpoints have been rounded outwards.}
    \label{tab:numerical-constants}
\end{table}

\begin{table}[htbp]
    \centering
    \begin{tabular}{c|ccc}
        $\alpha$
        & $\|{\partial^\alpha u^{11}}_{C^0(P)}$
        & $\|{\partial^\alpha u^{12}}_{C^0(P)}
             =\|{\partial^\alpha u^{21}}_{C^0(P)}$
        & $\|{\partial^\alpha u^{22}}_{C^0(P)}$ \\
        \hline
        $(0,0)$ & $0.896032$  & $0.442755$  & $0.896032$  \\
        $(1,0)$ & $2.597958$  & $2.044278$  & $4.024637$  \\
        $(0,1)$ & $4.024637$  & $2.044278$  & $2.597958$  \\
        $(2,0)$ & $6.407032$  & $11.913082$ & $23.731556$ \\
        $(1,1)$ & $11.913082$ & $3.366399$  & $11.913082$ \\
        $(0,2)$ & $23.731556$ & $11.913082$ & $6.407032$
    \end{tabular}
    \caption{Certified componentwise $C^0(P)$-bounds for the inverse Hessian and its coordinate derivatives of order at most two. Each displayed bound has been rounded upward. Here, for a multi-index $\alpha=(\alpha_1,\alpha_2)$ we wrote $\partial^\alpha:=\partial_1^{\alpha_1}\partial_2^{\alpha_2}$.}
    \label{tab:inverse-metric-bounds}
\end{table}

\clearpage

\bibliographystyle{apalike}
\bibliography{MASTER/LIB}

\end{document}